\documentclass[fullpage,12pt]{amsart}
\usepackage{hyperref,pb-diagram,amssymb,epic,eepic,verbatim,graphicx,graphics,epsfig,psfrag}
\hypersetup{colorlinks}
\usepackage{color}
\definecolor{darkred}{rgb}{0.5,0,0}
\definecolor{darkgreen}{rgb}{0,0.5,0}
\definecolor{darkblue}{rgb}{0,0,0.5}
\hypersetup{ colorlinks,
linkcolor=darkblue,
filecolor=darkgreen,
urlcolor=darkred,
citecolor=darkblue }

\newtheorem{theorem}{Theorem}[section]
\newtheorem{corollary}[theorem]{Corollary}

\newtheorem{proposition}[theorem]{Proposition}
\newtheorem{lemma}[theorem]{Lemma}
\newtheorem{lem}[theorem]{}
\theoremstyle{definition}
\newtheorem{definition}[theorem]{Definition}

\theoremstyle{remark}
\newtheorem{remark}[theorem]{Remark}
\newtheorem{example}[theorem]{Example}

\newcommand{\blem}{\begin{lem} \rm}
\newcommand{\elem}{\end{lem}}
\newcommand\M{\mathcal{M}}
\newcommand\D{\mathcal{D}}

\renewcommand\M{\mathcal{M}}
\newcommand\QQ{\mathcal{Q}}

\newcommand{\U}{\mathcal{U}}

\newcommand{\XX}{\mathcal{X}}

\newcommand{\R}{\mathbb{R}}

\newcommand{\C}{\mathbb{C}}
\newcommand{\cC}{\mathcal{C}}

\newcommand{\Z}{\mathbb{Z}}
\newcommand{\Q}{\mathbb{Q}}

\renewcommand{\P}{\mathbb{P}}

\newcommand\lie[1]{\mathfrak{#1}}

\newcommand{\g}{\lie{g}}

\newcommand{\on}{\operatorname}

\newcommand{\Jac}{\on{Jac}}

\newcommand{\Bl}{\on{Bl}}

\newcommand{\quot}{\on{quot}}

\newcommand{\Ch}{\on{Ch}}

\newcommand{\Td}{\on{Td}}

\newcommand{\dual}{\vee}

\newcommand{\Fr}{\on{Fr}}

\newcommand{\fin}{\on{fin}}

\newcommand{\Ad}{ \on{Ad} }

\newcommand{\Hom}{ \on{Hom}}

\newcommand{\Ind}{ \on{Ind}}

\newcommand{\Res}{\on{Resid}}
\newcommand{\Resid}{\on{Resid}}

\newcommand{\codim}{\on{codim}}

\newcommand\dirac{/\kern-1.2ex\partial} 
\newcommand\qu{/\kern-.7ex/} 
\newcommand\lqu{\backslash \kern-.7ex \backslash} 

\newcommand\dr{r_+ \kern-.7ex - \kern-.7ex r_-}
 
\newcommand{\lev}{{\on{lev}}} 

\usepackage{amsmath, amsfonts,amssymb, latexsym, mathrsfs,geometry,amscd}

\usepackage{tikz-cd,braket}

\newcommand{\out}[1]   {{}}
\newcommand{\labell}\label

\renewcommand{\d}{{\mbox{d}}}
\newcommand{\ol}{\overline}

\newcommand\eps{\epsilon}

\newcommand{\f}{\frac}
\newcommand{\lan}{\langle}
\newcommand{\ran}{\rangle}
\newcommand{\hh}{{\f{1}{2}}}

\newcommand{\ti}{\tilde}

\newcommand\pt{\on{pt}}

\newcommand\rk{\on{rk}}

\renewcommand{\ss}{\on{ss}}

\newcommand\mE{\mathcal{E}}
\newcommand\MM{\mathfrak{M}}

\newcommand\Gr{\on{Gr}}
\newcommand\Map{\on{Map}}

\newcommand\ev{\on{ev}}
\newcommand\Eul{\on{Eul}}
\newcommand\Pic{\on{Pic}}

\newcommand\ul{\underline}
\newcommand\mO{\mathcal{O}}

\newcommand\reg{{\on{reg}}}

\newcommand\bdefn{\begin{definition}}
\newcommand\edefn{\end{definition}}
\newcommand\bea{\begin{eqnarray*}}
\newcommand\eea{\end{eqnarray*}}
\newcommand\bcv{\left[ \begin{array}{r} }
\newcommand\ecv{\end{array} \right] }

\newcommand\bma{\left[ \begin{array} }
\newcommand\ema{\end{array} \right]}
\newcommand\ben{\begin{enumerate}}
\newcommand\een{\end{enumerate}}
\newcommand\beq{\begin{equation}}
\newcommand\eeq{\end{equation}}
\newcommand\bex{\begin{example}}
\newcommand\bsj{\left\{ \begin{array}{rrr} }
\newcommand\esj{\end{array} \right\}}

\newcommand\Cone{\on{Cone}}

\newcommand\LL{\mathcal{L}}

\newcommand\eex{\end{example}}

\newcommand\sx{*\kern-.5ex_X}

\newcommand{\fr}{{\on{fr}}}

\newcommand{\cS}{{\mathcal{S}}}

\def\mathunderaccent#1{\let\theaccent#1\mathpalette\putaccentunder}
\def\putaccentunder#1#2{\oalign{$#1#2$\crcr\hidewidth \vbox
to.2ex{\hbox{$#1\theaccent{}$}\vss}\hidewidth}}

\begin{document}
\title[
Wall-crossing for Gromov-Witten invariants]
{A wall-crossing formula for Gromov-Witten invariants under variation of
  git quotient}

\author{Eduardo Gonz\'alez} 


\address{
Department of Mathematics
University of Massachusetts Boston
100 William T. Morrissey Boulevard
Boston, MA 02125}
  \email{eduardo@math.umb.edu}

\author{Chris T. Woodward}

\thanks{Partially supported by NSF grants
  DMS-1104670 and DMS-0904358.}

\address{Mathematics-Hill Center,
Rutgers University, 110 Frelinghuysen Road, Piscataway, NJ 08854-8019,
U.S.A.}  \email{ctw@math.rutgers.edu}

\begin{abstract} 
  We prove a quantum version of a wall-crossing formula of Kalkman
  \cite{ka:co}, \cite{le:sy2} that compares intersection pairings on
  geometric invariant theory (git) quotients related by a change in
  polarization.  Each expression in the classical formula is quantized
  in the sense that it is replaced by an integral over moduli spaces
  of certain stable maps; in particular, the wall-crossing terms are
  gauged Gromov-Witten invariants with smaller structure group.  As an
  application, we show that the genus zero graph Gromov-Witten
  potentials of quotients related by wall-crossings of crepant type
  are equivalent up to a distribution in one of the quantum parameters
  that is almost everywhere zero.  This is a version of the crepant
  transformation conjecture of Li-Ruan \cite{liruan:surg},
  Bryan-Graber \cite{bryan:crep}, Coates-Ruan \cite{cr:crep} etc. in
  cases where the crepant transformation is obtained by variation of
  git.
\end{abstract}

\maketitle

\ \vskip -.5in  \ 

 \tableofcontents

\section{Introduction} 

\subsection{Kalkman's wall-crossing formula}

According to the geometric invariant theory introduced by Mumford
\cite{mu:ge}, the {\em git quotient} $X \qu G$ of an action of a
complex reductive group $G$ on a projective variety $X$ equipped with
a polarization (ample equivariant line bundle) \(\LL\to X\)
has coordinate ring equal to the $G$-invariant part of the coordinate
ring on $X$.  Geometrically $X \qu G$ is the quotient of an open {\em
  semistable locus} $X^{\ss} \subset X$ by an equivalence relation,
where a point $x \in X$ is semistable if there is a non-constant
invariant section of a tensor power of the polarization that is
non-zero at the point.  If the action of $G$ on the semistable locus
has only finite stabilizers, then $X \qu G$ is the quotient of
$X^{\ss}$ by the action of $G$, by which we mean here the {\em
  stack-theoretic} quotient, see for example \cite{dejong:stacks}.  If
$X$ is in addition smooth, then $X \qu G$ is a smooth proper
Deligne-Mumford stack with projective coarse moduli space.  In
Kempf-Ness \cite{ke:le}, see also Mumford et al \cite{mu:ge}, the
coarse moduli space of the git quotient is identified with the
symplectic quotient of $X$ by a maximal compact subgroup of $G$.

The question of how the git quotient depends on the polarization, or
equivalently, choice of moment map is studied in a series of papers by
Guillemin-Sternberg \cite{gu:bi}, Brion-Procesi \cite{br:ac},
Dolgachev-Hu \cite{do:va}, and Thaddeus \cite{th:fl}. Under suitable
stable=semistable and smoothness conditions, the git quotient
$X \qu G$ undergoes a sequence of blow-ups and blow-downs. The class
of birational equivalences which appear via variation of git is
reasonably large. In fact, for a class of so-called {\em Mori dream
  spaces}, any birational equivalence can be written as a composition
of birational equivalences induced by variation of git \cite{hk:dream,
  mc:dream}.

The question of how the cohomology of the quotient depends on the
polarization is studied in Kalkman \cite{ka:co}, which proves a
wall-crossing formula for the intersection pairings under variation of
symplectic quotient.  Similar results, in the context of Donaldson
theory, are given in Ellingsrud-G\"ottsche \cite{eg:var}.  Let $X$ be
a smooth projective $G$-variety as above such that $G$ acts locally
freely on the semistable locus (that is, with finite stabilizers) and
$H_G(X)$ its equivariant cohomology with rational coefficients.  There
is a natural map
\[ \kappa_X^G: H_G(X) \to H(X \qu G) \]
studied in Kirwan's thesis \cite{ki:coh}, given by restriction to the
semistable locus $ H_G(X) \to H_G(X^{\ss}) $ and then descent under
the quotient $ H_G(X^{\ss}) \cong H(X \qu G) $.  Consider the simplest
case $G = \C^\times$.  Let $X \qu_\pm G$ denote the git quotients
corresponding to polarizations $\LL_\pm \to X_\pm$.  Let
\[\kappa_{X,\pm}^G: H_G(X) \to H(X \qu_\pm G)\] 
be the Kirwan maps and let
\[ \tau_{X \qu_\pm G}: H(X \qu_\pm G) \to \Q, \quad \alpha \mapsto
\int_{[X \qu_\pm G]} \alpha \]
denote integration over $X \qu_\pm G$.  Kalkman's formula expresses
the difference between the integrals $\tau_{X \qu_\pm G} \circ
\kappa_{X,\pm}^G$ as a sum of fixed point contributions from
components $X^{G,t} \subset X^G$ that are semistable for elements in
the rational Picard group $\Pic_\Q^G(X)$ interpolating between $\LL_-$
and $\LL_+$,
\[[\LL_t:=\LL_-^{\otimes (1 - t)/2} \otimes \LL_+^{\otimes (1 + t)/2}]
\in \Pic_\Q^G(X) \]
for some rational $t \in (-1,1)$.  Each such fixed point has a
contribution to the localization formula for the integral of a class
$\alpha \in H_G(X)$ over $X$ given as follows.  Let $\nu_{X^{G,t}} \to
X^{G,t}$ denote the normal bundle of the inclusion $X^{G,t} \to X$,
and $ \Eul_G(\nu_{X^{G,t}}) \in H_G(X^{G,t}) $ its $G$-equivariant
Euler class, or equivalently, equivariant Chern class of degree equal
to the real rank.  We identify $H(BG)$ with the polynomial ring
$\Q[\xi]$ in a single element $\xi$ of degree $2$, representing the
hyperplane class in the cohomology of $BG = \C P^\infty$.  If $m_t =
\codim(X^{G,t})$ and
\begin{equation} \label{splits}
 \nu_{X^{G,t}} = \bigoplus_{i = 1}^{m_t}
   \nu_{X^{G,t},i} \end{equation}
is a decomposition into line bundles with weights $\mu_i \in \Z$ then
\[ \Eul_G(\nu_{X^{G,t}}) = \prod_{i=1}^{m_t} \Eul_G(\nu_{X^{G,t},i}) =
\prod_{i=1}^{m_t} ( c_1(\nu_{X^{G,t},i}) + \mu_i \xi ) .\]
Since $G$ acts with no non-trivial fixed vectors on $\nu_{X^{G,t}}$,
the Euler class has an inverse
\[\Eul_G(\nu_{X^{G,t}})^{-1} \in H(X^{G,t})[\xi,\xi^{-1}]\]
after inverting the equivariant parameter $\xi$.  If one has a
splitting as in \eqref{splits} then the inverted class admits
an expansion 
\begin{eqnarray*}
\Eul_G\left(\nu_{X^{G,t}}\right)^{-1}  &=& 
\prod_{i=1}^{m_t} \left(\mu_i \xi\right)^{-1} \left(
1 + \frac{c_1\left(\nu_{X^{G,t},i}\right)}{\mu_i \xi} \right)^{-1}  \\
&=& 
\prod_{i=1}^{m_t} \left(\mu_i \xi\right)^{-1} \left(1 - \frac{c_1\left(\nu_{X^{G,t},i}\right)}{\mu_i \xi}  
+ \left( \frac{c_1\left(\nu_{X^{G,t},i}\right)}{\mu_i \xi} \right)^2 - \ldots \right) .
\end{eqnarray*}
The contribution from a fixed point component $X^{G,t} \subset X^G$ is
the integral of the restriction $\iota_{X^{G,t}}^* \alpha$ times the
inverted Euler class, denoted
\[ \tau_{X^{G,t}}: H_G(X) \to \Q[\xi,\xi^{-1}], \quad \alpha \mapsto
  \int_{[X^{G,t}]} \iota_{X^{G,t}}^* \alpha \cup
  \Eul_G(\nu_{X^{G,t}})^{-1} .\]
Let $\Res $ denote the residue, that is, the coefficient of
$\xi^{-1}$:
\[ \Res_\xi: \Q[\xi,\xi^{-1}] \to \Q, \quad \sum_{n \in \Z} a_n \xi^n
\mapsto a_{-1} .\]
(More invariantly, the residue should be a map
$\Q[\xi,\xi^{-1}] \d \xi \to \Q$, but we omit the one-form from the
notation.)

\begin{theorem}[Kalkman wall-crossing formula, circle group case] \label{kalk1} 
Let $G = \C^\times$ and let $X$ be a smooth projective $G$-variety
equipped with polarizations $\LL_\pm \to X$.  Suppose that
stable=semistable for the $G$-action on $\P(\LL_+\oplus \LL_-)$.  Then
\begin{equation}  \label{kalkcirc}
\tau_{X \qu_+ G}\circ  \kappa_{X,+}^G - \tau_{X \qu_- G}
\circ \kappa_{X,-}^G =
\sum_{t \in (-1,1)} \Res_\xi \tau_{X^{G,t}} .\end{equation}
\end{theorem} 

In other words, failure of the following square to commute is measured
by an explicit sum of wall-crossing terms:
%
%
\begin{equation*}
\begin{tikzcd}[every arrow/.append style={-latex}]
  & H_G(X)\arrow{dr}{\kappa_{X,-}^G}  
  \arrow{dl}[swap]{\kappa_{X,+}^G} &  \\
  H(X \qu_+ G) \arrow{dr}[swap]{\tau_{X \qu_+ G}} & &
H(X \qu_- G) \arrow{dl}{\tau_{X \qu_- G}} \\ & \Q &
\end{tikzcd}
\end{equation*}

The formula \eqref{kalkcirc} also holds for certain quasi-projective
varieties, such as vector spaces whose weights are contained in an
open half-space, see the more general Theorem \ref{kalk2} below. 

\begin{example} {\rm (Integration over projective space)}
\label{projspace} 
The following simple example illustrates the notation involved. 
 let 
$G = \C^\times$ acting on $X = \C^k$ by scalar multiplication, 
\[ g [z_1,\ldots, z_n] = [gz_1,\ldots, gz_n] \]
so that $H_G(X) = \Q[\xi]$ where $\xi$ is the equivariant parameter
representing the hyperplane class under the isomorphism
$ H_G(X) \cong H(BG) \cong H(\C P^\infty) .$ Suppose that polarizations
$\LL_\pm$ correspond to the characters $\pm 1$. Invariant sections are
then monomials of positive resp.  negative degree, hence the
semistable locus is $X^{\ss,+}$ for $\LL_+$ and the emptyset for
$X^{\ss,-}$. Thus
\[X \qu_- G = \emptyset, \quad X 
\qu_+ G = \P^{k-1} \]
and the two chambers are separated by the value $t = 0$ so that
$0 \in X$ is semistable for $\LL_-^{(1-t)/2} \otimes \LL_+^{(1+t)/2}$.
The Kirwan map $\kappa_X^G: H_G(X) \to H(X \qu_+ G)$ sends the
generator $\xi \in H^2_G(X)$ to the hyperplane class
$\omega \in H^2(X \qu_+ G)$.  We compute the integrals
$ \int_{\P^{k-1}} \omega^a$ for $a \in \Z_{\ge 0}$ via wall-crossing.
In the negative chamber, the integral is zero, since the quotient is
empty.  By the Kalkman formula \eqref{kalkcirc}
\begin{eqnarray*} 
\int_{\P^{k-1}} \omega^a &=& 
 \Res_\xi \int_{[0]}
  \xi^a \cup \Eul_G(\C^k)^{-1} \\ 
&=& \Res_\xi  \xi^a / \xi^{k}  
=  \begin{cases}  1 & a = k- 1 \\
                      0  & \text{otherwise} \end{cases} 
\end{eqnarray*}   
confirming that $\omega^{k-1}$ is the dual of the fundamental class.
This ends the example.
\end{example} 

\subsection{Quantum Kirwan map and adiabatic limit theorem} 

The main result of this paper is a generalization of Theorem
\ref{kalk1} to the setting of genus-zero Gromov-Witten theory, that
is, quantum cohomology.  The Kirwan map, trace, and fixed point
contributions become {\em quantized} in the sense that each is
replaced by a formal map depending on a formal variable $q$ whose
specialization to $q = 0$ gives the classical version above.  First we
recall the definition of quantum cohomology of a smooth polarized
projective $G$-variety $X$. The {\em equivariant Novikov field}
\[ \Lambda^G_X \subset \Map(H^G_2(X,\Q),\Q)  \] 
\label{c1} associated to a $G $-variety $X$ with polarization
$\LL \to X$ consists of linear combinations of delta-functions $q^d$
at $d \in H_2^G(X,\Q)$ satisfying a finiteness condition:
\[ \Lambda^G_X := \left\{ \sum_{d \in H^G_2(X,\Q)} c_d q^d, \quad
\forall e > 0, \# \{ c_d | \lan d,c_1^G(\LL) \ran < e \} < \infty
\right\} .\]
The {\em equivariant quantum cohomology} is the tensor product
\[ QH_G(X) := H_G(X) \otimes
 \Lambda^G_X. \] 
We write $\Lambda^G_{X,\LL}$ resp. $QH_G(X,\LL)$ if we wish
to emphasize the dependence on $\LL$.
\noindent
A more standard definition in algebraic geometry would use the cone of
effective curve classes, but the definition we give here has better
invariance properties, for example, under Hamiltonian perturbation.
Below we will need several variations on this definition.  Let
$\Lambda_X^{G,\fin} \subset \Lambda_X^G$ denote the space of finite
linear combinations of the symbols $q^d$.  Denote by
$QH_{G,\fin}(X) := H_G(X) \otimes \Lambda_X^{G,\fin}\subset QH_G(X)$
the subspace with only finitely many non-vanishing exponents of $q$
non-vanishing.  Let $\Lambda^{G,\ge 0}_{X} \subset \Lambda_X^G$ be the
space consisting of expressions involving only powers $q^d$ with
positive pairing $ \lan d,c_1^G(\LL) \ran \ge 0$.  Denote by
$QH_{G,\ge 0}(X) \subset QH_G(X)$ the subspace
$H_G(X) \otimes \Lambda^{G,\ge 0}_X$.

The quantum cohomology $QH_G(X)$ has the structure of a Frobenius
manifold, in particular, it is equipped with a family of products
\[ \star_\alpha:  T_\alpha QH_G(X)^2 \to T_\alpha QH_G(X), \quad \alpha \in QH_G(X) \]
defined by equivariant virtual enumeration of genus zero {\em stable
  maps}, that is, maps $u: C \to X$ from projective nodal curves $C$
of arithmetic genus zero to $X$ with no infinitesimal automorphisms.
The moduli stack $\ol{\M}_{g,n}(X,d)$ of stable $n$-marked maps of
homology class $d \in H_2(X,\Z)$ and arithmetic genus $g$ is a proper
Deligne-Mumford stack equipped with a perfect relative obstruction
theory and evaluation map
\[ \ev = (\ev_1,\ldots, \ev_n): \ol{\M}_{g,n}(X,d) \to X^n .\]
The action of $G$ on $X$ induces an action of $G$ on
$\ol{\M}_{g,n}(X,d)$ by translation; then the evaluation maps induce a map 
\[ \ev^* : H_G(X)^n \to H_G( \ol{\M}_{g,n}(X,d)) .\]
The quantum product at $\alpha \in QH_G(X)$ is defined by restricting
to $g = 0$ and defining for $\beta,\gamma \in H_G(X) \subset T_\alpha
QH_G(X)$
\[ \beta \star_\alpha \gamma := \sum_{n \ge 0} \sum_{d \in H_2(X,\Z)} \frac{q^d}{n!}
\ev_{n+3,*} \ev_1^* \alpha \cup \ldots \cup \ev_n^* \alpha \cup \ev_{n+1}^*
\beta \cup \ev_{n+2}^* \gamma \]
\label{c4}
where the push-forwards are defined using the Behrend-Fantechi virtual
fundamental classes \cite{bf:in}, \cite{be:gw}. The definition is extended to
$T_\alpha QH_G(X) \cong QH_G(X)$ by linearity over $\Lambda_X^G$. Note
that we take the virtual fundamental classes to lie in the equivariant
homology $H_G(\ol{\M}_{g,n}(X,d))$ of the coarse moduli space, rather
than in the Chow ring as in \cite{gr:loc}.  The quantum cohomology
$QH_G(X)$ can also be defined using the smaller ring
$\Lambda_{X}^{G,\ge 0}$ but not $\Lambda_{X}^{G,\fin}$, because of the
infinite sums.

The orbifold quantum cohomology $QH(X \qu G)$ of the quotient
$X \qu G$ is defined by virtual enumeration of stable maps from
orbifold curves, as follows.  For any element $g \in G$ let $Z_g$ be
the centralizer of $g \in G$, $[g] \subset G$ its conjugacy class, and
$\lan g \ran \subset Z_g$ the subgroup generated by $g$.  Denote by
$\ol{I}_{X \qu G}$ the {\em rigidified inertia stack} from
Abramovich-Graber-Vistoli \cite{agv:gw}, given by
\[
 \ol{I}_{X \qu G} = \bigcup_{[g]} X^{g,\ss} / (Z_g/ \lan g \ran).
\]
Denote by
\[ QH(X \qu G) = H(\ol{I}_{X \qu G}) \otimes \Lambda_X^G \]
the quantum cohomology of the git quotient $X \qu G$ defined using the
same Novikov field $\Lambda_X^G$; this larger ring contains
$\Lambda_{X \qu G}$ by virtue of Kirwan's injection $H_2(X \qu G) \to
H_2^G(X)$.  Virtual enumeration of representable morphisms from
orbifold curves satisfying certain conditions to $X \qu G$ defines a
family of products
\[ \star_\alpha: T_\alpha QH(X \qu G)^2 \to T_\alpha QH(X \qu G), \quad
\alpha \in QH(X \qu G) .\]

A quantum version of the Kirwan map
\[ \kappa_X^G : QH_G(X) \to QH (X \qu G) \]
(we keep the same notation as in the classical case) was constructed
by the second author in \cite{qkirwan}. The map $\kappa_X^G$ is a formal, non-linear map
with the property that each linearization
\[ D_\alpha \kappa_X^G: T_\alpha QH_G(X) \to T_{\kappa_X^G(\alpha)}QH ( X
\qu G), \quad \alpha \in QH_G(X) \]
is a $\star$-homomorphism, defined by virtual enumeration of {\em
  affine gauged maps}.  Such a map is by definition a representable
morphism $u: \P(1,r) \to X/G$ from a weighted projective line
$\P(1,r), r > 0$ to the quotient stack $X/G$ mapping the stacky point
at infinity $\P(r) \subset \P(1,r)$ to the semistable locus $X \qu G
\subset X/G$.  (More precisely, the domain is a smooth Deligne-Mumford
stack of dimension one with a single stacky point with automorphism
group of order $r$.)  These are the algebro-geometric analogs of the
{\em vortex bubbles} considered in Gaio-Salamon \cite{ga:gw}.  The
compactified moduli stack $\ol{\M}_{n,1}^G(\C,X,d)$ of affine gauged
maps of homology class $d \in H_2^G(X,\Q)$ is, if stable=semistable
for $X$, a proper smooth Deligne-Mumford stack with a perfect relative
obstruction theory over the complexification of Stasheff's
multiplihedron \cite{qkirwan}. It has evaluation maps
\[ \ev \times \ev_\infty: \ol{\M}_{n,1}^G(\C,X) \to (X/G)^n \times
\ol{I}_{X \qu G},\]
at the markings and the point at infinity.  The quantum Kirwan map is
defined for $\alpha \in H_G(X) \subset QH_G(X)$ and a sequence of
classes $\beta_n \in H(\ol{\M}_{n,1}(\C))$ by
\[ \kappa_X^G(\alpha) := \sum_{n \ge 0} \sum_{ d \in H_2^G(X,\Q)}
\frac{q^d}{n!}  \ev_{\infty,*} \ev^* (\alpha,\ldots, \alpha ) \cup f^*
\beta_n.\]
We denote by $\kappa_{X}^{G,n}$ its $n$-th Taylor coefficient in
$\alpha$. The map $\kappa_X^G$ is defined over the smaller
equivariant Novikov ring $\Lambda_{X}^{G,\ge 0}$, but one obtains good
surjectivity properties only using the Novikov field $\Lambda_X^G$,
see \cite{gw:surject}.  The map $\kappa_X^G$ is a quantization of
Kirwan's in the sense that $D_0 \kappa_X^G |_{q = 0}$ is the map
studied in \cite{ki:coh}.  It admits a natural $\C^\times$-equivariant
generalization from $QH_G(X)$ to $QH_{\C^\times}(X \qu G)$, induced by
the natural action of $\C^\times$ on $\P(1,r)$.

A quantization of the classical integration map over $X \qu G$ is
given by the {\em graph potential} in Givental \cite{gi:eq}
\[ \tau_{X \qu G}: QH(X \qu G) \to \Lambda_X^G \]
defined by virtual enumeration of genus zero orbifold stable maps to
$C \times (X \qu G)$, for \(C=\P\),
of homology class $(1,d)$ for some $d \in H_2(X \qu G,\Q)$.  Let
\[\ol{\M}_n(C,X \qu G,d) := \ol{\M}_{0,n}(C \times (X
\qu G),(1,d))\] 
denote the stack of such maps of class $d \in H_2(X \qu G,\Q)$, which
we view as an element of $H_2^G(X,\Q)$ via the inclusion of the
semistable locus.  It has evaluation and forgetful maps  
\begin{equation} \label{evf}
\ev: \ol{\M}_n(C,X \qu G,d) \to (X \qu G)^n ,\quad f: \ol{\M}_n(C,X
\qu G ,d) \to 
\ol{\M}_n(C)
 \end{equation} 
 where $\ol{\M}_n(C)$ is the moduli space of stable maps to $C$ of
 class $[C]$.  The graph potential is defined for
 $\alpha \in H(\ol{I}_{X \qu G}) \subset QH(X \qu G)$ and a sequence
 of classes $\beta_n \in H(\ol{\M}_n(C))$ by
 \begin{equation} \label{gpot} \tau_{X \qu G}(\alpha) := \sum_{n \ge
     0} \sum_{d \in H_2(X \qu G,\Q)} \frac{q^d}{n!}
   \int_{[\ol{\M}_n(C, X \qu G,d)]} \ev^* (\alpha, \ldots, \alpha)
   \cup f^* \beta_n.\end{equation}
 Again $\tau_{X \qu G}$ is defined over the equivariant Novikov ring
 $\Lambda_X^{G,\ge 0}$.  The potential $\tau_{X \qu G}$ admits a
 natural $\C^\times$-equivariant extension
 $QH_{\C^\times}(X \qu G) \to \Lambda_X^G$ induced by identifying
 $QH_{\C^\times}(X \qu G) \cong QH(X \qu G)[\xi]$ where $\xi$ is the
 equivariant parameter. \label{c5}

The graph potential $\tau_{X \qu G}$ is related via the quantum Kirwan
map to a {\em gauged Gromov-Witten potential}
$ \tau_X^G: QH_G(X) \to \Lambda_X^G $ defined by virtual enumeration
of {\em gauged maps}, by which we mean morphisms from $C$ to the
quotient stack $X/G$, satisfying a Mundet stability condition
\cite{mund:corr} generalizing semistability for vector bundles on
curves:

\begin{definition} {\rm (Mundet semistability)}
  \label{mundetsemistable} A {\em gauged map} from a smooth projective
  curve $C$ to the quotient stack $X/G$ is a morphism $v: C \to X/G$,
  consisting of a pair
\[( p: P \to C, u : C \to P(X) := P \times_G X) \] 
see \cite{dejong:stacks}. We suppose that the Lie algebra $\g$ of
\(G\)
is the complexification of a unitary form $\g_\R$.  Let $\g_\R$ be
equipped with an inner product invariant under the action of
$G_\R = \exp(\g_\R)$, inducing an identification
$\g_\R \to \g_\R^\dual$.
\begin{enumerate}
\item {\rm (Projections on Levi subgroups)} Let $R \subset G$ be a
  parabolic subgroup.  A {\em Levi subgroup} is a maximal reductive
  subgroup $L \subseteq R $.  The quotient $U = R/L$ admits an
  embedding in $G$ as a unipotent subgroup.  Denote the corresponding
  Lie algebras $\lie{l} ,\lie{u} \subseteq \lie{r}$.  The group $R$
  admits a decomposition $R = LU$ and the projection
\begin{equation} \label{varphiL} \varphi_L: R \to L \end{equation} 
is a group homomorphism which may also be defined as follows: Let
$\lambda \in \lie{r}$ be an element acting positively on the roots of
$\lie{r}/\lie{l}$.  Then for $z \in \C^\times$ the automorphism
\[ \lie{r} \to \lie{r}, \quad r \mapsto \Ad(z^\lambda)  r  \] 
acts on the $\alpha$-weight space $\lie{u}_\alpha$ by the scalar
$z^{ \lan \lambda,\alpha \ran}$.  The corresponding Lie group
automorphism
\[ R \to R, \quad r \mapsto  z^\lambda r z^{-\lambda} \] 
has a limit as $z \to 0$ which is the projection $R \to L$.  For
example, if $R$ is the group of upper triangular matrices in
$G = GL(r)$ then $U$ is the unipotent group of upper triangular
matrices with $1$'s along the diagonal and $Ad(z^\lambda)$ conjugation
by $z^\lambda$ acts on the $ij$-th entry in the matrix by the scalar
$z^{\lambda_i - \lambda_j}$; the latter tends to zero for $i < j$ if
$\lambda_i < \lambda_j$. \label{c6}
\item \label{bundle} {\rm (Associated graded bundle)} 
Given a reduction of $P$ to a
  parabolic subgroup $R \subset G$ given by a section
  $\sigma: C \to P/R$ and a element $\lambda$ in the Lie algebra
  $\lie{r}$ of $R$ acting positively on $\lie{u}$ and commuting with
  $\lie{l}$, there is an {\em associated graded} morphism given by a
  bundle
\[\Gr(P) \to C\] 
whose structure group is the Levi $L$ and whose transition maps are
obtained by composing the transition maps for $P$ with the projection
$\varphi_L$ of \eqref{varphiL}.    Thus in particular $\Gr(P)$ is the
central fiber in a family of bundles 
\begin{equation} \label{tiP} 
\ti{P} \to C \times \C \end{equation} 
whose fiber over $z$ is the bundle $P_z$ obtained by conjugating the
transition maps of $P$ with $z^{\lambda}$.  The bundle $\Gr(P)$ has a
natural automorphism $\C^\times \to \Gr(P)$ generated by $z^\lambda$,
since the structure group of $\Gr(P)$ reduces to $L$ and $\lambda$
commutes with $L$.
\item {\rm (Associated graded morphism)}
Consider the associated bundle associated to the family of bundles
\eqref{tiP} 
\[ \ti{P}(X) := ( \ti{P} \times X)/G \]
A section $\ti{u}$ of $\ti{P}(X)$ is given by a collection of maps
$\ti{u}_i: U_i \to X$ in local trivializations of $\ti{P}(X)$,
satisfying $\ti{u}_j = \tau_{ji} \ti{u}_i$ where $\tau_{ji}$ are the
transition maps of the bundle. Therefore a section of $\ti{P}(X)$ over
$C \times \C^\times$ is given by
$(\zeta,z) \mapsto z^\lambda u_i(\zeta)$, where $u_i: U_i \to X$ are
the local maps defining $u$.  By Gromov compactness (the bundle
$\ti{P}(X)$ is quasiprojective) the section $\ti{u}$ extends uniquely
over the central fiber $\Gr(P) \to C$ as a stable map
\[\Gr(u): \hat{C} \to (\Gr(P))(X)\]
with domain $\hat{C}$.  Here $\hat{C}$ is a projective nodal curve
with some components possibly mapping into the fibers of $(\Gr(P))(X)$
and a distinguished {\em principal component} $C_0 \subset \hat{C}$
mapping isomorphically onto $C$ via composition with the projection
$(\Gr(P))(X) \to C$. 
\item {\rm (Hilbert-Mumford weight)} Since $\Gr(u)$ is a limit of the
  section $u$ under the automorphism $z^\lambda$ of $\ti{P}(X)$, the
  section $\Gr(u)$ is automatically fixed (up to automorphism) by
  $z^\lambda$ for $z \in \C^\times$.  It follows there exists  family of
  automorphisms 
\[ \phi(z): \hat{C} \to \hat{C}, z \in \C^\times \] 
such that
\[ \Gr(u) ( \phi(z)) = z^\lambda \Gr(u), \quad \forall z \in \C^\times
.\]
Since $C_0$ maps isomorphically onto $C$, the map $\phi(z)$ must be
trivial on $C_0$ and so $\Gr(u) | C_0$ the associated graded section
$\Gr(u)$ takes values in the fixed point set \(P(X^\lambda)\)
where $X^\lambda$ denotes the fixed point set of the automorphism
$z^\lambda: X \to X$.  The {\em Hilbert-Mumford weight}
  \begin{equation} \label{hweight} \mu_H(\sigma,\lambda) \in
    \Z \end{equation}
  (with notation defined in \eqref{bundle}) determined by the
  polarization $\LL$ is the usual Hilbert-Mumford weight, that is, the
  weight of the $\C^\times$-action
  \[ P(\LL) \to P(\LL), \quad l \mapsto z^\lambda l \]
over any point $u(z) \in P(X^\lambda)$ in the image of $C_0$ under
$\Gr(u)$. \label{c7}
\item {\rm (Ramanathan weight)} Assume that
  $\lambda \in \lie{r} \cong \lie{r}^\dual$ is a weight of $R$.  The
  Ramanathan weight of $\mu_{R}(\sigma,\lambda)$ of $(P,u)$ with
  respect to $(\sigma,\lambda)$ is given by the first Chern number of
  the line bundle determined by $\lambda$ via the associated bundle
  construction: If $\sigma^*P$ denotes the principal $R$-bundle
  associated to $\sigma$ as in \cite{ra:th} and $\C_\lambda$ is the
  one-dimensional representation of $R$ with weight $\lambda$ then the
  line bundle is 
\[ \sigma^*P \times_R \C_\lambda \to C .\] 
The {\em Ramanathan weight} is
 \begin{equation} \label{rweight} \mu_R(\sigma,\lambda) = \int_{[C]} c_1( P \times_R \C_\lambda)
  \in \Z.\end{equation} 
\item {\rm (Mundet weight)} The {\em Mundet weight} is the sum of the Hilbert-Mumford
and Ramanathan weights:
\[ \mu_M(\sigma,\lambda) : = \mu_H(\sigma,\lambda) +
\mu_R(\sigma,\lambda). \]
\item {\rm (Mundet semistability)} 
The morphism $(P,u)$ is {\em Mundet semistable} if 
\begin{equation} \label{mss} \mu_M(\sigma,\lambda) \leq 0 \end{equation}
for all such pairs $(\sigma,\lambda)$ \cite{mund:corr}. The morphism
is {\em Mundet stable} if the above inequalities \eqref{mss} are satisfied
strictly.
\end{enumerate} 
\end{definition} 

\begin{remark} We explain how this notation compares with that in
  Schmitt \cite{schmitt:git}.  Any one-parameter subgroup
  $\lambda : \C^\times \to G$ determines a parabolic subgroup 
\[  Q_\lambda := \Set{ g \in G | \exists l \in G, \lim_{z \to 0}
\lambda(z) g \lambda(z)^{-1}   = l  .} \] 
Let $\lambda$ be such a subgroup and 
\[ \beta : C \to P/Q_\lambda \] 
a parabolic reduction $\beta$ of $P$ to the parabolic subgroup
$Q_\lambda$ determined by $\lambda$, and Schmitt \cite{schmitt:git}
denotes by $(E_\bullet(\beta), \alpha_\bullet(\beta) )$ the weighted
filtration of $E = P(V^\dual)$ determined by $\beta$ of length say
$s$, with $\alpha_\bullet$ the coefficients of the coweight generating
$\lambda$.  Schmitt \cite{schmitt:git} defines
  \[ M(E_\bullet(\beta), \alpha_\bullet(\beta)) = \sum_{i=1}^s
  \alpha_i ( \deg(E) \rk(E_i) - \deg(E_i) \rk (E)) \in \Q.\]
which agrees with minus Ramanathan weight of \eqref{rweight} by a
standard computation involving Chern classes, and 
\[ \mu(E_\bullet,
\alpha_\bullet, \varphi) \in \Q \] 
the Hilbert-Mumford weight \eqref{hweight}, see
\cite[p. 139]{schmitt:git} (opposite to our convention).  The Mundet
weight for stability parameter $\delta \in \Q$ is then minus
\[  M(E_\bullet, \alpha_\bullet) + \delta \mu(E_\bullet,
\alpha_\bullet, \varphi)  \in \Q\] 
and a bundle with map is semistable iff this quantity is non-negative
for all pairs $(\lambda,\beta)$.  This ends the Remark.  \end{remark}

To obtain a proper moduli stack with a perfect obstruction theory, we
allow bubbling in the fibers.

\begin{definition} 
  A {\em nodal $n$-marked gauged map} over a scheme $S$ of homology
  class $d$ consists of an $n$-marked prestable curve
  $(\hat{C}, \ul{z})$ over $S$ and a morphism
\[ u: \hat{C} \to C \times X/G, \quad u_*[\hat{C}] = (1,d) \in
H_2(C,\Z) \times H_2^G(X,\Z) \]  
By the condition on the homology class, over each point $s\in S$ there
exists a principal component $C_0 \subset \hat{C}_s$ which maps under
$u$ and projection to the first factor isomorphically to $C$, and a
collection of {\em bubble components} 
\[ C_1,\ldots,C_k \subset\hat{C}, \quad \dim(u(C_i)) = 0, \ i =
1,\ldots k \] 
that map to points $p_i= \pi_1(u(C_i))$ in $C$, $\pi_1$ being the
projection $C \times X/G \to C$.  A marked gauged map
$(\hat{C},u,\ul{z})$ over a point is {\em Mundet semistable} if the
following two conditions hold: \label{c8}
\begin{itemize} \item the restriction $u |_{C_0} : C_0 \to X/G$ is
  Mundet semistable, and 
\item each bubble component $C_i, i =1,\ldots, k$
  on which $u$ is given by a trivial $G$-bundle with constant section
  has at least three special (marked or nodal) points.
\end{itemize} \end{definition}

We introduce the following notation.  Let $\ol{\MM}_n^G(C,X,d)$ denote
the stack of $n$-marked gauged maps and $\ol{\M}_n^G(C,X,\LL,d)$ (or
$\ol{\M}_n^G(C,X,d)$ for short if the polarization is understood) the
substack consisting of Mundet semistable gauged maps for the
polarization $\LL$.  Taking $X$ and $G$ to be points, and $d$ to be
trivial, one obtains the moduli stacks
\[ \ol{\MM}_n(C) := \ol{\MM}_{\on{genus}(C),n}(C,[C]), \quad
\ol{\M}_n(C) := \ol{\M}_{\on{genus}(C),n}(C,[C])\]
of prestable resp. stable maps to $C$ of class $[C] \in H_2(C,\Z)$.
The category of Mundet-semistable gauged maps from $C$ to $X/G$ of
homology class $d$ and $n$ markings forms an Artin stack
$\ol{\M}_n^G(C,X,\LL,d)$, which if all automorphism groups are finite
is a proper Deligne-Mumford stack with evaluation map and forgetful
morphisms \cite{cross}, \cite[Theorem 1.1]{reduc}

\begin{example}   \label{toric} {\rm (Toric Case)}  Suppose that $G$ is a torus
acting on a vector space $X$ with weight decomposition
$\bigoplus_{i=1}^k X_i$ so that $G$ acts on the one-dimensional
representation $X_i$ with weights $\mu_i \in \g^\dual$.  Suppose that
$X$ is equipped with a polarization given by a trivial line bundle
with character $\chi \in \g^\dual$.  Let $C$ be a curve of genus $0$.
Then
\begin{eqnarray*}
\ol{\M}_0^G(C,X,d) &=& {\M}_0^G(C,X,d) = H(\mO_C(d) \times_G X) \qu G
\\ &\cong& \bigoplus_{i=1}^k X_i^{\oplus (\lan \mu_i, d \ran +
1) } \qu G \end{eqnarray*}
where the polarization on $H^0(\mO_C(d) \times_G X)$ is the trivial
line bundle with character $\chi$, see \cite{qkirwan}.  For example,
if $G = \C^\times$ acts on $X = \C^k$ by scalar multiplication then
$\ol{\M}_0^G(C,X,d) = \P^{(d+1)k - 1}$ for $ d \ge 0$ and positive
character $\chi$.
\end{example} 

\begin{remark}\label{obstheoryrem} 
  In \cite[Example 6.4(e)]{qkirwan} a relative obstruction theory for
  $\ol{\M}_n^G(C,X,d)$ is constructed by the method in Behrend
  \cite{be:gw}, with a small extension to the case of quotient stacks
  in Olsson \cite[Theorem 1.5]{ol:def}.  \label{c9} The complex in the
  relative obstruction theory is $(Rp_* e^* T_{X/G})^\dual$, and comes
  equipped with a morphism to the relative cotangent complex $L_\pi$.
  A morphism to the shifted cotangent complex $L_{\ol{\MM}_n(C)}[1]$
is  obtained by composing with $L_\pi\to L_{\ol{\MM}_n(C)}[1]$.
  Completing the diagram in the derived category gives rise to an
  absolute deformation theory as in \cite[Appendix A]{gr:loc}.  If all
  automorphism groups are finite then this obstruction theory is
  perfect.
\end{remark} 

The moduli stack admits evaluation and forgetful maps
\[\ev: \ol{\M}_n^G(C,X,d) \to (X/G)^n ,\quad f: \ol{\M}_n^G(C,X,d) \to 
\ol{\M}_n(C).\]

\begin{definition}  {\rm (Gauged Gromov-Witten potential)} 
  Suppose that all automorphism groups are finite for
  Mundet-semistable gauged maps.  The gauged potential $\tau_X^G$ for
  a smooth projective curve $C$ is the formal map defined
  by \label{c10} for $\alpha \in H_G(X)$ and a sequence of classes
  $\beta_n \in H(\ol{\M}_n(C))$
  \begin{multline}
    \tau_X^G   ( \alpha,\beta)  = 
\sum_{n \ge 0, d \in H_2^G(X,\Z)} \frac{q^d}{n!}
    \int_{ [\ol{\M}_n^G(C,X,d)]} \ev^* (\alpha,\ldots, \alpha ) \cup 
    f^* \beta_n
\end{multline}
extended to $QH_G(X)$ by linearity. 
\end{definition} 

The gauged potential and the graph potential of the quotient are
related by the {\em adiabatic limit theorem} of \cite{qkirwan} (which
is a generalization of an earlier result of Gaio-Salamon
\cite{ga:gw}): Let $\rho$ be a rational {\em stability parameter}; we
consider Mundet stability with respect to the polarization $\LL^\rho$
with $\rho \to \infty$.  

\begin{definition} Let $\ol{\M}_{n,1}(C)$ be the moduli stack of
  scaled $n$-marked maps from \cite{qkirwan}; a generic element is an
  $n$-marked map $\pi: \hat{C} \to C$ with a relative differential
  $\lambda \in H^0(\hat{C},\omega_\pi)$.  It contains a prime divisor
  $ \ol{\M}_n(C)$ corresponding to maps with zero differential
  $\lambda = 0$, and for any partition
  $I_1 \cup \ldots \cup I_r = \{ 1 ,\ldots, n \}$ a prime divisor
  isomorphic to $\ol{\M_r} \times \prod_{j=1}^r \ol{\M}_{|I_j|}(\C)) $
  whose generic element is a curve $\hat{C}$ with infinite
  differential $\lambda = \infty$ on the one unmarked component
  $C_0 \cong C$ and finite differentials on the remaining components
  $C_1,\ldots, C_r$.
\end{definition}

\begin{theorem} [Adiabatic limit theorem] \label{largerel} If
  $X \qu G$ is a locally free quotient then all automorphism groups
  are finite for $\rho$ sufficiently large (more precisely, for any
  class $d \in H_2^G(X,\Z)$ there exists an $r > 0$ such that
  $\rho > r$ implies that all automorphism groups are finite) and
  \[ {\tau}_{X \qu G} \circ {\kappa_X^G}  = \lim_{\rho \to \infty}
  {\tau}_X^G  \]
in the following sense of Taylor coefficients:  
For any class 
$\beta \in \ol{\M}_{n,1}(C)$ let 
\[ \sum_{k=1}^l \prod_{I_1 \cup \ldots \cup I_r = \{ 1 ,\ldots, n \} }
\beta^k_{\infty} \otimes \beta^k_1 \otimes \ldots \beta^k_r , \quad
\beta_0 \]
be its restrictions to 
\[ H \left( \ol{ \M}_r(C) \times \prod_{j=1}^r \ol{\M}_{|I_j|}(\C) \right) , \quad \text{resp.}
\ H( \ol{\M}_n(C) )\]
respectively.   Then 
\[ \sum_{k=1}^l \tau_{ X \qu G}^r (\alpha,\beta_\infty^k) \circ
\kappa_X^{G,|I_j|}( \alpha, \beta_j^k) = \lim_{\rho \to \infty} \tau_X^{G,n}( \alpha, \beta_0) . \]
\end{theorem}

\noindent In other words, the diagram
\begin{equation} 
  \label{classdiag}
\begin{tikzcd}[every arrow/.append style={-latex}]
  QH_G(X) \arrow{dr}[swap]{\tau_X^G}
  \arrow{rr}{{\kappa}_X^G} & &QH_{\C^\times}( X \qu G)
  \arrow{dl}{\tau_{X \qu G}} \\ 
  & \Lambda_X^G &
   \end{tikzcd}
\end{equation}
commutes in the limit $\rho \to \infty$.  

\begin{remark} We often take in examples as insertion the class
  $\beta \in H^6(\ol{\M}_{3,1}(C))$ pulled back from the point class
  $\beta_0 \in H^6(\ol{\M}_3(C))$ under the forgetful map
  $\phi: \ol{\M}_{3,1}(C) \to \ol{\M}(C)$, in which case the class
  $\beta_{3,\infty}$ is also the point class which restricts to the
  point class $\beta_0 \in H^6(\ol{\M}_3(C))$ in which case the
  classes $\beta_\infty$ is also the point class and the classes
  $\beta_j$ trivial, since $\phi^{-1}(\pt)$ contains a single
  curve with infinite scaling, consisting of one infinitely-scaled
  components and three finitely-scaled components each with a single
  marking.  For $C = \P$ the projective line the result is a comparison between the
  three-point Gromov-Witten invariants in $X \qu G$ and the
  three-point gauged Gromov-Witten invariants in $X$.
See Example \ref{threepoint} below. 
\end{remark}

\subsection{Quantum Kalkman formula}

In order to study the dependence of the Gromov-Witten graph potential
of the quotient on the choice of polarization, suppose that $\LL_\pm \to
X$ are two polarizations of $X$ and 
\[
\LL_t := \LL_-^{(1-t)/2}\otimes
\LL_+^{(1+t)/2}
\] 
is the family of rational polarizations given by interpolation.  Let
$X \qu_\pm G$ denote the git quotients, 
\[ \kappa_{X,\pm}^G :
QH_G(X,\LL_\pm) \to QH_{\C^\times}(X \qu _\pm G) \] 
the {\em quantum} Kirwan maps (note that we do not introduce new
notation for the quantum version; the classical Kirwan map is obtained
by setting $q = 0$) for the two polarizations and 
\[ \tau_{X \qu_\pm
  G}: QH_{\C^\times}(X \qu_\pm G) \to \Lambda_{X,\LL_\pm}^G \] 
the graph potentials.  Denote by 
\[QH_G^{\on{fin}}(X) \subset
QH_G(X,\LL_-) \cap QH_G(X,\LL_+)\]
the subset of the quantum cohomology of finite sums in the Novikov
variable. The main result of this note is a formula for the difference
\[\tau_{X \qu_+ G} \circ \kappa_{X,+}^G - \tau_{X \qu_- G}
\circ \kappa_{X,-}^G: QH_G^{\on{fin}}(X) \to \Lambda_X^G \]
as a sum of fixed point contributions given by gauged Gromov-Witten
invariants with smaller structure group. Namely, for any non-zero
$\zeta \in \g$ generating a one-parameter subgroup $\C^\times_\zeta$
denote by $X^\zeta$ the fixed point set of $\C^\times_\zeta$ generated
by $\zeta$ and $G_\zeta$ the centralizer of $\zeta$.  The
adjoint action of $\zeta$ on $\g$ is semisimple and the Lie algebra
$\g_\zeta$ of $G_\zeta$ is then
\[  \g_\zeta = \{ x \in \g \ | \ [x,\zeta] = 0 \} .\]
It follows that $G_\zeta$ is reductive.

\label{c11}   Denote by $X^{\zeta,t} = X^\zeta$ the locus that
is semistable with respect to $\LL_t$. For any $t \in (-1,1)$ such
that $X^{\zeta,t}$ is non-empty, we introduce in Definition \ref{fps}
a stack $\ol{\M}^G_n(C,X,\LL_t,\zeta)$ of {\em reducible $n$-marked
  gauged maps} from $C$ to $X/G_\zeta$ consisting of a principal
component $C_0 \cong C$ mapping to $X^{\zeta,t}/G_\zeta$ and bubbles
$C_1,\ldots, C_k \subset \hat{C}$ mapping to $X/G_\zeta$. This stack
admits a perfect equivariant obstruction theory whose relative part is
the cone on the map $(Rp_*e^* T_{X/G})^\dual \to \C \zeta^\dual$ given
by the infinitesimal action, see Remark \ref{obstheoryrem}
above. 
(The complex $(Rp_* e^* T_{X/G})^\dual$ is not
perfect because of the $\C^\times$-automorphisms; taking the cone has
the effect of cancelling this additional automorphism.) Denote by
$\nu_t$ the virtual normal complex, given as the moving part of the
obstruction theory on $\ol{\M}^G_n(C,X,\LL_t)$ pulled back to
$\ol{\M}^G_n(C,X,\LL_t,\zeta)$ and by
\[ \Eul(\nu_t) \in H(\ol{\M}^G_n(C,X,\LL_t,\zeta))[\xi,\xi^{-1}] \] 
its (invertible) Euler class, where $\xi$ is the equivariant
parameter. The Mundet stability condition in general depends on a
choice of equivariant K\"ahler class; here we are interested in the
{\em large area limit} in which the gauged Gromov-Witten invariants
are related to the Gromov-Witten invariants of the git quotient.
Denote by
\[ \ti{\Lambda}_X^G := \Map(H_2^G(X,\Z),\Q) , \quad 
  \ti{\Lambda}_X^G[\xi,\xi^{-1}] := \Map(H_2^G(X,\Q),
\Q[\xi,\xi^{-1}])
\] 
\label{c2} \label{c122} the space of $\Q$
resp. $\Q[\xi,\xi^{-1}]$-valued functions on
$H_2^G(X,\Z)$. \label{c12} Note that $\ti{\Lambda}_X^G$ has no ring
structure extending that on $\Lambda_X^G$.  The space
$\ti{\Lambda}_X^G$ can be viewed as the space of distributions in the
quantum parameter $q$, and we use it as a master space interpolating
Novikov parameters for the quotients as \(t\) varies.

\begin{definition} [Fixed point potential]
  Let $X,G,\LL_\pm$ be as above, and $\zeta \in \g,t \in (-1,1)$ such
  that $X^{\zeta,t}$ is non-empty. The {\em fixed point potential}
  associated to this data is the map
\begin{multline}
  \tau_{X,G,t,\zeta}: QH_{G,\fin}(X) \to
  \ti{\Lambda}_X^G[\xi,\xi^{-1}]
 \\ \quad \alpha \mapsto \sum_{d \in H_2^G(X,\Z)} \sum_{n \ge 0}
 \int_{[\ol{\M}_n^{G}(C,X,\LL_t,\zeta,d)]} \frac{q^d}{n!}  \ev^* (\alpha, \ldots,
 \alpha) \cup \Eul(\nu_t)^{-1} \cup f^* \beta_n
\end{multline}
for $\alpha \in H_G(X)$, extended to $QH_{G,\fin}(X)$ by linearity.
\end{definition} 

We may now state the main result of the paper, in the case of torus
actions.

\begin{theorem} [Quantum Kalkman formula, abelian case]
  \label{qkalkcirc} Suppose that $G$ is a torus and $X$ is a smooth
  projective $G$-variety equipped with polarizations $\LL_\pm \to X$,
  and all automorphism groups are finite for the polarization
  $\LL_\pm$. Then
  \begin{equation} \label{projfixedeq} \tau_{X \qu_+ G}\circ \
    \kappa_{X,+}^G - \tau_{X \qu_- G}\circ \kappa_{X,-}^G = \sum_{t
      \in (-1,1)} \sum_{\zeta \in \g^\times/G} \Res_\xi
    \tau_{X,G,t,\zeta} .\end{equation}
\end{theorem} 

In other words, failure of the following square to commute is measured
by an explicit sum of wall-crossing terms given by certain gauged
Gromov-Witten invariants:

\[\begin{tikzcd} 
  QH_G(X,\LL_-)  \arrow{d}[swap]{\kappa_{X,-}^G}
& {QH^{\fin}_G(X)}  
\arrow{l} \arrow{r}& 
{QH_G(X,\LL_+)}
 \arrow{d}{\kappa_{X,+}^G} \\ 
 {QH_{\C^\times}(X \qu _- G)} \arrow{d}[swap]{\tau_{X \qu_- G}} & &
  {QH_{\C^\times}(X \qu _+ G)} \arrow{d}{\tau_{X \qu_+ G}} \\
  {\Lambda_{X,\LL_-}^G} \arrow{r} &{\ti{\Lambda}_X^G} &{\Lambda_{X,\LL_+}^G.}
\arrow{l}
\end{tikzcd}\]

The diagram is somewhat more complicated than in the classical case in
\eqref{classdiag}, because the maps
$\tau_{X \qu_\pm G} \ \kappa_{X,\pm}^G$ are defined using different
Novikov rings \label{c13} $\Lambda_{X,\LL_\pm}^{G,\ge 0}$.  If a symbol $q^d$
appears in one Novikov ring $\Lambda_{X,\LL\pm}^{G,\ge 0}$ but not the
other then the corresponding contribution to
$\tau_{X \qu_\pm G} \ \kappa_{X,\pm}^G$ must be equal to the
$q^d$-term in the wall-crossing contribution on the right.  See
Example \ref{threepoint}.

The wall-crossing formula for Gromov-Witten invariants Theorem
\ref{qkalkcirc} should be considered mirror to various results in on
the behavior of the derived category of bounded complexes of coherent
sheaves under variation of git quotient appearing recently in Segal
\cite{seg:equiv}, Halpern-Leistner \cite{hal:der} and
Ballard-Favero-Katzarkov \cite{ball:var}, and more generally for
crepant birational transformations, earlier in \cite{kaw:dk}.  

\begin{remark} The fixed point potential $\tau_{X,G,t}^G$ quantizes
  the fixed point contributions in Kalkman's formula \eqref{kalkcirc},
  in the following sense: Let $C$ be a genus zero curve and consider
  the $n$-th Taylor coefficient $\tau_{X,t,\zeta}^{G,n}$. Consider the
  integral with insertion of the class $\beta_n \in H(\ol{\M}_n(C))$
  given by
\[ \tau_{X,t,\zeta}^{G,n}(\alpha,\beta_n) := \sum_d q^d
\int_{[\ol{\M}_n^G(C,X,\LL_t,\zeta,d)]} \ev_n^* \alpha \cup f^* \beta_n
\cup
  \Eul(\nu_t)^{-1}.\]
If $\beta_1 \in H(\ol{\M}_1(C)) \cong H(C)$ is the point class then
\begin{equation} \label{kclaim} \tau_{X,t,\zeta}^{G,1} (\alpha, \beta_1)
  |_{q = 0} = \int_{[X^{\zeta,t}]} \alpha \cup
  \Eul_G(\nu_{X^{\zeta,t}})^{-1} \end{equation} 
which is the contribution from the fixed point components appearing in
Kalkman's wall-crossing formula. Indeed any map from the genus zero
curve $C$ to $X/G$ of homology class $0$ consists of a trivial bundle
and constant section. It follows that
\[ \ol{\M}_n^G(C,X,\LL_t,\zeta,0) \cong X^{\zeta,t} \times
\ol{\M}_n(C).\]
\label{c14} In particular, for $n = 1$ we obtain
\[ \ol{\M}_1^G(C,X,\LL_t,\zeta,0) \cong X^{\zeta,t} \times C \] 
which implies the claim. Furthermore, the Euler class $\Eul(\nu_t)$ is
the Euler class of the virtual normal complex to $X^{\zeta,t}$.
\end{remark}

\begin{remark}   We have stated the formula \eqref{projfixedeq} in its simplest form;
  there are various extensions which include:
\begin{enumerate} 
\item {\rm (Twistings by Euler classes)} One can introduce {\em
    twisted gauged Gromov-Witten invariants} as follows.  The
  universal curve 
\[  p: \ol{\cC}_n^G(C,X) \to \ol{\M}_n^G(C,X) \] 
admits a universal gauged map 
\[  e: \ol{\cC}_n^G(C,X) \to X/G . \]  
For any complex of $G$-equivariant vector bundles $E \to X$ denote
by
\begin{equation} \label{index} \Ind(E) :=  Rp_* e^* E \end{equation} 
the index of the complex $E/G \to X/G$.  The complex $\Ind(E)$ is an
object in the bounded derived category of $\ol{\M}_n^G(C,X)$.  Indeed,
$p$ is a local complete intersection morphism \cite[Appendix]{co:qrr}
and so $Rp_* e^* E $ admits a resolution by vector bundles.  The Euler
class
\begin{equation} \label{epsE} 
\eps(E) := \Eul_{\C^\times}(\Ind(E)) \in H(\ol{\M}_n^G(C,X))[\xi,\xi^{-1}]
\end{equation}
is well-defined $\xi \in H^2_{\C^\times}(\pt)$ is the equivariant
parameter. For any equivariant bundle $E$ on $X$, inserting the Euler
class of \eqref{epsE} gives twisted gauged Gromov-Witten invariants.
Introducing similar twistings in the definition of $\kappa_X^G$, the
wall-crossing formula extends to this case as well.

\item {\rm (Wall-crossing for individual Gromov-Witten invariants)}
  Although we have written the formula \eqref{projfixedeq} as a
  difference of potentials, after unraveling the definitions one
  obtains wall-crossing formulas for individual Gromov-Witten
  invariants, or at least finite combinations of them.  See Example
  \ref{c1} below.
\item {\rm (Wall-crossing for non-convex actions)} In some cases, the
  action of $G$ on $X$ is not convex at infinity (e.g. $G$ is a torus
  acting on a vector space $X$ with weights
  $\mu_1,\ldots, \mu_k \in \g^\dual$ not contained in any open
  half-space $H \subset \g^\dual$) and the moduli spaces
  $\ol{\M}_n^G(C,X)$ are non-compact. Often, the moduli spaces
  $\ol{\M}_n^G(C,X)$ admit an auxiliary group action with proper fixed
  point loci, and thus the maps $\tau_{X \qu G}, \kappa_X^G$ etc. can
  be defined via localization.  The wall-crossing formula holds in
  this case as well, as long as the auxiliary group action extends to
  the various master spaces involved and the fixed point loci on these
  master spaces is proper (after fixing the homology class of the
  gauged map.)  See Example \ref{simpleflop} below.
\end{enumerate} 
\end{remark}

\begin{example} {\rm (Three-point Gromov-Witten invariants of projective space)}
\label{threepoint} 
The following simple example illustrates the notation involved in
Theorem \ref{qkalkcirc}. As in Example \ref{projspace}, let
$G = \C^\times$ acting on $X = \C^k$ by scalar multiplication, so that
$H_G(X) = \Q[\xi]$ where $\xi$ is the equivariant parameter.  Suppose
that polarizations $\LL_\pm$ correspond to the characters $\pm
1$ and 
\[X \qu_- G = \emptyset, \quad X  
\qu_+ G = \P^{k-1} \]
and the two chambers are separated by the value $t = 0$ so that  
$0 \in X$ is semistable for $\LL_-^{(1-t)/2} \otimes \LL_+^{(1+t)/2}$.  
Let $\omega \in H^2(X \qu_+ G)$ denote the hyperplane class. For
integers $a,b,c \ge 0$ we compute genus $0$, $d=1$, $n=3$ invariants
$ \lan \omega^a, \omega^b, \omega^c \ran_{0,1} $ of $\P^{k-1}$ via
wall-crossing; this was already covered in Cieliebak-Salamon
\cite{ciel:wall}. Of course from the elementary computation of quantum
cohomology of projective space one knows that
\[
  \lan \omega^a, \omega^b, \omega^c \ran_{0,1} = \begin{cases} 1 & a + b + c = 2k- 1 
    \\ 0 & \text{otherwise} .\end{cases}\]

  First we relate the above three-point invariant to a gauged
  invariant. Since $c_1^G(X) = k \xi$, the minimal Chern number of $X$
  is $k$.
  For dimensional reasons, there are no quantum corrections in
  $ D_0 \kappa_X^G(\xi^i), i \leq k - 1$.  Thus 
\[ D_0 \kappa_X^G(\xi^i) = \omega^i, i \leq k-1 .\] 
The adiabatic
  limit Theorem \ref{largerel} with insertions implies that the genus
  zero, three-point invariants
  $\lan \omega^a, \omega^b, \omega^c \ran_{0,d} $ of class
  $d \in H_2(X \qu G) \cong \Z$ in the quotient $X \qu G$ equal gauged
  Gromov-Witten invariants:
\[ \lan \omega^a, \omega^b, \omega^c \ran_{0,d} = \int_{[ \ol{\M}_3^G(\P, X,d)]} \ev_1^* \xi^a \cup \ev_2^* \xi^b \cup 
\ev_3^* \xi^c \cup f^* \beta_3 \]
where $\P$ is the projective line and $\beta_3 \in H^6(\ol{\M}_3(\P))$
is the point class, that is, the class fixing the location of the
three marked points.  

We apply the wall-crossing formula Theorem \ref{qkalkcirc} to compute
the gauged invariant. There are no holomorphic spheres in $X$, so the
moduli stack $\ol{\M}_0^{G}(\P,X,\LL_t,d)^G$ is a point, for \(t=0\),
consisting of the bundle $P$ with first Chern class
$c_1(P) = d \in H^2_G(X,\Z) \cong \Z$ with constant section equal to
zero. Thus the fixed point stack is
\[ \ol{\M}_3^{G}(\P,X,\LL_{t=0},\zeta,d) \cong \ol{\M}_3(\P) .\]
for any non-zero $\zeta \in \g$. 
The index bundle of $TX$ at the fixed point for $d = 1$ is \label{c15}
\[ \Ind(T_{X/G}) |_0 = H^0(\P, \mO_\P(1)^\times \times_{\C^\times}
\C^k) \cong \C^{2k} .\]
It has Euler class
$\Eul(\nu_{t=0}) = \xi^{2k} .$
By the wall-crossing formula \ref{qkalkcirc}
\begin{eqnarray*} 
\lan \omega^a, \omega^b, \omega^c \ran_{0,1} &=& \Res_\xi 
\int_{[\ol{\M}_3^{G}(\P,X,L_0,\zeta,1)]} \frac{ \ev_1^* \xi^a \cup \ev_2^*
\xi^b \cup \ev_3^* \xi^c \cup f^* \beta_3}{\Eul(\nu_{t=0})} \\ &=& \Res_\xi 
\xi^{a + b + c} / \xi^{2k} \\ &=& \begin{cases} 1 & a + b + c = 2k- 1 
  \\ 0 & \text{otherwise} .\end{cases}
\end{eqnarray*}   
Thus 
\[\omega^a \star \omega^b = q \omega^{a + b - k}, \quad k \leq a +
b \leq 2k-1 \] 
as is well-known. For example, taking $a = b = k-1$ we obtain that
there is a unique line in projective space passing through two generic
points and a generic hyperplane.  We give another Fano example in
Example \ref{c1}, where we compute the change in a coefficient in the
fifth quantum power of the first Chern class for a blow-up of the
projective plane.  This ends the example.
\end{example} 

The most interesting case of the wall-crossing formula Theorem
\ref{qkalkcirc} is the {\em crepant case} by which we mean that the
sum of the weights at any fixed point vanishes (Definition
\ref{crepant}); the term {\em crepant} was introduced by Reid
\cite{reid:min} as the opposite of discrepant in the context of the
minimal model program.  In the last section we give a proof of a
version of the crepant transformation conjecture of Li-Ruan
\cite{liruan:surg}, Bryan-Graber \cite{bryan:crep},
Coates-Corti-Iritani-Tseng \cite{coates:computing}, Coates-Ruan
\cite{cr:crep} on equivalence of Gromov-Witten theories in this case:
We say that two elements of $\Map(H_2(X \qu G,\Q), \Z)$ are equal
almost everywhere (a.e.)  in the quantum parameter $q$ if their
difference is an element of the form $\sum_\beta f(\beta) q^\beta$
with $f(\beta + d \delta), d \in \Z$ polynomial in $d$. If, as
distributions, these functions are tempered then by Fourier transform
the difference is supported on a set of measure zero.

\begin{theorem} [Wall-crossing for crepant birational transformations 
of git type]
\label{cytype} 
Suppose that $X,G$ are as in Theorem \ref{qkalkcirc}, and $C$ has genus
zero.  If all the wall-crossings are crepant then
\[ {\tau}_{X \qu_- G} 
\circ  {\kappa}_{X,-}^G  
 \underset{a.e.}{=} {\tau}_{X \qu_+ G}
\circ {\kappa}_{X,+}^G  .\]
\end{theorem}

\begin{remark}  \label{cytyperem}
\begin{enumerate} 
\item 
Theorem \ref{cytype} implies many of the special cases already known
in the literature, although actually computing the transformations
$\kappa_{X,\pm}^G$ relating the graph potentials $\tau_{X \qu_\pm G}$
can be a difficult task.  Note that Iwao-Lee-Lin-Wang \cite{lee:flop},
Lee-Lin-Wang \cite{lee:fmi} extend the invariance to cases not
necessarily related by variation of git, while the results here are
more general than that of \cite{lee:fmi}, \cite{lee:flop} since we
allow ``weighted flops''.  More recently Coates-Iritani-Jian
\cite{coates:crep} have proved a version of the crepant transformation
conjecture for toric complete intersections, which overlaps in many
cases with the results here.  
\item The results here are for the graph potential, whereas the
  results in Coates-Iritani-Jiang \cite{coates:crep} are for the
  fundamental solution.  One natural expects the results here to
  extend to the case of localized graph potentials (fundamental
  solutions) using the results of Halpern-Leistner \cite{hal:der} and
  Ballard-Faver-Katzarkov \cite{ball:var}.  We hope to return to
  this in future work.
\item \label{almosteverywhere} Almost everywhere equality in the
  formal parameter $q$ means the following: considering both sides as
  elements in $\Map(H_2^G(X,\Z)/\on{torsion},\Q)$, \label{c3} the
  difference is a polynomial in at least one direction.  In
  particular, in the case of a single quantum parameter the difference
  is tempered distribution its Fourier transform in that direction has
  support of measure zero, see Section \ref{cy}.
\item The proof of Theorem \ref{cytype} uses an action of the Picard
  stack of the curve on the fixed point stacks, see Lemma \ref{pic}.
  In the crepant case the (almost) invariance under this action
  implies that, after summing over degrees, the wall-crossing term is
  a sum of derivatives of delta-functions in the quantum parameter.
\end{enumerate}
\end{remark}

\begin{example} \label{simpleflop} {\rm (Simple three-fold flop,  
cf. Li-Ruan \cite{liruan:surg}, Iwao-Lee-Lin-Wang \cite{lee:fmi},
Lee-Lin-Wang \cite{lee:flop})} The following simple example may help
  to explain the notion of ``almost everywhere vanishing'' of the
  wall-crossing contributions in the quantum parameter.  Let $G =
  \C^\times$ acting on $X = \C^4$ with weights $\pm 1$ each of
  multiplicity $2$.   Let $\LL$ be the trivial polarization and
  $\LL_t$ the trivial bundle shifted by tensoring with a
  representation of weight $t$. The invariant sections of $\LL_t$ are
  spanned by monomials 
  \[ z_1^{d_1} z_2^{d_2} z_3^{d_3} z_4^{d_4} \in H^0(\LL_t)^G, \quad
  d_1 + d_2 - d_3 - d_4 = -t .\]
  Thus the semistable locus $X^{\ss,t}$ is either $(z_1,z_2) \neq 0$
  or $(z_3,z_4) \neq 0$ depending on the sign of $t$, and the git
  quotients $X \qu_\pm G$ factor over
  $\P = (\C^2 - \{ 0 \})/\C^\times$ giving identifications
\[X \qu_\pm G = \mO_{\P}(-1)^{\oplus 2} ,\] 
where $\P$ is the projective line and we abuse notation by denoting by
$\mO_{\P}(-1)^{\oplus 2}$ the total space of two copies of the
tautological line bundle $\mO_{\P}(-1)$. The quotients $X \qu_+ G$ and
$X \qu_- G$ are isomorphic but the birational transformation relating
them, induced by the variation of git quotient, is a simple flop. 

We consider the wall-crossing formula corresponding to the three-point
invariants with each insertion given by the hyperplane class. In order
to make sense of the non-compact integration in the $d = 0$ case one
must take an equivariant extension with respect to the action of
$\C^\times$ acting by scalar multiplication, and use localization for
the residual $\C^\times$-action. Thus if $\theta$ is the equivariant
parameter for the auxiliary $\C^\times$ and $\xi \in H_G^2(X)$ is the
hyperplane class, \label{c16}
\[\kappa_X^G(\xi + \theta) = \pm \omega + \theta\] 
where $\omega \in H^2(X \qu_\pm G)$ is the symplectic class
integrating to $1$ on the zero-section
$\P \subset \mO_{\P}(-1)^{\oplus 2}$. The fixed point set of the
residual $\C^\times$ action is the zero section of $X \qu_\pm G$,
which acts with weights $\pm 2$ (with multiplicity two) on the normal
bundle. The degree zero moduli space is $X \qu_\pm G$ itself, and
integration of $\alpha = \xi^3$ \label{c17} over this non-compact
space can be defined via localization as
\begin{eqnarray*}  I_\pm &:=& \int_{X \qu_\pm G} \kappa_{X,\pm}^G (\alpha) 
  := \int_{\P} (\pm \omega + \theta)^3 \cup
  \Eul( \mO_{\P}(-1)^{\oplus 2})^{-1} \\ &=& 
  \int_{\P} \frac{(\pm \omega + \theta )^3}{ (- \omega \mp 2 \theta)^{2}}
  \\ &=& \int_{\P} \pm 3 \omega \theta^2 (2
         \theta)^{-2} + \theta^3 (2 \theta)^{-2} ( 1 - \omega/( \pm 2 \theta) +
  ....)^2 \\ &=& \int_{\P} \frac{\theta}{4}\pm \frac{3\omega}{4} \mp
  \frac{\omega}{4}\\ &=&\pm 3/4 \mp 1/4 = \pm 1/2 .\end{eqnarray*}
The auxiliary circle action naturally extends to the various master spaces
involved, and the wall-crossing formula holds in this case as well,
defining each contribution via localization.  There is a unique fixed
point $0 \in X$, with normal bundle isomorphic to $X$.  Thus the
wall-crossing term is
\[ \Resid_\xi \xi^3/( \xi^2 (-\xi)^2)  = 1 .\]
Thus Kalkman's wall-crossing formula reduces to 
\[I_+ - I_- = (1/2) - (-1/2) = 1 .\]

We now study wall-crossing for invariants of positive degree. For each
$d \in H_2^G(X,\Z) \cong \Z$, there is a unique gauged map
$u: C \to X/G$ of class $d$ mapping $C$ to the fixed point $0 \in X$.
Its normal complex $Rp_* e^* TX$ has weight $1$ with multiplicity
$2 + 2d$ and weight $-1$ with multiplicity $2 - 2d$, by
Riemann-Roch. Thus the wall-crossing term for class $d$ is
\[ \Resid_\xi \xi^3/( \xi^{2 + 2d} (-\xi)^{2-2d}) = 1.\]
The class $d$ occurs in the expression
$\tau_{X \qu_\pm G} \kappa_{X,\pm}^G$ iff $\pm d \ge 0$. \label{resp}
Indeed, the contribution of $q^d$ is the contribution of $q^d$ to the
gauged Gromov-Witten potential in Theorem \ref{largerel}, which by
definition is an integral over Mundet stable maps $\ol{\M}^G(C,X,d)$,
which in this case is the git quotient of
\[ \ol{\M}^G(C,X,d)  \cong H^0(\mO_C(d) \times_{\C^\times} (\C^{\oplus 2}_1 \oplus \C^{\oplus
  2}_{-1})) \qu \C^\times  \] 
by the $\C^\times$ action corresponding to the polarization $\LL_\pm$,
where $\C_{ \pm 1}$ denote the one-dimensional representations with
weight $\pm 1$.  For $d \neq 0$, only one factor
\[H^0(\mO_C(d)^\times \times_{\C^\times} (\C^{\oplus 2}_{\pm 1})) \cong
\C_{\pm 1}^{\oplus 2 (|d| + 1)}\]
\label{c18} is non-vanishing. It follows easily that the moduli space
of gauged maps of class $d$ is empty for the polarization $\LL_\mp$
corresponding to the character $\mp 1$. Each $q^d, d \neq 0$ appears
in only one of the Novikov rings $\Lambda_{X,\pm}^{G,\ge 0}$. Thus the
higher degree integrals for $\LL_+$ are $1$ for $d > 0$ and $0$ for
$d < 0$ resp. for $\LL_-$ are $- 1$ for $d< 0 $ and $0$ for $d > 0$,
each corresponding to an integral over multiple covers of the zero
section. Summing over classes $d$ the wall-crossing formula for gauged
invariants in Theorem \ref{qkalkcirc} becomes
\[ \left(\frac{1}{2} + \sum_{d > 0} q^d \right) - \left( - \frac{1}{2} - \sum_{d < 0}  q^d \right) = \sum_{d \in
  \Z} q^d \underset{a.e.}{=} 0 .\]
The reader may compare with the treatment of simple flops in
\cite{liruan:surg}, \cite[Corollary 3.2]{lee:fmi} which contains
essentially the same computation.  Note that here we have not given an
explicit description of the maps $\kappa_{X,\pm}^G$ which relate the
two graph potentials.  This ends the example.
\end{example}  

\section{Kalkman's wall-crossing formula} 

In this section we give a proof of Kalkman's formula
\ref{kalk1} first for circle group actions and then for the general
case. 

\subsection{The wall-crossing formula for circle actions} 

The wall-crossing formula is somewhat simpler for the case of a circle
group, so we begin with that case.  Let $G = \C^\times$ and $X$ a
smooth projective $G$-variety as above, equipped with polarizations
$\LL_\pm \to X$.  The proof of the wall-crossing formula is by
localization on a proper Deligne-Mumford stack $\ti{X}$ whose fixed
points include $X \qu_\pm G$ and the fixed point components $X^{G,t}$
with $t \in (-1,1)$.  From the point of view of symplectic geometry,
this is the {\em symplectic cut} construction in Lerman \cite{le:sy2},
but we need the algebraic approach here given in Thaddeus
\cite[Section 3]{th:fl}.

Recall from the introduction the notation for variation of git
quotient. Let
\[\LL_t := \LL_-^{(1-t)/2}
\otimes \LL_+^{(1+t)/2} , t \in (-1,1)\cap \Q\] 
denotes the family of rational polarizations interpolating between
$\LL_\pm$.  Denote by $X \qu_t G$ the corresponding git quotients,
by which we mean the stack-theoretic quotient $X^{\ss,\LL_t}/G$ of the
semistable locus $X^{\ss,\LL_t}$ for $\LL_t$ by the action of
$G$. (Most authors would enclose a stack-theoretic quotient by square
brackets, which we omit since we always mean stack-theoretic quotient
unless otherwise stated.)  In symplectic geometric terms, this means
that if
\[\Phi_\pm: X \to \g_\R^\dual \]
are moment maps for action of a maximal compact $G_\R$ of $G$ on
$\LL_\pm$ with respect to a unitary connection then 
\[ \Phi_t := \frac{ 1-t}{2}
\Phi_-  + \frac{1 + t}{2} \Phi_+ : X \to \g_\R^\dual \]
is a moment map for the action of $G$ on $\LL_t$. Even more concretely,
if $\Phi_\pm$ are equal up to a constant $c = \Phi_+ - \Phi_-$ then
the maps
\[ \Phi_t = \frac{\Phi_- + \Phi_+}{2} + \frac{t}{2} c \]  
are all equal up to a constant $\frac{t}{2}c $ depending on $t$.  For
the following see Thaddeus \cite[Section 3]{th:fl}.

\begin{lemma} [Existence of a master space] \label{master} Suppose
  that $G$ acts with finite stabilizers on the semistable loci
  $X^{\ss,\pm}$ and for any $t \in (-1,1)$ and any $t$-semistable
  point $x \in X$, $G_x$ acts with finite stabilizer on the fiber
  $(\LL_+\otimes \LL_-^{-1})_x$.  There exists a smooth proper
  Deligne-Mumford $\C^\times$-stack $\ti{X}$ equipped with a line
  bundle ample for the coarse moduli space whose git quotients
  $ \ti{X} \qu_t \C^\times$ are isomorphic to those $ X \qu_t G$ of
  $X$ by the action of $G$ with respect to the polarization $\LL_t$
  and whose fixed point set $\ti{X}^{\C^\times}$ is given by the union
\[ \ti{X}^{\C^\times} = (X \qu_- G) \cup (X \qu_+ G) \cup \left( \bigcup_{t \in
    (-1,1)} X^{G,t}\right) \]
where $X^{G,t}$ is the component of $X^G$ that is semistable for
parameter $t$.  Furthermore, the normal bundle of $\nu_{G,t}$ of
$X^{G,t}$ in $\ti{X}$ is isomorphic to the normal bundle of $X^{G,t}$
in $X$, equivariantly after the identification $G \cong \C^\times$.
\end{lemma}

\begin{proof} The projectivization $\P(\LL_- \oplus \LL_+)$ of the
  direct sum $\LL_- \oplus \LL_+$ of the polarizations $\LL_\pm \to X$
  has a natural polarization given by the relative hyperplane bundle
  $\mO_{\P(\LL_- \oplus \LL_+)}(1)$ having fibers \label{c19}
\begin{equation} \label{relhyp} \mO_{\P(\LL_- \oplus \LL_+)}(1)_{[l_-,l_+]} 
  = \on{span}(l_- +
l_+) .\end{equation} 
Let 
\[ \pi: \P(\LL_- \oplus \LL_+) \to X \]  
denote the projection.  (A word of warning: the notation $\pi$ will be
used for a number of different projections in this paper.)  The group
$\C^\times$ acts on $\P(\LL_- \oplus \LL_+)$ by
\[ w [l_-,l_+] = [l_-, wl_+], \quad w \in \C^\times, l_\pm \in \LL_\pm
. \] 
For $k \ge 0$ the space of sections of
$\mO_{\P(\LL_- \oplus \LL_+)}(k)$ has a decomposition
\begin{equation} \label{hdecomp} H^0(\P(\LL_- \oplus \LL_+),\mO_{\P(\LL_- \oplus \LL_+)}(k))
\cong \bigoplus_{k_- + k_+ = k} 
H^0(X, \LL_-^{\otimes k_-} \otimes \LL_+^{\otimes
  k_+})  \end{equation} 
under the natural $\C^\times$-action with eigenspaces given by the
sections of 
\[ \pi^* \LL_-^{k_-} \otimes \pi^* \LL_+^{k_+}, \quad k_- + k_+ = k,
k_\pm \ge 0 .\]  
The $G$-semistable locus in $\P(\LL_- \oplus \LL_+)$ is the union of
non-vanishing loci of \label{c20} non-zero invariant eigensections.  Hence
\begin{equation} \label{union} \P(\LL_- \oplus \LL_+)^{\ss} =  
  X^{\ss,-} \cup X^{\ss,+} \cup \bigcup_{t \in [-1,1]}
\pi^{-1}(X^{\ss,t}) \cap (\P(\LL_- \oplus \LL_+) - \P(L_-) -
\P(\LL_+))\end{equation} 
where $X^{\ss,t} \subset X$ is the semistable locus for $\LL_t$ and
$X^{\ss,\pm}$ are considered subsets of $\P(\LL_- \oplus \LL_+)$ via
the isomorphisms $X \to \P(\LL_\pm)$.  Let
\[ \ti{X} := \P(\LL_- \oplus \LL_+) \qu G \] 
denote the geometric invariant theory quotient, by which we mean the
stack-theoretic quotient of the semistable locus.  The assumption on
the action of the stabilizers in Lemma \ref{master} implies that the
action of $G$ on the semistable locus in $\P(\LL_- \oplus \LL_+)$ has
only finite stabilizers. \label{c21}   It follows that $\ti{X}$ is a proper smooth
Deligne-Mumford stack.  The quotient $\ti{X}$ contains the quotients
of $\P(\LL_\pm) \cong X$ with respect to the polarizations $\LL_\pm$,
that is, $X \qu_\pm G$.

Next we describe the fixed point set. The fixed points for the action
of $\C^\times$ on $\ti{X}$ are represented by pairs $[l_-,l_+]$ with a
positive dimensional stabilizer under the action of
$G \times \C^\times$. Necessarily either $l_- = 0, l_+ = 0$, or
$l_-,l_+$ are both non-zero but the projection to $X$ is $G$-fixed. In
the latter case, semistability implies that there exists an invariant
section of
$\LL_t=\LL_-^{(1-t)/2}\otimes \LL_+^{(1+t)/2}, t \in [-1,1]$,
non-vanishing at $x$. Since $l_\pm$ are both non-zero, the weight $t$
of $\C^\times$ on the fiber cannot be in $\{-1, 1 \}$, hence $x$
is $t$-semistable for some $t \in (-1,1)$.  The normal bundle $\nu $of
$\ti{X}^{\C^\times}$ lifts to the normal bundle $\ti{\nu}$ for the
fixed point set some one-parameter subgroup in
$\P(\LL_- \oplus \LL_+)$, which projects to the normal bundle in
$X^{\C^\times}$.  Both projections have trivial fiber, hence the claim
on normal bundles.
\end{proof}

\begin{proof}[Proof of the classical Kalkman's wall-crossing Theorem \ref{kalk1}]   
  First note that the integrand in the wall-crossing formula lifts to
  the master space in a natural way: The projection
  $\P(\LL_- \oplus \LL_+) \to X$ is $G$-equivariant and
  $\C^\times$-invariant. Composing the pull-back map
\[H_G(X) \to H_{G \times \C^\times}( \P(\LL_- \oplus \LL_+))\] 
with the Kirwan map
\[H_{G \times \C^\times}( \P(\LL_- \oplus \LL_+)) \to H_{\C^\times}(\ti{X})\] 
one obtains a canonical map
\[\delta: H_G(X) \to H_{\C^\times}(\ti{X}) .\]  
The composition of $\delta$ with the Kirwan map
\[\ti{\kappa}_{X,t}^{\C^\times}: H_{\C^\times}(\ti{X}) \to H(\ti{X}
\qu_t \C^\times ) = H(X \qu_t G)\]
is pull-back to the $t$-semistable locus and so equal to
\[ \ti{\kappa}_{X,t}^{\C^\times} \circ \delta = \kappa_{X,t}^G: H_G(X) \to H(X
  \qu_t G) .\] 
  In particular, $\delta(\alpha) \in H_{\C^\times}(\ti{X})$ restricts
  to $\kappa^G_{X,\pm} \alpha$ on the two distinguished fixed point
  components $X \qu_\pm G \subset \ti{X}^{\C^\times}$.

  Taking the residue of the localization formula for the circle action on
  the master space gives Theorem \ref{kalk1}, see Lerman \cite{le:sy2}.
  Indeed for any equivariant class $\alpha \in H_G(X)$ of top degree,
  its pullback to $\P(\LL_- \oplus \LL_+)$ descends to a class
  $\delta(\alpha) \in H_{\C^\times}(\ti{X})$ whose restriction to
  $X \qu_\pm G$ is $\kappa_{X,\pm}^G(\alpha)$, and whose restriction
  to $X^{G,t} \subset \ti{X}^{G}$ is $\iota_{X^{G,t}} \alpha$.  Since
 \[ \deg(\delta(\alpha)) = \deg(\alpha) = \dim(\ti{X}) - 2 ,\]
 the integral of $\delta(\alpha)$ over $\ti{X}$ vanishes.  On the
 other hand, by localization the integral of $\delta(\alpha)$ is
  \begin{equation} \label{locint} \int_{[\ti{X}]} \delta(\alpha) =
    \int_{[X \qu_- G]} \frac{\kappa_{X,-}^G(\alpha)}{\Eul_G(\nu_-)} +
    \int_{[X \qu_+ G]} \frac{\kappa_{X,+}^G(\alpha)}{\Eul_G(\nu_+)} +
    \sum_{t \in (-1,1)} \int_{[X^{G,t}]} \frac{\iota^*_{X^{G,t}}
      \alpha}{\Eul_G(\nu_{X^{G,t}})} \end{equation}
where $\nu_\pm \to X \qu_\pm G$ are the normal bundles to
$X \qu_\pm G$ in $\ti{X}$.  Since the normal bundle of the ``sections
at zero and infinity'' $\P(\LL_\pm)$ in $\P(\LL_- \oplus \LL_+)$ may
be canonically identified with $(\LL_+ \otimes \LL_-^{-1})^{\pm 1}$,
the group $\C^\times $ acts on $\nu_\pm $ with weights $\mp 1$.  Hence
the inverted Euler classes are
\[ \Eul_G(\nu_\pm)^{-1} = ( \mp \xi + c_1(\nu_\pm))^{-1} = \mp \xi^{-1}
(1 + c_1(\nu_\pm) (\mp \xi)^{-1} + \ldots) .\]
Taking residues on both sides of \eqref{locint} one obtains
\[ 0 = - \tau_{X \qu_+ G} \kappa_{X,+}^G \alpha + \tau_{X \qu_- G} 
\kappa_{X,-}^G \alpha + \Res_\xi \sum_{t \in (-1,1)} 
\int_{[X^{G,t}]} \frac{\iota^*_{X^{G,t}} \alpha}{ \Eul_G(\nu_{X^{G,t}})} \]
as claimed.
\end{proof} 

\subsection{Kalkman wall-crossing for actions of non-abelian groups} 

Kalkman's wall-crossing formula can be used to study the intersection
pairings for variation of git for the action of an arbitrary connected
complex reductive group $G$ on a smooth polarized projective variety
$X$ as follows.  As above, let 
\[ \ti{X} = \P(\LL_- \oplus \LL_+) \qu G .\]
We examine the structure of the fixed point set of $\C^\times$ on
$\ti{X}$. For any $\zeta \in \g$, denote by $G_\zeta \subset G$ the
stabilizer of the line $\C\zeta$ under the adjoint action of $G$.
Recall from the introduction that 
\[   X^\zeta = \{ x \in X \ | \ z x = x, \ \forall z \in
\C^\times_\zeta \} \]
is the fixed point set of the one-parameter subgroup
$\C^\times_\zeta$.  Since $G_\zeta$ commutes with $\C^\times_\zeta$,
it acts on $X^\zeta$.  The $t$-semistable locus
\[ X^{\zeta,t} \subset X^\zeta \]
has, by assumption, the property that the action of
$G_\zeta/\C^\times_\zeta$ has finite stabilizers, that is,
$ \g_x = \C\zeta$ for all $ x \in X^{\zeta,t} .$
It follows the flowout of the semistable locus is 
\[ G X^{\zeta,t} = G \times_{G_\zeta} X^{\zeta,t} .\]
\label{c211}  In particular, there exists a canonical action of $\g/\g_\zeta$
(considered as an abelian group) fiber-wise on the normal bundle
$\nu_{X^{\zeta,t}}$.  Denote by $\nu_{X^{\zeta,t}}/(\g/\g_\zeta)$
the quotient by the action.

\begin{lemma}[Structure of the fixed point components]
\label{fplem}
Suppose that stable=semistable for the $G$-action on $\P(\LL_- \oplus
\LL_+)$, so that $\ti{X}$ is a smooth proper Deligne-Mumford stack with
$\C^\times$ action, constructed in the proof of Lemma \ref{master}.
\begin{enumerate}  
\item \label{itema} For any fixed point $ \ti{x} \in \ti{X}^{\C^\times}$ with $\ti{x} =
  [l]$ for some $l \in \P(\LL_- \oplus \LL_+)$, there exists
$\zeta \in \g$ such that 
\[ \forall z \in \C^\times, \quad  zl = z^\zeta l .\]
\item \label{itemb} Any fixed point $ \ti{x} \in \ti{X}^{\C^\times}$
  is equal to $[l]$ for some $l \in \P(\LL_- \oplus \LL_+)$ in the
  fiber over $x \in X$ that is $t$-semistable for some $t \in (-1,1)$
  and has stabilizer generated by $\zeta \in \g$, with the property
  that the weight of the one-parameter subgroup generated by $\zeta$
  on $(\LL_-^{(1-t)/2)} \otimes \LL_+^{(1+t)/2})_x$ vanishes:
\[ z^\zeta l = l, \quad \forall l \in (\LL_-^{(1-t)/2)} \otimes
\LL_+^{(1+t)/2})_x .\]
\item \label{itemc} Denote by
  $X^\zeta \qu_t (G_\zeta/ \C^\times_\zeta)$ the git quotient of
  $X^\zeta$ by the group $G_\zeta/\C^\times_\zeta$ with respect to the
  polarization determined by the restriction of
  $\LL_-^{(1-t)/2} \otimes \LL_+^{(1+t)/2}$.  For each $\zeta \in \g$
generating a one-parameter subgroup $\C^\times_\zeta$ \label{c22}
  there is a morphism
\[ \iota_\zeta:
X^\zeta \qu_t (G_\zeta/ \C^\times_\zeta) \to \ti{X}^{\C^\times}.\] 
The images of all the $\iota_\zeta$ cover $\ti{X}^{\C^\times}$,
disjointly after passing to conjugacy classes of one-parameter
subgroups $\C^\times_\zeta$. \label{c23}
\item \label{itemd} For any $\alpha \in H_G(X)$, the pull-back of
  $\ti{\kappa}(\alpha) | \ti{X}^{\C^\times}$ under $\iota_\zeta$ is
  equal to image of $\alpha$ under the restriction map $H_G(X) \to
  H_{\C^\times_\zeta}(X^\zeta \qu_t (G_\zeta/\C^\times_\zeta))$.
\item \label{iteme} The pull-back of the normal bundle
  $\nu_{\ti{X}^{\C^\times}}$ of $\ti{X}^{\C^\times}$ under
  $\iota_\zeta$ is isomorphic to the quotient of
  $\nu_{X^{\zeta,t}} / (\g/\g_\zeta)$ by a fractional action of
  $G_\zeta / \C^\times_\zeta$, via an isomorphism that intertwines the
  action of $\C^\times_\zeta$ on
  $(\nu_{X^{\zeta,t}} / (\g/\g_\zeta) ) \qu (G_\zeta /
  \C^\times_\zeta)$
  with the action of $\C^\times$ on $\nu_{\ti{X}^{\C^\times}}$.
\end{enumerate}
\end{lemma} 

\begin{proof} \eqref{itema} Denote by $\P(\LL_- \oplus \LL_+)^\times$
  the complement of $\P(\LL_-) \cup \P(\LL_+)$ in
  $\P(\LL_- \oplus \LL_+)$.  If $[l] \in \ti{X}^{\C^\times}$, with
  $l \in \P(\LL_- \oplus \LL_+)^\times$ then $[l] \in \ti{X}^\xi$,
  where $\xi$ is a generator of the Lie algebra of $\C^\times$ and
  $\ti{X}^\xi$ is the zero set of the vector field $\xi_{\ti{X}}$
  generated by $\xi$.  Since $\ti{X}$ is the quotient of
  $\P(\LL_- \oplus \LL_+)$ by $G$, if $\xi_L$ denotes the vector field
  on $\P(\LL_- \oplus \LL_+)$ generated by $\xi$ then
  $ \xi_L(l) = \zeta_L(l)$ for some $\zeta \in \g$.  Since $G$ acts
  locally freely $\zeta$ must be unique.  Integrating gives
  $z\cdot l = z^\zeta l$ for all $z \in \C^\times$, hence
  \eqref{itema}.  \label{fplemproof} \eqref{itemb} is a consequence of
  \eqref{itema} and the definition of semistability in terms of
  invariant sections.  Item \eqref{itemc} is a consequence of items
  \eqref{itema} and \eqref{itemb}.  Item \eqref{itemd} follows from
  the fact that $\iota_\zeta \circ \ti{\kappa}$ is pullback to
  $X^\zeta \qu_t (G_\zeta/\C^\times_\zeta)$.  For \eqref{iteme}, the
  normal bundle $\nu_{\ti{X}^{\C^\times}}$ of $\ti{X}^{\C^\times}$
  restricted to the image of $\iota_\zeta$ is isomorphic to the
  quotient of the normal bundle of
  $G \times_{G_\zeta} \P(\LL_- \oplus \LL_+)^{\xi - \zeta}$ by $G$,
  which in turn is isomorphic to the quotient of the normal bundle of
  $\P(\LL_- \oplus \LL_+)^{\xi - \zeta}$ by $\g/\g_\zeta$.  The
  projection to $X$ identifies this normal bundle with the normal
  bundle $\nu_{X^{\zeta}}$ to $X^\zeta$ quotiented by $\g/\g_\zeta$.
  Thus $\iota_\zeta^* \nu_{\ti{X}^{\C^\times}}$ is isomorphic to a
  quotient of $\nu_{X^{\zeta}}/ (\g/\g_\zeta)$ by
  $G_\zeta/\C^\times_\zeta$.
\end{proof} 

\begin{remark} An anonymous referee has pointed out the following
  alternative perspective on the fixed point loci.  Consider the
  $G_\zeta$-equivariant bundle
\[ \ti{\nu} := \nu_{X^{\zeta,t}} / (\g/\g_\zeta) \to X^{\zeta,t} .\] 
Taking quotients gives a vector bundle
$\ti{\nu}/G_\zeta \to X^{\zeta,t}/G_\zeta $ over the Artin stack
$X^{\zeta,t}/G_\zeta$.  On the other hand, we have a
$\C^\times$-bundle
\[ P_\zeta := (\LL_- \otimes \LL_+^{-1})^\times \to X^{\zeta,t} \] 
which is a sub-bundle of $\P(\LL_- \oplus \LL_+) |_{X^{\zeta,t}}$, and
where superscript $^\times$ denotes removal of the zero section.  The
map $P_\zeta \to X^{\zeta,t}$ is equivariant for the $G_\zeta$-action
and the Deligne-Mumford stack $P_\zeta/G_\zeta$ gives a component of
$\ti{X}^{\C^\times}$.  These spaces fit into a diagram (quotients
being stack-theoretic)
\[\begin{tikzcd} 
P_\zeta/G_\zeta 
\arrow[r,hook]
\arrow[d,"p_\zeta"] 
& \ti{X}^{\C^\times}
 \\
X^{\zeta,t} / G_\zeta  .& 
\end{tikzcd}\]
The pull-back of the vector bundle $\ti{\nu}/G_\zeta 
\to X^{\zeta,t}/G_\zeta$ by $p_\zeta$ is isomorphic to the restriction
of the normal bundle $\nu_{\ti{X}^{\C^\times}}$ to the component
$P_\zeta/G_\zeta$:
\[ p_\zeta^* [\ti{\nu}/G_\zeta] \cong \nu_{\ti{X}^{\C^\times}}
  |_{P_\zeta/G_\zeta} .\]
  Here the left-hand side $p_\zeta^* [\ti{\nu}/G_\zeta]$ is naturally
  a $\C^\times$-equivariant bundle over $P_\zeta/G_\zeta$ (since
  $p_\zeta^* \ti{\nu}$ is a $G_\zeta \times \C^\times$-equivariant
  bundle over $P_\zeta$) and this isomorphism intertwines the
  $\C^\times$-actions.  The various pull-backs fit into a commutative
  diagram of maps in equivariant cohomology
\[
\begin{tikzcd} 
H_G^*(X) 
\arrow[rrr, "\ti{\kappa}"] \arrow[rd] \arrow[d]
& & & H^*_{\C^\times}(\ti{X})  \arrow[d]
\\
H_{G_\zeta}^*(X^{\zeta,t}) \arrow[r,"p_\zeta^*"] 
& 
H^*_{G_\zeta \times \C^\times}(P_\zeta) \arrow[r,equal]
& 
 H^*_{\C^\times}(P_\zeta/G_\zeta) 
& 
H^*_{\C^\times}( \ti{X}^{\C^\times}) 
\arrow[l]
\end{tikzcd}  
\]
It follows that the fixed point contribution from 
$X^{\zeta,t}$ in the wall-crossing formula is given by 
\[ \Res_\xi \int_{P_\zeta/G_\zeta} \ti{\kappa}(\alpha) 
\cup p_\zeta^* \Eul_{G_\zeta}(\ti{\nu})^{-1} . \] 
One may re-write this as an integral over
$X^{\zeta,t}/ (G_\zeta/\C^\times)$ as follows.  Let
$k \in \Z - \{ 0 \} $ be the weight of the $C^\times_\zeta$-action on
fibers of $P_\zeta \to X^{\zeta,t}$; by changing the choice of
one-parameter subgroup by a sign we may assume that $k$ is positive.
The projection $P_\zeta \to X^{\zeta,t}$ defines a $B(\Z/k\Z)$-bundle
\[ \pi_\zeta : P_\zeta/G_\zeta \to
X^{\zeta,t}/(G_\zeta/\C_\zeta^\times) \] 
Choose a splitting of the Lie algebra
\[ g_\zeta \cong \C \zeta \oplus (\g_\zeta/\C \zeta) .\]
Let 
\[ \xi \in H^2_{\C^\times}(\pt), \quad  \hbar \in H^2_{\C^\times_\zeta}(\pt) \] 
denote the standard generators.  Consider the commutative diagram
\[ 
\begin{tikzcd} 
H^*_{\C^\times}(P_\zeta/G_\zeta) = 
H_{G_\zeta \times
  \C^\times}(P_\zeta)  
= H_{G_\zeta}^*(X^{\zeta,t})  \arrow[r,"\cong"] 
& H^*(X^{\zeta,t}/(G_\zeta/\C^\times_\zeta)) \otimes 
\C(\hbar)  
 \\ 
H^*(P_\zeta/G_\zeta) 
\arrow[u,hook]
& 
H^*(X^{\zeta,t} / (G_\zeta/\C^\times_\zeta)) 
\arrow[u,hook] \arrow[l,"\pi^*_\zeta"] \arrow[u,hook]
\end{tikzcd} .
\] 
Note that the identification of the $\C^\times_\zeta$-equivariant
parameter $\hbar$ in $H^*_{G_\zeta \times \C^\times}(P_\zeta)$
requires the choice of a splitting.  We claim that the pullback of
$\xi$ to $ H^*_{\C^\times}(P_\zeta/G_\zeta)$ corresponds to a class
$k\hbar + \omega \in H^*(X^{\zeta,t}/(G_\zeta/\C^\times)) \otimes
\C(\hbar)$
for some nilpotent element
$\omega \in H^2(X^{\zeta,t} / (G_\zeta/ \C^\times_\zeta)$ where $k$ is
the fiber weight above.  Indeed let
\[ K = \{ (z,z^k) \in
\C^\times_\zeta \times \C^\times  \}  \] 
be the subgroup that acts trivially on $P_\zeta$.  
Then 
\[ \xi - k \hbar{} \in H^2_{G_\zeta \times \C^\times}(P_\zeta) \] 
lies in the image of
\[ H^2_{G_\zeta \times \C^\times} (\pt) 
\to H^2_{(G_\zeta \times \C^\times)/K}(P_\zeta) 
\to H^2_{G_\zeta \times \C^\times}(P_\zeta) .\]
Since $(G_\zeta \times \C^\times)/K$ acts on $P_\zeta$ locally freely,
$\xi - k \hbar$ is nilpotent.  This proves the claim.  The splitting
also defines action of
$(G_\zeta/\C^\times_\zeta) \times \C^\times_\zeta$ on $\ti{\nu}$, so
that
\[ \Eul_{G_\zeta \times  \C^\times}(p^*_\zeta \ti{\nu}) =
    \Eul_{\C^\times_\zeta}(\ti{\nu}/(G_\zeta/\C_\zeta^\times) )\] 
under the isomorphism 
\[ H^*_{G_\zeta \times \C^\times}(P_\zeta) 
\cong H^*( X^{\zeta,t}/(G_\zeta/\C^\times)) 
\otimes \C[\hbar] . \] 
The fixed point contribution can be rewritten as 
\[ \Res_\xi \frac{1}{k} 
\left. 
\int_{X^{\zeta,t}/(G_\zeta/\C^\times_\zeta) }
\alpha|_{X^{\zeta,t} }
\cup \Eul_{\C^\times_\zeta}(\ti{\nu}/(G_\zeta/\C^\times_\zeta))^{-1}
\right|_{\hbar = (\xi - \omega)/ k } \]
where we regard $\alpha |_{X^{\zeta,t}}$ as an element of
  $H^*_{G_\zeta}(X^{\zeta,t}) \cong
  H^*(X^{\zeta,t}/(G_\zeta/\C^\times_\zeta))$ using the splitting and
  the factor $1/k$ arises as the degree of 
\[ \pi_\zeta : P_\zeta/G_\zeta \to
X^{\zeta,t}/(G_\zeta/\C^\times_\zeta) .\]
Under the substitution $\hbar = (\xi - \omega)/k$ we have for $i \in
\Z_+$ 
\[ \hbar^{-i} \mapsto \frac{k^i}{(\xi - \omega)^i} = k^i \left( 
\frac{1}{\xi^i} +  \frac{i \omega}{\xi^{i+1}} 
+ \frac{i (i+1) \omega^2}{ 2 \xi^{i+2}} + \ldots \right) \] 
and thus
\[ \Res_\xi \frac{1}{\hbar^i} |_{\hbar = (\xi - \omega/k)}
= \begin{cases} k & i = 1 \\
0 & \text{otherwise}  \end{cases} \] 
Therefore the above residue equals 
\[ \Res_{\hbar} \int_{X^{\zeta,t}/ (G_\zeta/\C^\times_\zeta)} \alpha
|_{X^\zeta,t} \cup \Eul_{\C^\times_\zeta}(\ti{\nu}/
(G_\zeta/\C^\times_\zeta))^{-1} . \]
This ends the Remark.
\end{remark}

We introduce the following notation for fixed point contributions.

\begin{definition} \label{fpc} For any fixed point component
  $X^{\zeta,t} \subset X^\zeta$ that is $t$-semistable, denote by
  $\nu_{X^{\zeta,t}}$ the normal bundle of $X^{\zeta,t}$ as above.
  Let
\[\tau_{X,\zeta,t}: H_G(X) \to \Q[\xi,\xi^{-1}], \quad
\alpha \mapsto \int_{[X^{\zeta,t} \qu (G_\zeta/\C^\times_\zeta)]}
\frac{\alpha}{\Eul_{\C^\times_\zeta}((\nu_{X^{\zeta,t}}/(\g/\g_\zeta)) / (G_\zeta/\C^\times_\zeta))} \]
where we have omitted the restriction-and-quotient map
$H_G(X) \to H_{\C_\zeta^\times} (X^{\zeta,t}\qu
(G_\zeta/\C^\times_\zeta))$
to simplify notation, and $\xi$ is the equivariant parameter for
$\C^\times_\zeta$.
\end{definition} 

We define an equivalence class on rational elements of the Lie algebra
as follows.  Recall that an element $\zeta \in \g$ is rational if
$\C\zeta$ is the Lie algebra of a one-dimensional subgroup
$\C^\times_\zeta$.  We declare two rational elements $\zeta_0,\zeta_1$
to be {\em equivalent} if the one-dimensional subgroups are conjugate:
\begin{equation} \label{equiv} ( \zeta_0 \sim \zeta_1 ) \iff \left(
    \exists g \in G, \quad \C^\times_{\zeta_0} = g \C^\times_{\zeta_1}
    g^{-1} \right) \end{equation}
or equivalently, $\C\zeta_1$ is related to $\C\zeta_0$
by the adjoint action.  Denote by $[\zeta]$ the equivalent class of a
rational element $\zeta \in \g$.

\begin{theorem} [Kalkman wall-crossing]  \label{kalk2} Let
$X$ be a smooth projective $G$-variety and $\LL_\pm \to X$ polarizations
 with stable=semistable for the $G$ action on $\P(\LL_- \oplus \LL_+)$.
  Then
\begin{equation} \label{kalkman}
\tau_{X \qu_+ G} \kappa_{X,+}^G - \tau_{X \qu_- G} \kappa_{X,-}^G =
\sum_{t \in (-1,1),[\zeta]} 
\Resid_\xi \tau_{X,\zeta,t} \end{equation}
where the sum is over equivalence classes $[\zeta]$.
\end{theorem}  

\begin{proof} This is an immediate consequence of Kalkman's result for
  circle actions, Theorem \ref{kalk1}, applied to the master space
  constructed in Lemma \ref{master}, using the identification of the
  fixed point components and normal bundles in Lemma \ref{fplem}.
\end{proof} 

\begin{example} \label{classexamples}
\begin{enumerate} 
\item \label{blowupP2} \label{eight} (Blow-up of the projective plane
  as a quotient of a product of projective lines by a circle action)
  Let $X = (\P)^3$ with polarization
\[\LL = \mO_{\P}(a) \boxtimes \mO_{\P}(b) \boxtimes \mO_{\P}(c),
\quad a \ll b \ll c  .\] 
We let $G = \C^\times$ acting on each projective line $\P$ by
\[  g[z_0,z_1] = [z_0,g z_1], \quad g \in \C^\times, z_0,z_ \in \C  .\]
Let $G$ act on $\mO_{\P}(n)$ so that the weights at the fixed points
$[1,0]$, $[0,1]$ are $-n/2,+n/2$.  Let $G$ act diagonally on the
factors in $X = (\P)^3$.  Consider the family of polarizations
$\LL_t = \LL \otimes \C_t$ obtained by shifting $\LL$ by a trivial
line bundle with weight $t$.  

The chamber structure for the various git quotients is governed by the
weights of the action on the polarizing line bundle on the fixed
points, given by $(\pm a \pm b \pm c)/2$.  Thus there are nine
chambers, of which two have empty git quotients and seven \label{seven} non-empty
chambers.  In the first and last chamber, we have
$X \qu_t G \cong P(\C^3)$ resp. $P((\C^3)^\dual)$, while the six
wall-crossings represent three blow-ups and three blow-downs involved
in the Cremona transformation.

  We study the application of the Kalkman formula to the square of the
  first Chern class.  That is, let
\[\alpha = c_1^G(X)^2 \in H^4(X) .\] 
Since $TX$ is isomorphic to the pull-back of $T (X \qu_t G)$ plus a
trivial line bundle with fiber $\g$, we have
\[ \kappa_{X,\pm}^G ( c_1^G(X)) = c_1(X \qu_t G) .\]
Consider the wall-crossing from the chamber $t < -a - b - c$ to the
first non-empty chamber $t \in (a - b - c,-a + b - c)$, corresponding
to the wall-crossing over the ``lowest'' fixed point
$x = ([1,0],[1,0],[1,0]) \in X$. Since all weights of the action on
the tangent bundle at this fixed point $x \in X$ are $1$, we have
\[c_1^G(X) | x = 3\xi \in H_G(\{ x \}) .\] 
Hence
\[ \tau_{X \qu_+ G} \kappa_{X,+}^G(\alpha) - \tau_{X \qu_- G} \kappa_{X,-}^G(\alpha) = \Res_\xi
\frac{(3\xi)^2}{\xi^3} = 9 .\]
Indeed, 
$c_1(\P^2)$ is three times the generator of $H^2(\P^2)$, so
\[ \tau_{X,-}^G \kappa_{X,-}^G(\alpha) = \int_{[\P^2]} c_1(\P^2)^2 = 9 .\]
Consider next the wall-crossing from the chamber with quotient $X
\qu_- G = \P^2$ to the chamber with quotient $X \qu_+ G = \Bl(\P^2)$,
where $\Bl(\P^2)$ is the blow-up of $\P^2$ at a point.  
Letting $\pi: \Bl(\P^2) \to \P^2$ denote the blow-down map we have 
$c_1( \on{Bl} \P^2) = 3 \pi^* H - E $
where $H,E$ are the hyperplane class and class of the exceptional
divisor respectively. Hence 
\begin{eqnarray} \label{square}
\tau_{X \qu_+ G} \kappa_{X,+}^G(\alpha) &=& (3 \pi^* H - E)^2 \\
&=& (3
\pi^* H)^2 - 6 \pi^* H E + E^2 = 9 - 0 -1 = 8 .\end{eqnarray}
To compare this result with the wall-crossing formula, note that the
fixed point set $X^{G,t}$ consists of a unique point
$([0,1],[1,0],[1,0])$ which is semistable exactly for $t = a - b - c$,
with tangent weights $-1,1,1$.  It follows that the first Chern class
squared and Euler classes are
\[c_1( TX | {X^{G,t}} )^2 = (- \xi + \xi + \xi)^2 = \xi^2, \quad
\Eul_G(TX |{X^{G,t}}) = - \xi^3 .\]
Hence the wall-crossing term is 
\[ \Res_\xi \left( \int_{[X^{G,t}]} \iota_{X^{G,t}}^* \alpha \cup
\Eul_G(\nu_{X^{G,t}})^{-1} \right) = \Res_\xi \frac{\xi^2}{-\xi^3} =
-1 .\]
The wall-crossing formula reduces to 
\[ \tau_{X \qu_+ G}(\alpha,\beta) - 
\tau_{X \qu_- G}( \alpha,\beta) =8 - 9 = - 1 . \]  
This matches the well-known fact that each blow-up of $\P^2$ lowers
$c_1^2$ by $1$.
\item \label{twotorus} (Blow-up of the projective plane as a quotient of affine
  four-space by a two-torus) Let us do the same example in a different way, namely as a quotient of an affine space. Let $X = \C^4$ with
  $G = (\C^\times)^2$ acting with weights $(1,0),(1,0),(1,1),(0,1)$.
  Consider the path $\LL_t$ of polarizations whose first Chern classes
  $H^2_G(X,\Q) \cong \Q^2$ are the line segment from $(1,2)$ to
  $(2,1)$.  The chamber structure is determined by the rays generated
  by the weights, so that the ``negative'' chamber is spanned by
  $(0,1),(1,1)$ and the ``positive chamber'' by $(1,1),(1,0)$.  The
  quotient in the negative chamber $X \qu_- G$ is isomorphic to $\P^2$
  via the map
\[ X \qu_- G \to \P^2, \quad [a,b,c,d] = [ (a,b,cd^{-1},1) ] \mapsto
[a,b,cd^{-1}] .\]
On the other hand $X \qu _+ G$ is isomorphic to the blow-up of $\P^2$
with the map to $\P^2$ blowing down the exceptional divisor given by
\[ X \qu_+ G \to \P^2, \quad  [a,b,c,d] \mapsto [a,b,cd^{-1}]  .\] 
As one moves in a line from say $(-1,2)$ to $(2,-1)$ the symplectic
quotients are in order $\emptyset, \P^2, \Bl(\P^2), \emptyset$ where
the initial and final contractions are projective bundles over the
fixed point sets $\on{pt}$ resp. $\P$ for the residual action of
$\C^\times$ on the quotient of $X$ by the diagonal action.  See Figure
\ref{twochamber}, where the two chambers for the possible quotients
are shown together with the moment polytopes of the quotients in each
chamber.

\begin{figure}[ht]
\includegraphics[height=2in]{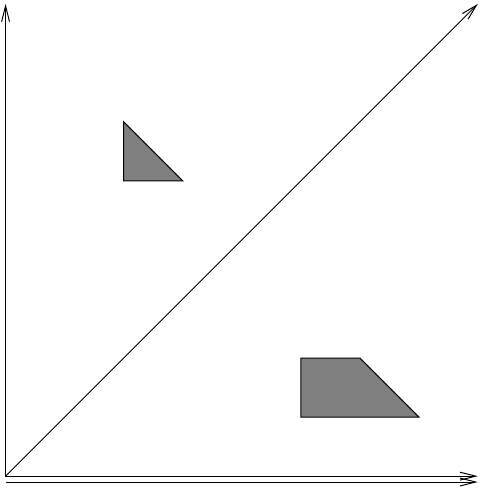}
\caption{Chambers for $(\C^\times)^2$ action on $\C^4$} 
\label{twochamber}
\end{figure}

Consider first the wall-crossing from the empty chamber to the
negative chamber in Figure \ref{twochamber}.  We have
$ c_1^G(X) = 3 \xi_1 + 2\xi_2 $.  Hence the wall-crossing term is
\[ \Res_{\xi = \xi_1}
 \frac{ (3\xi_1 + 2\xi_2)^2}{ \xi_1^2 (\xi_1 + \xi_2)}
|_{\xi_2 = 0 } = 9 \]
hence $\int_{[\P^2]} c_1(\P^2)^2 = 9$ as expected. 

The wall-crossing term for passing from the negative to positive
chamber is the residue:
\[ \Res_{\xi }
\left. \frac{ (3\xi_1 + 2\xi_2)^2}{ \xi_1^2
  \xi_2} \right|_{\xi = \xi_1 = -\xi_2 }  = -1 .\]
We obtain from the wall-crossing formula that 
$\int_{[\Bl(\P^2)]} c_1(\P^2)^2 = 8$
as we already computed in item \eqref{eight} above.  

Consider next wall-crossing from the positive chamber to the empty
chamber.  The fixed point locus contributing to the wall-crossing term
is 
\[X^{\zeta,t}/ (G_\zeta/\C^\times_\zeta) \cong \P \]
since the multiplicity of the weight $(1,0)$ is $2$.  The wall-crossing
term is
\[ \Res_\xi \int_{[\P]} \kappa_{X}^{\C^\times} \frac{ (3\xi_1 +
  2\xi_2)^2}{ \xi_2 (\xi_1 + \xi_2)}  \]
where
$\kappa_{X}^{\C^\times}: H_G(X) \to H_{G/\C^\times}(X \qu \C^\times)$
is the Kirwan map for the quotient by the first factor of $\C^\times$,
so that $X \qu \C^\times$ has $\C^\times$-quotients $X \qu_\pm G$.
Under $\kappa_{X}^{\C^\times}$ we have that $\xi_1$ maps to the 
generator $\omega$ of $H^2(\P)$ while $\xi_2$ maps to the parameter 
$\xi$ for the residual $\C^\times$-action. The wall-crossing term is
therefore
\begin{eqnarray*}
 \Res_\xi \int_{[\P]} \frac{ (3\omega + 2\xi)^2}{ \xi (\omega +
  \xi)} &=& \Res_\xi \int_{[\P]} \frac{ (3\omega + 2\xi)^2}{
  \xi^2} \left( 1 - \frac{\omega}{ \xi} + \frac{\omega^2}{\xi^2} - \ldots \right) \\&=&
  \Res_\xi \int_{[\P]} 12 \frac{\omega}{\xi} - 4 \frac{\omega}{\xi} = -8 . \end{eqnarray*}
Thus as expected the wall-crossing formula reduces to  
$\int_{\emptyset} c_1(\emptyset) = 8 - 8 = 0 .$
\item (Resolution of a crepant singularity) The following example
  illustrates the application of the wall-crossing to orbifolds, in
  which the integrals become rational, and also the non-compact case,
  in which the integrals must be defined via localization.  Suppose
  that $X = \C^{k+1}$ and $G = \C^\times$ acts with weights $1$ with
  multiplicity $k$ and $-k$ with multiplicity $1$.  Take $\LL$ to be
  the trivial polarization, and $\LL_t$ the family obtained by
  shifting by a trivial line bundle with weight $t$.  Thus invariant
  sections are spanned by monomials.
\[ z_0^{d_0} z_1^{d_1} \ldots z_k^{d_k}, \quad -k d_0 + d_1 + \ldots +
d_k =  t .\]
The latter requires $d_0 \neq 0$ for $t < 0$ resp.
$(d_1,\ldots, d_k) \neq 0$ for $t > 0$.  It follows that the
semistable locus for $t< 0$ is $z_0 \neq 0$ and for $t > 0$ the locus
where $(z_1,\ldots, z_k) \neq 0$. The git quotients are then
$X \qu_+ G = \C^k/\Z_k$ (by our conventions, a stack with trivial
canonical bundle) while $X \qu_- G$ is isomorphic to the total space
of $\mO_{\P^{k-1}}(k) \to \P^{k-1}$.

  We apply the wall-crossing formula to the Euler class of the
  quotients.   Let $\alpha = c_k^G( \C^{k+1})$
  (the next-to-highest Chern class of $\C^{k+1}$) so that
  $\kappa_{X,\pm}^G(\alpha)$ is the top Chern class of $X \qu_\pm G$.
  Then
\[ \tau_{X \qu_+ G}(\kappa_{X,+}^G(\alpha)) = \int_{[\C^k/\Z_k]} \Eul(\C^k/\Z_k) = 1/k .\]
Indeed, interpreting this integral via localization using the
$\C^\times$-action given by scalar multiplication, there is a unique
fixed point with stabilizer of order $k$ which contributes $
\Eul(\C^k)/ (k\Eul(\C^k) ) = 1/k$ to the integral.  On the other side of
the wall,
\[ \tau_{X \qu_- G}(\kappa_{X,-}^G(\alpha)) = \int_{[\mO_{\P^{k-1}}(k) \to
      \P^{k-1}]} \Eul(\P^{k-1}) = k .\]
The wall-crossing term is 
\[ \Res_\xi \frac{ 
 (1 + \xi)^k(1 - k\xi)}{\xi^k (-k \xi)} = k - 1/k \]
as expected. 
\item {\rm (A del Pezzo surface as a quotient of a product of
    projective lines by a non-abelian group) } \label{sl2} Let
  $X = (\P)^{\times 5}$ with the diagonal action of $G =SL(2,\C)$. The
  action of $SL(2,\C)$ on $\P$ lifts uniquely to an action on the
  hyperplane bundle $\mO(1)$ and we take as polarization the exterior
  tensor product $\mO(1)^t \boxtimes \mO(1)^{\boxtimes 4}$ varying on
  the first factor only. The stability condition for a tuple
  $(v_1,\ldots,v_5) \in X$ is
\[\sum_{v_i \neq w} \mu_i v_i \leq \sum_{v_i =w}\mu_i v_i, \quad 
\mu_1= t, \mu_i= 1,i\neq 1, \quad \forall w \in \P \] 
The values of $t$ for which there are semistable points with infinite
stabilizers are $t = 2,4$; these points are given by 
\[v_1=v_i \neq v_j= v_k= v_l, \quad v_1 \neq v_2= v_3 = v_4= v_5.\]  
Thus the chamber structure for $t$ is $(0,2),(2,4),(4,\infty)$. Since
the quotient in the last chamber is empty, and all the weights are
one, the quotient $X \qu_t G$ for the second chamber must be $\P^2$.
Passing from the second to first chamber the quotient $X \qu_t G$
undergoes a blow-up at $4$ points.  Consider the wall-crossing given
by passing between chambers so that
\[ X \qu_- G = \P^2, \quad X \qu_+ G = \Bl^4(\P^2) . \]  
We compute the square of the first Chern class by wall-crossing: Let
\[\alpha =
  c_1^G(X)^2, \quad \text{so that} \quad \kappa_{X,\pm}^G(\alpha) = c_1(X \qu_\pm
  G)^2 .\]
  The fixed points of $\C^\times$ action for the singular value
  $t = 2$ are the $4$ (up to the action of the disconnected group
  $G_\zeta= N(T)$) configurations with $v_1= v_i \neq v_j= v_k= v_l$.
  Consider the case that $x = (x_1,\ldots,x_5) \in X$ is fixed by the
  maximal torus action; the tangent bundle of this point has weights
  $1,1,1,-1,-1$ hence $c_1^G(X) |_x = \xi$.  The tangent bundle
  $T_x X$ modulo $\g/\g_\zeta$ has weights $1,1,-1$.  Thus
\begin{eqnarray*}
 \int_{[X \qu_+ G]} c_1(X \qu_+ G)^2 &=& \int_{[X \qu_- G]} c_1(X \qu_-
G)^2 + 4 \Resid_\xi \frac{\xi^2}{-\xi^3} \\
&=& 9 -4 = 5 .\end{eqnarray*}
This again matches the fact that each blow-up of $\P^2$ lowers $c_1^2$
by $1$. \end{enumerate} \end{example}

\subsection{The virtual wall-crossing formula} 
\label{vsec} 

Our main result will be derived from a virtual extension of
wall-crossing formulas, similar to Kiem-Li \cite{kl:vwc} using the
virtual localization formula of Graber-Pandharipande \cite{gr:loc}; we
assume that the reader is familiar with the concepts of equivariant
perfect obstruction theories etc.  from those papers.  In this section
we explain the virtual extension and give an application to a simple
complete intersection.

We first need a simple lemma on obstruction theories on quotients. Let
$\XX$ be a Deligne-Mumford $G$-stack equipped with a $G$-equivariant
perfect obstruction theory which admits a global resolution by vector
bundles. This means that $\XX$ is equipped with an object $E$ of the
bounded derived category of $G$-equivariant coherent sheaves, together
with a morphism from $E$ to the cotangent complex, satisfying certain
axioms, see \cite{bf:in}, \cite{gr:loc}; a typical example is an
invariant complete intersection as in Example \ref{nodal} below. Let
$\ul{\g}^\dual$ denote the sheaf of sections of the trivial bundle
with fiber $\g^\dual$. \label{csheaf} \label{sheaf} The $G$-action on
$\XX$ induces a canonical morphism
$a^\dual:L_{\XX} \to \ul{\g}^\dual $ that we call the {\em
  infinitesimal action}.  The obstruction complex $E$ comes equipped
with a lift $\ti{a}^\dual: E \to \ul{\g}^\dual$ of the infinitesimal
action $a^\dual: L_{\XX} \to \ul{\g}^\dual$.  If $\XX^{\ss}$ is the
semistable locus for some polarization and stable=semistable, denote
$E^{\ss}, L_{\XX^{\ss}}$ etc.  the restrictions to the semistable
locus.  The following lemma is probably well-known:

\begin{lemma} \label{qot}If $\XX^{\ss}$ is the semistable locus for some
polarization and stable=semistable, then the perfect obstruction
theory $E^{\ss} \to L_{\XX^{\ss}} := L_{\XX} | \XX^{\ss} $ descends to
a perfect obstruction theory on the quotient $\XX^{\ss} / G$.
\end{lemma}

\begin{proof}  
 From the fibration $\pi: \XX^{\ss} \to \XX^{\ss}/G$ one obtains an exact
 triangle of cotangent complexes
 \begin{equation} \label{etg} \pi^* L_{\XX^{\ss}/ G} \to L_{\XX^{\ss}}
   \to \ul{\g}^\dual\to (L_{\XX^{\ss}/G})[1] .\end{equation}
Let $\Cone(\ti{a}^\dual)$ denote the mapping cone on the lift of the
infinitesimal action $\ti{a}^\dual$.  The exact triangle
\[ \Cone(\ti{a}^\dual)
\to E \to \ul{\g}^\dual \to \Cone(\ti{a}^\dual)[1]\]
admits a morphism to \eqref{etg}, in particular making
$\Cone(\ti{a}^\dual) \to L_{\XX^{\ss}/G}$ into an obstruction theory
with support in $[-1,1]$.  By the assumption on the stabilizers, this
obstruction theory is perfect.
\end{proof} 

We now study the virtual normal complexes of the fixed point stacks.
If $\zeta \in \g$ \label{oneps} an element generating a one-parameter subgroup then
the fixed point stacks $\XX^{\zeta}$ also have equivariant perfect
obstruction theories compatible with that on $\XX$. We choose a
splitting of Lie algebras
\[  \g_\zeta \cong \C\zeta \oplus (\g_\zeta/ \C\zeta) \] 
inducing a splitting of Lie groups after passing to a finite cover
$\ti{G}_\zeta \to G_\zeta$
\[ \ti{G}_\zeta \cong G_\zeta/\C_\zeta^\times \times \C^\times_\zeta
.\]
Denote by $\nu_{\XX^\zeta}$ \label{nuzeta} the conormal complex for the embedding
\[ \XX^\zeta/(G_\zeta/\C^\times_\zeta) \to
\XX/(G_\zeta/\C^\times_\zeta) .\]
Denote by $\XX^{\zeta,t}$ the locus of $\XX^\zeta$ semistable for
$\LL_t$ and by $\nu_{\XX^{\zeta,t}}$ the restriction of
$\nu_{\XX^{\zeta}}$ to $\XX^{\zeta,t}$. The Euler class
$ \Eul_{\C^\times_\zeta}(\nu_{\XX^{\zeta,t}})$ is well-defined in the
equivariant cohomology $H(\XX^{\zeta,t})[\xi,\xi^{-1}]$, after
inverting the equivariant parameter $\xi$. From the virtual
localization formula in Chang-Kiem-Li \cite[Theorem 3.4]{ckl:loc},
improving Graber-Pandharipande \cite{gr:loc}, one obtains the
following virtual version of the wall-crossing Theorem \ref{kalk2},
similar to results of Kiem-Li \cite{kl:vwc}:

\begin{theorem} [Virtual Kalkman wall-crossing]
  \label{vkalkman}
  Let $\XX$ be a proper Deligne-Mumford $G$-stack equipped with a
  $G$-equivariant perfect obstruction theory which admits a global
  resolution by vector bundles. Let $\LL_\pm \to \XX$ be $G$-line
  bundles that are ample for the coarse moduli spaces, so that
  stable=semistable for $\P(\LL_- \oplus \LL_+)$. Let $\tau_{\XX
  \qu_\pm G}$ resp. $\tau_{\XX,\zeta,t}$ denote integration
resp. equivariant integration over $\XX \qu_\pm G$
resp. $\XX^{\zeta,t}$ times
$\Eul_{\C^\times_\zeta}(\nu_{\XX^{\zeta,t}})^{-1}$.  Then
\begin{equation} 
\tau_{\XX \qu_+ G} \ \kappa_{\XX,+}^G - \tau_{\XX \qu_- G} \
\kappa_{\XX,-}^G = 
\sum_{t \in (-1,1),[\zeta]}
\Resid_\xi \tau_{\XX,\zeta,t} \end{equation}
where the sum is over $[\zeta]$ from \eqref{equiv}. 
\end{theorem}  

\begin{example} \label{nodal} {\rm (Wall-crossing over a nodal fixed
    point)} The following simple example may help illustrate the
  notation. Suppose that $X = \P^1 \cup_{\infty \sim \infty} \P^1$ is
  a nodal projective line with a single node $\infty$, equipped with
  the standard $\C^\times$-action on each component, so that the
  weights of the action on the tangent spaces at the node $\infty$ are
  $\pm 1$.  We equip $X$ with a polarization so that the weights are
  $\pm 1$ at the smooth fixed points $0 \in \P^1$, and $0$ at the
  nodal point. Then $X \qu_t G$ is a point for $t \in (-1,1)$, and is
  singular for $t=0$. Since $X$ is a complete intersection, $X$ has
  a perfect obstruction theory \cite[Example before Remark 5.4]{bf:in}
  and the virtual wall-crossing formula of Theorem \ref{vkalkman}
  applies. We examine the wall-crossing for the trivial class
  $\alpha = 1$ at the singular value $ t= 0$. The virtual normal
  complex at the nodal point is the quotient of $\C_1 \oplus \C_{-1}$,
  the sum of one-dimensional representations with weights $1,-1$,
  modulo their tensor product $\C_1 \otimes \C_{-1}$, which has weight
  zero. Hence the normal complex has inverted Euler class
\[ \Eul_G( \nu_{X,G,t})^{-1} = \frac{0 \xi}{ \xi(-\xi)} = 0 .\]
The integrals on the left and right hand sides are $1$ (being the
integrals over points) while the wall-crossing term is
\begin{eqnarray*} 
1 - 1 &=& \tau_{\XX \qu_+ G} \ \kappa_{\XX,+}^G - \tau_{\XX \qu_- G}
\ \kappa_{\XX,-}^G \\ &=& \Resid_\xi \tau_{\XX,\zeta,0}^{G} =
\Resid_\xi \int_{[ \on{pt}]} \Eul_G( \nu_{X,G,t})^{-1} = \Resid_\xi 0
= 0 \end{eqnarray*}
as desired.
\end{example} 

We begin the proof of Theorem \ref{vkalkman} by construction of a
master space.

\begin{lemma} \label{vmaster} Let $\XX$
  be a Deligne-Mumford $G$-stack equipped with a $G$-equivariant
  perfect obstruction theory which admits a global resolution by
  vector bundles as well as an embedding in a smooth Deligne-Mumford
  $G$-stack.  Let $\LL_\pm \to \XX$ be polarizations ($G$-line bundles
  with ample coarse moduli spaces) such that stable=semistable for
  $\LL_\pm$ and for any $t \in (-1,1)$ and any $t$-semistable point $x
  \in \XX^G$, $G_x$ acts with finite stabilizer on the fiber
  $(\LL_+\otimes \LL_-^{-1})_x$.  There exists a proper
  Deligne-Mumford $\C^\times$-stack $\ti{\XX}$ equipped with a line bundle
  ample for the coarse moduli space whose git quotients $ \ti{\XX}
  \qu_t \C^\times$ are isomorphic to those $ \XX \qu_t G$ of $\XX$ by the
  action of $G$ with respect to the polarization $\LL_t$ and whose fixed
  point set $\ti{\XX}^{\C^\times}$ is given by the union
\[\ti{\XX}^{\C^\times} = (\XX \qu_- G) \cup (\XX \qu_+ G) \cup
\bigcup_{[\zeta]} \left( \iota_\zeta( \XX^\zeta \qu_t
  (G_\zeta/\C^\times_\zeta)) \right)\]
where $[\zeta]$ is as in \eqref{equiv} and $\iota_\zeta$ is the
natural map to $\ti{X}$ as in Lemma \ref{fplem}. Furthermore,
$\ti{\XX}$ has a perfect obstruction theory admitting a global
resolution by vector bundles with the property that the virtual normal
complex of $\XX^\zeta \qu_t (G_\zeta/\C^\times_\zeta)$ is isomorphic
to the image of $\nu_{\XX^\zeta} / (\g/\g_\zeta)$ under the quotient
map $\XX^\zeta \to \XX^\zeta \qu (G_\zeta / \C^\times_\zeta)$, by an
isomorphism that intertwines the action of $\C^\times_\zeta$ on
$(\nu_{\XX^\zeta} / (\g/\g_\zeta) ) \qu (G_\zeta / \C^\times_\zeta)$
with the action of $\C^\times$ on $\nu_{\ti{\XX}^{\C^\times}}$.
\end{lemma}

\begin{proof} The construction of the master space is the same as in
  \ref{master}, that is, the master space is the stack-theoretic
  quotient $ \ti{\XX}=\P(\LL_- \oplus \LL_+) \qu G .$ The assumption
  on the action of the
  stabilizers implies that the action of $G$ on the semistable locus
  in $\P(\LL_- \oplus \LL_+)$ is locally free, so that
  stable=semistable for $\P(\LL_- \oplus \LL_+)$. It follows that
  $\ti{\XX}$ is a proper Deligne-Mumford stack, and by Lemma \ref{qot}
  has a perfect obstruction theory induced from the natural
  obstruction theory on
  $\P(\LL_- \oplus
  \LL_+)$
  given by considering it as a bundle over $\XX$. The quotient
  $\ti{\XX}$ contains
  the quotients of $\P(\LL_\pm) \cong \XX$ with respect to the
  polarizations $\LL_\pm$, that is, $\XX \qu_\pm G$. 

  The same argument
  in Lemma \ref{fplem} describes the fixed point loci: they correspond
  to fixed point loci in $\P(\LL_- \oplus \LL_+)$ for one-parameter
  subgroups of $\C^\times \times G$. Given such a locus
  $\P(\LL_-\oplus \LL_+)^{\xi -\zeta}$, the pull-back of the virtual normal complex is by
definition the moving part of $\Cone(\ti{a}_{\P(\LL_-\oplus
  \LL_+)}^\dual)$, where 
\[\ti{a}^\dual_{\P(\LL_-\oplus \LL_+)}: E_{\P(\LL_-
  \oplus \LL_+)} \to \ul{\g}^\dual_\zeta\]
is the lift of the infinitesimal action of $G_\zeta$.  Consider the
fibration $\pi: \P(\LL_- \oplus \LL_+) \to \XX$.  By definition $E_{\P(\LL_-
  \oplus \LL_+)}$ fits into an exact triangle
\[   E_\XX \to E_{\P(\LL_- \oplus \LL_+)} \to L_\pi \to E_\XX[1] .\]
Over the complement
$\P(\LL_- \oplus \LL_+)^\times \subset \P(\LL_- \oplus \LL_+)$ of the
sections at zero and infinity we may identify $L_\pi \cong \ul{\C}$
using the $\C^\times$-action on the fibers, by the assumption on the
weights of the $\C_\zeta$ action on the fiber. The projection to $\XX$
identifies the mapping cones
\[\Cone(\ti{a}^\dual_{\P(\LL_-\oplus \LL_+)} |_{\P(\LL_- \oplus \LL_+)^{\xi - \zeta}}) \to \pi^*
  \Cone(\ti{a}^\dual_{\XX^\zeta})\]
where 
\[\ti{a}^\dual_{\XX^\zeta}: E_{\XX} | \XX^\zeta \to (
\ul{\g_\zeta/\C\zeta} )^\dual\]
is the lift of the infinitesimal action of $\g_\zeta/\C\zeta$.  Now
the virtual normal complex is by definition the $\C^\times$-moving part
of the perfect obstruction theory; the Lemma follows.
\end{proof} 

\begin{proof}[Proof of Theorem \ref{vkalkman}] The proof of
  \ref{vkalkman} is similar to that of Theorem \ref{kalk2}. Namely we
  take the residue of the virtual localization formula applied to
  $\ti{X}$: For
  any equivariant class $\alpha \in H_G(\XX)$ of top degree, its
  pullback to $\P(\LL_- \oplus \LL_+)$ descends to a class $\ti{\alpha}
  \in H_{G}(\ti{\XX})$ whose restriction to $\XX \qu_\pm G$ is
  $\kappa_{X,\pm}^G(\alpha)$, and whose pullback under $X^{\zeta,t}
  \to \ti{X}^{\C^\times}$ is $\iota_{X^{\zeta,t}} \alpha$. By virtual
  localization the integral is
\[ \int_{[\ti{\XX}]} \alpha = \int_{[\XX \qu_- G]}
  \frac{\kappa_{\XX,-}^G(\alpha)}{\Eul_G(\nu_-)} + \int_{[\XX \qu_+
      G]} \frac{\kappa_{\XX,+}^G(\alpha)}{\Eul_G(\nu_+)} + \sum_{t \in
      (-1,1), [\zeta]} \int_{[\XX^{\zeta,t}/(G_\zeta/\C^\times)]}
    \frac{\iota^*_{{\XX}^{\zeta,t}}
    \alpha}{\Eul_{G_\zeta}(\nu_{\XX^{\zeta,t}})} .\]
Taking residues and using Lemma \ref{vmaster} to identify the last
term with
\[ 
\sum_{t \in (-1,1)} \Resid_\xi \int_{[\ti{\XX}^{G,t}]} \frac{\iota^*_{\ti{\XX}^{G,t}}
  \alpha}{\Eul_G(\nu_{\ti{\XX}^{G,t}})} = \sum_{t \in
  (-1,1),[\zeta]} \Resid_\xi \tau_{\XX,\zeta,t} \]
gives the formula in the Theorem.
\end{proof}

\section{Wall-crossing for Gromov-Witten invariants} 

In this section, we prove a quantum generalization of Kalkman's
wall-crossing formula Theorem \ref{kalk2}.  In the first two
subsections, we define the wall-crossing terms as integrals over
moduli spaces of gauged maps fixed by a central subgroup.  The last
two subsections contain a construction of a master space for moduli
spaces of gauged maps, and a proof of the wall-crossing formula via
virtual localization on the master space.  The construction of the
master space is obtained from one for a different compactification of
gauged maps introduced by Schmitt \cite{schmitt:git}, pulled back
under a relative version of Givental's morphism from stable maps to
the quot scheme.  Schmitt's compactification has the advantage that it
is constructed by git methods so the classical techniques apply.

\subsection{Construction of a master space}
\label{master2}

The proof of the wall-crossing formula \ref{gwall} depends on the
construction of {\em master space} in the sense of \ref{master} whose
quotients are the moduli spaces of Mundet stable gauged maps.

\begin{proposition} [Existence of a master space]  
\label{masterprop}
Under suitable stable=semistable conditions, there exists a proper
Deligne-Mumford $\C^\times$-stack $\ol{\M}_n^G(C,X,\LL_-,\LL_+)$ with
the following properties:
\begin{enumerate}
\item \label{eqobs} 
$\ol{\M}_n^G(C,X,\LL_-,\LL_+)$ admits a perfect
  $\C^\times$-equivariant relative
  obstruction theory
\item \label{gitquot} the git quotients of $\ol{\M}_n^G(C,X,\LL_-,\LL_+)$ are the
  moduli stacks
  \[ \ol{\M}_n^G(C,X,\LL_-,\LL_+) \qu_t \C^\times \cong
  \ol{\M}_n^G(C,X, \LL_t) \]
  for parameter $t \in (-1,1)$;
\item \label{inc} 
the $\C^\times$-fixed substack includes
  $\ol{\M}_n^G(C,X,\LL_-,d)$ and $\ol{\M}_n^G(C,X,\LL_+,d)$; 
\item  \label{embeds} $\ol{\M}_n^G(C,X,\LL_-,\LL_+,d)$ admits an embedding in a
  non-singular Deligne-Mumford stack.
\end{enumerate} 
\end{proposition} 

The proof will be given after several constructions.  First recall the
quot-scheme compactification of Mundet semistable morphisms from $C$
to $X/G$ constructed by Schmitt \cite[Theorem 2.7.1.4]{schmitt:git}:

\begin{definition} \label{bmaps}
 {\rm (Bundles with maps)} 
\begin{enumerate}
\item Let $X = \P^{r-1}$ and $G= GL(r)$. A {\em projective bundle with
    map} over a smooth projective curve $C$ over a point
  $S =\{ \on{pt} \}$ is a datum $(E,L,\varphi)$ consisting of a vector
  bundle $E \to C$ of rank $r$; a line bundle $L \to C$; and a
  surjective morphism $\varphi: E \to L$; to obtain a compactification
  one allows this morphism to be non-zero rather than surjective.  On
  the locus $\varphi \neq 0$ we obtain a section of the associated
  projective bundle $\Fr(E^\dual) \times_G X = \P(E^\dual)$,
\begin{equation} \label{defines} \{ \varphi \neq 0 \} \to \Fr(E^\dual)
  \times_G X, \quad z \mapsto \on{im} \varphi^\dual_z .\end{equation}
For more general schemes $S$, a projective bundle with map over a
curve $C \times S \to S$ of degree $l$ is a datum $(E,L,\tau,\varphi)$
where $\tau$ is a morphism from $S$ to the Jacobian $\Jac^l(C)$ of
degree $l$ line bundles, $N(\tau)$ the corresponding line bundle on
$C$ defined by pulling back a Poincar\'e bundle over
$C \times \Jac^l(C)$, and $\varphi: E \to N(\tau) \otimes L$ is a
non-zero map.  For $G$ a product of groups $GL(r_i), k =i,\ldots, k$,
a projective bundle with map is collection of bundles $E_i$ of rank
$r_i$, a line bundle $L \to C$ and a surjective morphism
$\varphi: \oplus_{i=1}^k E_i \to L$.
\item For a reductive subgroup $G$ of a product
  $GL(\ul{r}) = \times_{i=1}^k GL(r_i)$ let $\rho: G \to GL(\ul{r})$
  be a faithful homogeneous representation, that is, so that the
  central subgroup of $G$ maps to the center $(\C^\times)^k$ of
  $GL(\ul{r})$.
  A {\em projective bundle} with map is a projective 
  $GL(\ul{r})$-bundle with map $(E,L,\varphi)$ and a reduction 
  $\sigma: \Fr(E) \to \Fr(E)/G$ of the frame bundle to $G$.
\item A $G$-bundle with map $(E,L,\varphi,\sigma)$ is {\em Mundet
    semistable} if it satisfies the inequality of Definition
  \ref{mundetsemistable} for every pair $(\sigma,\lambda)$ consisting
  of a parabolic reduction $\sigma$ and antidominant coweight
  $\lambda$.
\end{enumerate} 
\end{definition}

We introduce the following notation the moduli stacks of bundles with
maps with given numerical invariants.

\begin{definition} \label{noresp}  Given a $G$-module $V$ and
  integers $d_E,d_L$ let $\M^{G,\quot}(C,V,d_E,d_L)$ denote the stack
  of Mundet semistable $G$-bundles with maps
  \[ (P, L \subset E:= P(V)^\dual, \varphi :E \to L
  ) \] 
  whose bundles $(E,L)$ have first Chern classes 
\[ c_1(E) = d_E, \quad c_1(L) = d_L \in H_2(C) \cong \Z .\]
\end{definition}

Schmitt \cite[Theorem 2.7.1.4]{schmitt:git} proves using a git
construction that $\ol{\M}^{G,\quot}(C,V,d_E,d_L)$ has projective
coarse moduli space.  We will need some details regarding the git
construction which involves {\em level structures}, defined as
follows.

\begin{definition} \label{levstr} Let $n \ge 1$ be an
  integer. An \label{an} {\em $n$-twisted level structure} for a
  projective bundle with map $(E,L,\varphi)$ is a collection of
  sections $s_1,\ldots,s_l$ generating $E$: a surjective map
  \[ s: \mO_C^{\oplus l}(-n) \to E .\]
  An isomorphism between two projective bundles with maps and level
  structures $(E^k,L^k,\varphi^k,s^k), k \in \{ 1,2 \}$ is a pair of
  isomorphisms $E^1 \to E^2, L^1 \to L^2$ intertwining the maps
  $\varphi_k$ and level structures $s_k$.  In the case of
  $G \subset GL(\ul{r})$ bundles, a level structure is a level
  structure for each factor $E_i \subset E, i = 1,\ldots, k$.
\end{definition}

Denote by $\ol{\M}^{G,\quot,\lev}(C,V,d_E,d_L)$ the compactified stack
of projective bundles with maps and level structures. \label{levels}
The group $GL(\ul{r})$ acts on
$\ol{\M}^{G,\quot,\lev}(C,\P^{r-1},d_E,d_L)$ by changing the level
structure:
\[ g ( E, L, \varphi, s) \mapsto (E,L,\varphi , gs) .\]
Schmitt \cite[Section 2.7]{schmitt:git} constructs a line bundle
$D(\LL)$ ( the pull-back of an ample line bundle on a suitable quot
scheme by a finite morphism) such that the quotient of the inverse
image of the semistable locus in
$\ol{\M}^{G,\quot,\lev}(C,\P^{r-1},d_E,d_L)$ by $GL(\ul{r})$ is
$\ol{\M}^{G,\quot}(C,\P^{r-1},d_E,d_L)$. (The notation stands roughly
speaking for determinant line bundle.)

A well-known construction of Givental \cite{gi:eq} provides a morphism
from the Kontsevich-style compactification to the quot-scheme
compactification.  Let $X$ be a smooth projective $G$-variety.  Given
$d \in H_2^G(X,\Z)$ let $d_E \in \Z$ denote the image of $d$ under
\[ H_2^G(X,\Z) \to H(BG,\Z) \to H(BGL(r),\Z) \cong \Z \] 
and $d_L \in \Z$ the image of $d$ under 
\[ H_2^G(X,\Z) \cong H_2(BG) \oplus H_2(X) \cong \Z . \]

\begin{lemma}  There is a proper morphism of Artin stacks
  \begin{equation} \label{givmorph} 
\ol{\M}^G(C,X,\LL,d) \to 
    \ol{\M}^{G,\quot}(C,\P^{r-1},\LL,d_E,d_L) \end{equation}
  which maps 
  \[ u: \hat{C} \to E: = \P(V^\dual) \]
  with principal component $u_0: C \to \P(V^\dual)$ to the pair
  $(L,\varphi)$ where the line bundle $L$ is defined by
  \begin{equation} \label{some} L := (u_0)^* E \otimes
    \bigotimes_{i=1}^k \mO( - d_i p_i) \end{equation}
  where $d_i$ is the degree of the $i$-th bubble component
\[ u|C_i: C_i \to \P(E^\dual), i =1 ,\ldots ,k  \]  
and $p_i \in C$ is its projection $\pi( u(C_i))$ onto the principal
component $C$; and if $\varphi_0$ is the quotient corresponding to
$u_0$ then $\varphi$ is defined by tensoring with a section of
$\mO( - d_i p_i)$, so that in a local coordinate $z$
\begin{equation} \label{somes} \varphi(z) := \varphi_0(z) (z -
  p_i)^{-d_i} .\end{equation}
\end{lemma}

\begin{proof} In the setting of families of stable maps this is an
  application of Popa-Roth \cite[Theorem 7.1]{po:stable}, see also
  Marian-Oprea-Pandharipande \cite[Section 5.2]{marian} for a similar
  construction.  The stack $\ol{\M}^{G,\lev}(C,X,\LL,d)$ \label{nopol}
  admits a universal bundle
\[E \to C \times \ol{\M}^{G,\lev}(C,X,\LL,d) .\] 
Letting 
\[ \ol{\cC}^{G,\lev}(C,X,\LL,d) \to \ol{\M}^{G,\lev}(C,X,\LL,d) \] 
denote the universal curve, we have a universal map \label{univmap}
\[ \ol{\cC}^{G,\lev}(C,X,\LL,d) \to \P(E), \quad ( u: \hat{C} \to C
\times X/G , z \in \hat{C},s) \mapsto u(z)
 \]
 where $s$ denotes the level structure from above.  The morphism in
 \cite[Theorem 7.1]{po:stable} maps this datum to the morphism
 $E^\dual \to L$ for the line bundle
 $L \to C \times \ol{\M}^{G,\lev}(C,X,\LL,d)$ defined by \eqref{some};
 at least locally.  Since the construction in \cite[Theorem
 7.1]{po:stable} is functorial, the local constructions patch together
 to the required morphism, even though $ \ol{\M}^{G,\lev}(C,X,\LL,d)$
 is {\em a priori} a stack of possibly infinite type.  The map
 \eqref{somes} giving a projective bundle with map and level structure
 in the sense of Schmitt \cite[Section 2.7]{schmitt:git}.  Taking the
 quotient by the action of $GL(\ul{r})$ gives the desired morphism
 $\ol{\M}^G(C,X,\LL,d) \to \ol{\M}^{G,\quot}(C,\P^{r-1},\LL,d_E,d_L)$.
\end{proof} 

\begin{corollary}
$\ol{\M}^G_n(C,X,d)$ is proper for each $d \in H_2^G(X,\Z)$.
\end{corollary} 

\begin{proof} Forgetting the markings and stabilizing the stable
  section defines a morphism
\[ \ol{\M}^G_n(C,X,d) \to \ol{\M}^G(C,X,d); \]
this is the composition of morphisms forgetting a single marking, each 
of which is isomorphic to a universal marked curve.  Combining these 
we obtain a proper morphism 
\[ \ol{\M}^G_n(C,X,d) \to \ol{\M}^{G,\quot}(C,\P^{r-1},d_E,d_L) .\]
Since the composition of proper morphisms is proper,
$\ol{\M}^G_n(C,X,d)$ is proper.
\end{proof}

A master space for stable gauged maps is constructed by considering
bundles with {\em pairs} of maps.  Associated to each map is a
determinant line bundle, and a repeat of the construction in Thaddeus
\cite{th:fl} will create a master space for the variation of stability
condition.  First we introduce a suitable moduli space of bundles with
pairs of maps.

\begin{definition} { \rm (Stack of bundles with pairs of maps)}
\begin{enumerate}
\item Suppose $\LL_\pm \to X$ are polarizations. Given tuples
  $\ul{r}_-,\ul{r}_+ > 0$ a class $d_G \in H^2(BG)$ and integers
  $ d_{\ul{L}} = (d_{L_-},d_{L_+}) $ and $G$-modules $V_-,V_+$ a {\em
    bundle with pair} is a tuple
\[   \Set{  (P, L_-,L_+, \varphi_-,\varphi_+) | \left( \begin{array}{l}  \varphi_- : P(V_-^\dual) =:
E_- \to L_- \\  \varphi_+ : P(V_+^\dual) =: E_+ \to L_+ \end{array} \right)   } \] 
consisting of a $G$-bundle $P$ with first Chern class $d_G$, line
bundles $L_-,L_+$ of degrees $d_{\ul{L}}$ and non-zero maps
$\varphi_-,\varphi_+$ from the associated vector bundles
$E_\pm:= P(V_\pm^\dual)$.
\item A stability condition on bundles with pairs is given by
  combining the Ramanathan and Hilbert-Mumford weights in
  \eqref{hweight}, \eqref{rweight}: For weights $\rho_-,\rho_+ > 0$ a
  parabolic reduction $\sigma$ and Lie algebra element $\lambda$
  generating a one-parameter subgroup we define
\[ \mu_{\rho_-,\rho_+}(\sigma,\lambda)=  \mu_R(\sigma,\lambda)+ \rho_-
\mu_{H,-}(\sigma,\lambda) + \rho_+ \mu_{H,+}(\sigma,\lambda)  \]
where $\mu_{H,\pm}(\sigma,\lambda)$ is the weight of the one-parameter
subgroup on the associated graded for the map $\varphi_\pm$.  A datum
$(P,L_-,L_+,\varphi_-,\varphi_+)$ is {\em semistable} iff
\[ \mu_{\rho_-,\rho_+} \leq 0 \quad \forall (\sigma,\lambda) .\]   
Let $\ol{\M}^{G,\quot}(C,V_-,V_+,d_G,d_{\ul{L}})$
denote the moduli stack of $(\rho_-,\rho_+)$-semistable data
$(P,L_-,L_+,\varphi_-,\varphi_+)$.
\end{enumerate} 
\end{definition} 

A very similar construction appears in \cite[Section
2.8.1]{schmitt:git} under the name of {\em twisted affine bumps}, but
with a different stability condition.

\begin{lemma} For sufficiently large twisting in Definition
  \ref{levstr}, there exists a projective
  $GL(\ul{r}_-) \times GL(\ul{r}_+)$-scheme
  $\ol{\M}^{G,\quot,\lev}(C,V_-,V_+,d_G,d_{\ul{L}})$ such that for any
  $\rho_+, \rho_-$, the stack
  $\ol{\M}^{G,\quot} (C,V_-,V_+,d_G,d_{\ul{L}})$ has
  coarse moduli space that is the good quotient of an open subset of
  semistable points $\ol{\M}^{G,\quot,\lev}(C,V_-,V_+,d_G,d_{\ul{L}})$
  \[ \ol{\M}^{G,\quot}(C,V_-,V_+,d_G,d_{\ul{L}}) =
  \ol{\M}^{G,\quot,\lev}(C,V_-,V_+,d_G, d_{\ul{L}})^{\ss} /
  GL(\ul{r}_-) \times GL(\ul{r}_+) . \]
Furthermore, the semistable locus is a git semistable locus 
in the sense that there exists
a
 finite injective equivariant morphism  
 \[ \ol{\M}^{G,\quot,\lev}(C,V_-,V_+,d_G,d_{\ul{L}}) \to
 \ol{\QQ}^{G,\lev}(C,V_-,V_+,d_G,d_{\ul{L}}) \]
 to a $GL(\ul{r}_-) \times GL(\ul{r}_+)$-scheme
 $ \ol{\QQ}^{G,\lev}(C,V_-,V_+,d_G,d_{\ul{L}})$ and a line bundle
\[ D(\LL_-,\LL_+) \to \ol{\QQ}^{G,\lev}(C,V_-,V_+,d_G,d_{\ul{L}}) \]
so that the following holds: A bundle with pair
$(P,L_-,L_+,\varphi_-,\varphi_+)$ is semistable iff its image in
$\ol{\QQ}^{G,\lev}(C,V_-,V_+,d_G,d_{\ul{L}}) $ is semistable, that is,
there exists a non-trivial invariant section of $D(\LL_-,\LL_+)$
non-vanishing at $(P,L_-,L_+,\varphi_-,\varphi_+)$.
\end{lemma} 

\begin{proof} We will embed the given moduli stack into a larger
  moduli stack via a tensor product construction.  Let
  $\ol{\M}^{G,\quot,\lev}(C,V_-,V_+,d_{\ul{L}})$ denote the moduli
  stack of objects that are tuples
  $(P, \varphi_-: E_- \to L_-, \varphi_+: E_+ \to L_+,s_+,s_-) $.
  where the bundles $E_-, E_+$ are equipped with level structures
  $s_+,s_-$ acted on by $GL(\ul{r})$.  The tensor product map
\[ (\varphi_-; : E_- \to L_-, \varphi_+: E_+ \to L_+) \mapsto
(\varphi_-^{\rho_-} \otimes \varphi_+^{\rho_+}: E_-^{\rho_-} \otimes E_+^{\rho_+}, L_-^{\rho_-} \otimes L_+^{\rho_+}) \] 
induces an embedding
  \begin{multline} \label{membed}
    \ol{\M}^{G,\lev}(C,V_-,V_+,d_{\ul{L}}) \to
    \ol{\M}^{G,\lev}(C,V_-^{\rho_-} \otimes V_+^{\rho_+},d_{L_-} d_{L_+})  , \\
    (P, \varphi_-,\varphi_+,s) \mapsto (P, \varphi_-^{\rho_-}
    \otimes \varphi_+^{\rho_+} ,s) .
\end{multline} 
Because the Hilbert-Mumford weights are additive under tensor products
(the weights on the hyperplane bundles are additive, by construction)
the morphism \eqref{membed} preserves the semistability conditions.
By the construction on \cite[p. 277]{schmitt:git}, there exists an
injective finite morphism 
\[\iota:  \ol{\M}^{G,\lev}(C,V_-^{\rho_-} \otimes V_+^{\rho_+},d_G,d_{L_-}
d_{L_+}) \to \ol{\QQ}^{G,\lev}(C,V_-^{\rho_-} \otimes
V_+^{\rho_+},d_G,d_{L_-} d_{L_+}) \]
to a projective scheme, denoted
$\ol{\QQ}^{G,\lev}(C,V_-^{\rho_-} \otimes V_+^{\rho_+},d_G,d_{L_-}
d_{L_+})$
and line bundle $D(\rho_-,\rho_+)$ on the codomain
$\ol{\QQ}^{G,\lev}(C,V_-^{\rho_-} \otimes V_+^{\rho_+},d_G,d_{L_-}
d_{L_+})$
so that a datum $(P,E,L,\varphi,s)$ is semistable iff
$\iota(P,E,L,\varphi,s)$ is git-stable with respect to
$D(\rho_-,\rho_+)$.  Let
\begin{eqnarray*} \ol{\QQ}^{G,\lev}(C,V_-, V_+,d_G,d_{\ul{L}}) &=& \iota(
\ol{\M}^{G,\lev}(C,V_-,V_+,d_G,d_{\ul{L}})) \\
&\subseteq& \ol{\QQ}^{G,\lev}(C,V_-^{\rho_-} \otimes V_+^{\rho_+},d_G,d_{L_-}
d_{L_+}) \end{eqnarray*}
and $D(\LL_-,\LL_+)$ the pull-back of $D(\LL)$.  Then
$\ol{\QQ}^{G,\lev}(C,V_-, V_+,d_G,d_{\ul{L}})$ is projective and the
morphism $\ol{\M}^{G,\lev}(C,V_-,V_+,d_G,d_{\ul{L}})$ to
$\ol{\QQ}^{G,\lev}(C,V_-, V_+,d_G,d_{\ul{L}})$ is finite and
injective, since it is the restriction of a finite and injective
morphism.  The claim on the semistable locus follows by restriction.
\end{proof}

The constructions above give moduli stacks of bundles with sections of
projectivizations.  We extend this to sections of associated bundles
with arbitrary smooth projective fibers as follows:

\begin{definition} {\rm (Moduli stack of bundles with pairs of maps)}  
  Given the projective $G$-variety $X$ and $G$-equivariant embeddings
  \[ \iota_\pm: X \to \P(V_\pm) \]
  let $\ol{\M}^{G,\quot}(C,X,V_-,V_+,d_G,\d_{\ul{L}})$ denote the
  substack of $\ol{\M}^{G,\quot}(C,V_-,V_+, d_G, d_{\ul{L}})$
  consisting of data
\[ ( P, \varphi_-: E_- \to L_-, \varphi_+: E_+ \to L_+)  \] 
so that
\[ ([\varphi_-(z)], [\varphi_+(z)]) \in (\iota_- \times \iota_+)(X)
\subset \P(V_-) \times \P(V_+) \]
for generic $z \in C$.  Including level structures for the bundles
$E_-,E_+$ into the data gives a $GL(\ul{r}_-) \times GL(\ul{r}_+)$-stack
\begin{eqnarray*} 
  \ol{\M}^{G,\quot,\lev}(C,X,V_-,V_+,d_G,\d_{\ul{L}}) &\subset & 
                                                                 \ol{\M}^{G,\quot,\lev}(C,V_-,V_+,d_G,d_{\ul{L}}) 
\end{eqnarray*}
 Let 
                                                $\ol{\QQ}^{G,\lev}(C,X,d_G,d_{\ul{L}})$
                                                denote its image in 
                                                $\ol{\QQ}^{G,\lev}(C,V_-,
                                                V_+,d_G,d_{\ul{L}})$.
                                                Denote the quotient
                                                stack 
                                                \[
                                                \ol{\MM}^{G,\quot}(C,V_-,V_+,d_G,d_{\ul{L}})
                                                =
                                                \ol{\M}^{G,\quot,\lev}(C,V_-,V_+,d_G,d_{\ul{L}})
                                                / (GL(\ul{r}_-) \times
                                                GL(\ul{r}_+)) .\]
\end{definition}

\begin{definition} {\rm (Master space for quot scheme compactifications)}  
  Let $\LL_\pm \to X$ denote the pull-backs of the hyperplane bundle
  under the embeddings $\iota_\pm$ and
  \[ D(\LL_\pm) \to\ol{\QQ}^{G,\lev}(C,X,d_G,d_{\ul{L}}) \]
denote the pull-backs of the line bundles in the Lemma above.  Then
\[ \P ( D(\LL_-) \oplus D(\LL_+)) \to
\ol{\QQ}^{G,\lev}(C,X,d_G,d_{\ul{L}}) \]
is a $\P^1$ bundle, equipped with a natural $GL(\ul{r})$-action and a
polarization
\[ \mO_{\P(D(\LL_-) \oplus D(\LL_+))}(1) \to \P ( D(\LL_-) \oplus
D(\LL_+)) \]  
considered in \eqref{relhyp}.  Let
\[ \ol{\QQ}^G(C,X,\LL_-,\LL_+,d_G,d_{\ul{L}}) = \P ( D(\LL_-) \oplus D(\LL_+))
\qu GL(\ul{r}) \]
denote its git quotient, that is, the quotient of the semistable locus
$\P ( D(\LL_-) \oplus D(\LL_+))^{\ss}$ of objects where some
non-trivial invariant section $\P ( D(\LL_-) \oplus D(\LL_+) ))$ is
non-vanishing, by the action of $GL(\ul{r})$.  The pull-back
\[ \iota^* \P(D(\LL_-) \oplus D(\LL_+) )  \to 
 \ol{\M}^{G,\quot,\lev}(C,X,d) \] 
 is a $\P^1$-bundle and admits a finite injective morphism to
 $\ol{\QQ}^G(C,X,\LL_-,\LL_+,d_G,d_{\ul{L}})$.  Since
 $\P ( D(\LL_-) \oplus D(\LL_+) )) \qu ( GL(\ul{r}_-) \times
 GL(\ul{r}_+)) $ is a good quotient, so is
\[ 
\ol{\M}^{\quot}(C,X,\LL_-,\LL_+,d_G,d_{\ul{L}}) := \iota^* \P(D(\LL_-)
\oplus D(\LL_+) ) \qu GL(\ul{r}_-) \times GL(\ul{r}_+) . \]
\end{definition} 

\begin{proposition} If stable=semistable then the stack
  $\ol{\M}^{G,\quot}(C,X,\LL_-,\LL_+,d_G,d_{\ul{L}})$ is a proper
  Deligne-Mumford stack with a projective coarse moduli space.  The
  group $\C^\times$ acts naturally on
  $\ol{\M}^{G,\quot}(C,X,\LL_-,\LL_+,d_G,d_{\ul{L}})$ and the quotient
  of the semistable locus for
\[ D(\LL_t) := D(\LL_-)^{(1-t)/2} \otimes D(\LL_+)^{((1+t)/2} \] 
(pulled back from $\ol{\QQ}^{G,\lev}(C,X,\LL_-,\LL_+,d_G,d_{\ul{L}})$)
is $\ol{\M}^{G,\quot}(C,X,\LL_t,d_G,d_{\ul{L}})$, the stack of data
\[ (P, L_-,L_+, \varphi_-,\varphi_+) \]  
that are semistable with respect to the stability condition
\[ (\rho_-,\rho_+) = ( (1-t)/2, (1+t)/2). \]
  \end{proposition}
 
  \begin{proof} The coarse moduli space of
    $\ol{\M}^{G,\quot}(C,X,\LL_-,\LL_+,d_G,d_{\ul{L}})$ admits a finite injective
    morphism to the projective variety
    $\ol{\QQ}^G(C,X,\LL_-,\LL_+,d_G,d_{\ul{L}})$.  It follows that the stack
    $\ol{\M}^{G,\quot}(C,X,\LL_-,\LL_+,d_G,d_{\ul{L}})$ has projective coarse
    moduli space. By \cite[Theorem 2.3.4.1]{schmitt:git} and
    \cite[Proposition 2.2.3.7]{schmitt:git} and \cite[Corollary
    2.2.3.4]{schmitt:git} after twisting by a suitable tensor power of
    a positive line bundle on $C$, every bundle with map
    $(P,L_-,L_+,\varphi_-,\varphi_+)$ that is semistable for
    $D(\LL_t)$ for some $t \in [-1,1]$ appears in this quotient
    construction, and so by construction
    $\ol{\M}^{G,\quot}(C,X,\LL_t,d_G,d_{\ul{L}})$ is the git quotient for the
    polarization $D(\LL_t) $.
\end{proof}

The Kontsevich style compactification of the master space is obtained
from similar compactifications applied to the master space for quot
scheme compactifications.  Let $ \ol{\MM}^{G}(C,X, V_-,V_+,d)$ denote
the substack of $ \ol{\MM}^{G}(C,X, d) $ consisting of bundles that
appear in $\ol{\MM}^{G,\quot}(C,V_-,V_+,d_G,d_{\ul{L}})$ together with
a stable section of the associated $X$ bundle $P(X) \to C$; as the
twisting in Definition \ref{levstr} goes to infinity the union of
these loci includes all of $\ol{\MM}^G(C,X,d)$.   The Givental morphism
\eqref{givmorph} for the pair of embeddings
$\iota_\pm: X \to \P(V_\pm)$ gives a morphism
\begin{equation} \label{givpairs} \ol{\MM}^{G}(C,X, V_-,V_+,d) \to
  \ol{\MM}^{G,\quot}(C,V_-,V_+,d_G,d_{\ul{L}})
\end{equation} 
which maps the pair $(C,u: \hat{C} \to P(X))$ to the line bundles
$L_\pm$ and maps $\varphi_\pm : P(V_\pm^\dual) \to L_\pm$ associated
to the morphisms $\iota_\pm \circ u: \hat{C} \to P(V_\pm)$.  For the
rest of the section we fix the degrees $d,d_G,d_{\ul{L}}$ and omit
them to simplify the notation.

\begin{definition} {\rm (Master space for stable gauged maps)}  Denote by $\ol{\M}^{G}(C,X,\LL_-,\LL_+\out{,d})$
  the fiber product of the  Givental morphism  \eqref{givpairs} 
  with the projection from
  the master space
\[\ol{\M}^{G}(C,X,\LL_-,\LL_+\out{,d}) = \ol{\MM}^G(C,X,V_-,V_+\out{,d}) 
\times_{ \ol{\MM}^{G,\quot}(C,V_-,V_+\out{,d_G,d_{\ul{L}}})}
\ol{\M}^{G,\quot}(C,\LL_-,\LL_+\out{,d_G,d_{\ul{L}}}) .\]
\end{definition} 

\begin{proof}[Proof of Proposition \ref{masterprop}]  
  Under the assumption that stable=semistable, $d \in H_+^G(X,\Z)$,
  $\ol{\M}_n^G(C,X,\LL_-,\LL_+\out{,d_G,d_{\ul{L}}})$ is a proper Deligne-Mumford stack.
  Indeed, the fiber product of proper morphisms is proper, and
  $ \ol{\M}^{G,\quot}(C,X,\LL_-,\LL_+\out{,d_G,d_{\ul{L}}}) $ is projective, and so its
  image in the $\ol{\M}^{G,\quot}(C,V_-,V_+\out{,d_G,d_{\ul{L}}})$ has proper coarse
  moduli spaces.  It follows that the coarse moduli space of
  $\ol{\M}^{G}(C,X,\LL_-,\LL_+\out{,d}) $ is proper, hence (assuming finite
  stabilizers) $\ol{\M}^{G}(C,X,\LL_-,\LL_+\out{,d}) $ is also proper.
  Consider the quotients by the circle action.  The git quotients are
  the fiber products
\begin{eqnarray*} 
  \ol{\M}^{G}(C,X,\LL_-,\LL_+\out{,d}) \qu_t \C^\times &=& \ol{\MM}^G(C,X,V_-,V_+\out{,d}) 
\times_{ \ol{\MM}^{G,\quot}(C,X,V_-,V_+\out{,d_G,d_{\ul{L}}}) }
 \left(  \ol{\M}^{G,\quot}(C,X,\LL_-,\LL_+\out{,d_G,d_{\ul{L}}}) \qu_t \C^\times
  \right) \\
&=&  \ol{\MM}^G(C,X,V_-,V_+\out{,d}) 
    \times_{\ol{\MM}^{G,\quot}(C,X,V_-,V_+\out{,d_G,d_{\ul{L}}})
    } \ol{\M}^{G,\quot}(C,X,\LL_t\out{,d_G,d_{\ul{L}}}) 
  \\
&=&  \ol{\M}^G(C,X,\LL_t\out{,d}) \end{eqnarray*}
which proves \eqref{gitquot}.  Item \eqref{inc} is similar. 

We wish to prove \eqref{eqobs}: Assuming stable=semistable as above,
$\ol{\M}_n^{G}(C,X,\LL_-,\LL_+\out{,d}) $ has an equivariant relative
obstruction theory over the moduli stack $\ol{\MM}_n(C)$ of prestable
maps to $C$ of class $[C]$.  The restriction of this obstruction
theory to $\ol{\M}_n^G(C,X,\LL_-,\LL_+\out{,d})$ is perfect and admits
a resolution by vector bundles. To prove this recall the construction
of the obstruction theory for $\ol{\MM}_n^G(C,X)$ from Remark
\ref{obstheoryrem}.  The complex in the relative obstruction theory is
denoted $E_{\ol{\M}_n^G(C,X,\LL_-,\LL_+)}$ in the diagram below,
constructed from the following commutative diagram of complexes of
coherent sheaves.  Let $L_\pi$ denote the cotangent complex relative
cotangent complex of
$\pi: \ol{\M}_n^G(C,X,\LL_-,\LL_+) \to \ol{\MM}_n^G(C,X)$. Consider
the diagram
\[
\begin{diagram} 
\node{L_\pi} \arrow{s} \node{E_{\ol{\M}_n^G(C,X,\LL_-,\LL_+)}}
  \arrow{w}\arrow{s} \node{\pi^* E_{\ol{\MM}_n^G(C,X)}} \arrow{w}
  \arrow{s} \node{L_\pi[1]} \arrow{w} \arrow{s}
  \\ \node{L_\pi} \node{L_{\ol{\M}_n^{G}(C,X,\LL_-,\LL_+)}}
  \arrow{w} \node{ \pi^* L_{\ol{\MM}_n^G({C},X)}} \arrow{w}
  \node{L_\pi[1]}\arrow{w} \end{diagram} \]
where the horizontal lines are exact triangles and the second vertical
arrow $\hat{\phi}$ exists by the third axiom in the definition of a
triangulated category.  The map $\hat{\phi}$ satisfies the axioms of
an obstruction theory by the five lemma applied to the cohomology of
the above diagram.  If the automorphisms are finite then the relative
obstruction theory is perfect: first cohomology
$H^1(E_{\ol{\M}_n^G(C,X,\LL_-,\LL_+)})$ is identified with the Lie
algebra of the group of automorphisms \cite[Theorem 1.5]{ol:def} which
vanish by assumption.  Existence of a resolution follows from the fact
that $\pi$ is a local complete intersection morphism, see
\cite[Appendix]{co:qrr}.

It remains to show item \eqref{embeds}, that for any
$d \in H_2^G(X,\Z)$, the moduli stack
$\ol{\M}_n^G(C,X,\LL_-,\LL_+\out{,d})$ admits an embedding in a
non-singular Deligne-Mumford stack. 
  Let
\[  \ol{\U}^{G,\quot,\lev}(C,X,\LL_-,\LL_+\out{,d_G,d_{\ul{L}}}) 
\to \ol{\M}^{G,\quot,\lev}(C,X,\LL_-,\LL_+\out{,d_G,d_{\ul{L}}})  \] 
denote the universal bundle over the moduli space of bundles with maps
and level structures.
$\ol{\M}^{G,\quot,\lev}(C,X,\LL_-,\LL_+\out{,d_G,d_{\ul{L}}})$,
equipped with the action of $GL(\ul{r}_-) \times GL(\ul{r}_+)$ by
changing the level structure.  Consider the embedding
\[\ol{\U}^{G,\quot,\lev}(C,X,\LL_-,\LL_+\out{,d_G,d_{\ul{L}}}) \times_G X \to \ol{\U}^{\quot,\lev}(C,X,\LL_-,\LL_+\out{,d_G,d_{\ul{L}}})
\times_{(GL(\ul{r}_-) \times GL(\ul{r}_+))} \P(V_- \oplus V_+).\]
The latter is projective (it is the pull-back of the universal bundle
on quot scheme) and so embeds
$(GL(\ul{r}_-) \times GL(\ul{r}_+))$-equivariantly in some $\P^N$.
Then $\ol{\M}_n^G(C,X,\LL_-,\LL_+\out{,d})$ is an embedded substack of
$\ol{\M}_{0,n}(\P^N)/(GL(\ul{r}_-) \times GL(\ul{r}_+))$, with objects
given by stable maps that are compositions of stable sections
\[C \to \ol{\U}^{G,\quot,\lev}(C,X,\LL_-,\LL_+\out{,d_G,d_{\ul{L}}}) \times_G X/(GL(\ul{r}_-) \times GL(\ul{r}_+))\]
with the inclusion into $\P^N/(GL(\ul{r}_-) \times GL(\ul{r}_+))$.
Since $\ol{\M}_{0,n}(\P^N)$ is a non-singular Deligne-Mumford stack,
the quotient $\ol{\M}_{0,n}(\P^N)/(GL(\ul{r}_-) \times GL(\ul{r}_+))$
is a non-singular Artin stack.  The group
$(GL(\ul{r}_-) \times GL(\ul{r}_+))$ acts locally freely on an open
subset $\ol{\M}_{0,n}(\P^N)^{\reg} \subset \ol{\M}_{0,n}(\P^N)$
containing $\ol{\M}_n^G(C,X,\LL_-,\LL_+\out{,d})$ by the
stable=semistable assumption and the quotient
$\ol{\M}_{0,n}(\P^N)^{\reg}/(GL(\ul{r}_-) \times GL(\ul{r}_+))$ is
Deligne-Mumford.
\end{proof} 

\subsection{Analysis of the fixed point contributions} 

In this subsection, we deduce the quantum wall-crossing formula by
applying virtual localization to the master space constructed in the
previous subsection.  By virtual Kalkaman Theorem \ref{vkalkman} for
the $\C^\times$ action on the master space
$ \ol{\M}_n^G(C,X,\LL_-,\LL_+,d)$ we obtain the following preliminary
version of the wall-crossing formula: For any fixed point component
$F \subset \ol{\M}_n^G(C,X,\LL_-,\LL_+,d)$ denote by $\nu_F$ the
normal complex, defined as the $\C^\times$-moving part of the perfect
obstruction theory in Proposition \ref{masterprop}.

\begin{proposition}
Suppose that $d$ is such that stable=semistable for  $\ol{\M}_n^G(C,X,\LL_-,\LL_+,d)$.
Then for any $\alpha \in H_G(X)^n$,
\begin{multline} 
  \int_{ [ \ol{\M}_n^G(C,X,\LL_+,d)] } \ev^* \alpha - \int_{ [
      \ol{\M}_n^G(C,X,\LL_-,d)] } \ev^* \alpha \\ = \Resid_\xi \left(
  \sum_{F} \int_{[F]} \iota_F^* \ev^* \alpha \cup
  \Eul_{\C^\times}(\nu_F)^{-1} \right)
\end{multline}
where $F$ ranges over the fixed point components of $\C^\times$ on
$\ol{\M}^G_n(C,X,\LL_-,\LL_+,d)$ not equal to $\ol{\M}^G_n(C,X,\LL_\pm,d)$.
\end{proposition} 

Next we describe the moduli spaces of circle-fixed gauged maps in
terms of gauged maps with smaller structure group. We begin with the
following remark on actions of central subgroups on the moduli stacks
of gauged maps.  To simplify notation we denote
$\ol{\M}^G_n(C,X) = \ol{\M}^G_n(C,X,\LL_t,d)$ the moduli stack of
$\LL$-semistable gauged maps of class $d \in H_2^G(X,\Z)$.

\begin{proposition} Let $Z \subset G$ a central subgroup.  The action
  of $Z$ on $X$ induces a natural action of $Z$ on
  $\ol{\M}^G_n(C,X)$.
\end{proposition}

\begin{proof} For any principal $G$-bundle $P \to C$, the right action
  of $Z$ on $P$ induces an action on the associated bundle $P(X)$, and
  so on the space of sections of $P(X)$.  The action of $Z$ on the
  space of sections of $P(X)$ preserves Mundet semistability (since
  the parabolic reductions are invariant under the action and the
  Mundet weights are preserved) and so induces an action of $Z$ on
  $\ol{\M}^G_n(C,X)$.
\end{proof}

The following is similar to the
description of fixed point sets in the case of stable maps in
Kontsevich \cite{ko:lo} and Graber-Pandharipande \cite[Section
4]{gr:loc}.

\begin{proposition}    Let $Z \subset G$ be a central subgroup. 
The fixed point locus for the action of $Z$ on $\ol{\M}_n^G(C,X)$ is
the substack whose objects are tuples 
\[(p: P \to C, u: \hat{C} \to P(X),\ul{z})\] 
such that
\begin{enumerate} 
\item $u$ takes values in $P(X^Z)$ on the principal component $C_0$;
\item for any bubble component $C_i \subset \hat{C}$ mapping to a
  point in $C$, $u$ maps $C_i$ to a one-dimensional orbit of $Z$ on
  $P(X)$; and
\item any node or marking of $\hat{C}$ maps to the fixed point set
  $P(X^Z)$.
\end{enumerate} 
\end{proposition} 

We will identify the following stacks with the fixed point sets in the
master space. 

\begin{definition}  \label{fps} {\rm (Fixed point stacks)}  
  For any $\zeta \in \g$ generating a one-parameter subgroup
  $\C^\times_\zeta \subset G$, recall that $G_\zeta$ denotes the
  centralizer of $\C^\times_\zeta$ and so contains $\C^\times_\zeta$
  as a central subgroup.  For each rational $\zeta \in \g$ let
\[\ol{\M}_n^{G_\zeta}(C,X,\LL_t,\zeta,d)
\subset \ol{\M}_n^{G_\zeta}(C,X,\LL_t,d) \]
denote the stack of $\LL_t$-Mundet-semistable morphisms from $C$ to
$X/G_\zeta$ that are $\C^\times_\zeta$-fixed and take values in
$X^{\zeta}$ on the principal component.
\end{definition} 

The Mundet semistability condition for fixed gauged maps simplifies
somewhat in the limit of large polarization, see \cite[Lemma
6.3]{reduc} for more details.  Let $X^{\zeta,t}$ denote the (possibly
empty) locus of $\LL_t$-semistable points in $X^\zeta$.

  \begin{lemma} \label{fixedlargearea} {\rm (Large-area limit of fixed
      gauged maps)} For any class $d$ there exists a $\rho_0$ such
    that for $\rho > \rho_0$, any Mundet-semistable fixed map for
    polarization $\LL_t$ must consist of a principal component mapping
    to $X^{\zeta,t}/G_\zeta$ and bubbles mapping to $X/G_\zeta$.
\end{lemma} 

\begin{proof} 
  Mundet semistability for $\LL^\rho$ implies that the Hilbert weight
  $\mu_H(\sigma,\lambda)$ is at most $\rho^{-1}$ times minus the
  Ramanathan weight $\mu_R(\sigma,\lambda)$, for any $\sigma,\lambda$.
  In particular, this holds $\lambda = - \zeta$ for $\zeta$
  antidominant and $\sigma$ the trivial parabolic reduction, in which
  case the Ramanathan weight for a pair $(P,u)$ with respect to
  $(\sigma,\lambda)$ is simply $- \lan c_1(P), \zeta \ran$ The latter
  is bounded by a constant $c(d) \Vert \zeta \Vert$ depending on
  $d \in H_2^G(X,\Z)$, since $c_1(P)$ is the projection of $d$ onto
  $H_2(BG)$.  So the Hilbert weight $\mu_M(\sigma,\lambda)$ is less
  than $c \Vert \lambda \Vert$ where $c := c(d) \rho^{-1} $.  Choose
  $\rho_0$ sufficiently large so that for any $\rho > \rho_0$, any
  point with Hilbert-Mumford weight $\mu_M(\sigma,\lambda)$ less than
  $c \Vert \lambda \Vert$ for all $\lambda$ is semistable; see for
  example \cite[Lemma 3.12]{ki:coh}.  Then any $\LL_t$-semistable pair
  $(P,u)$, the section $u$ takes values in $P(X^{\zeta,t})$.
\end{proof}

\begin{proposition} [Fixed points as reducible gauged maps] 
\label{red}
 Any $\C^\times$-fixed component of $\ol{\M}^G_n(C,X,\LL_-,\LL_+)$ is in 
 the image of $\ol{\M}^{G_\zeta}_n(C,X,\LL_t,\zeta)$ in 
 $\ol{\M}^G_n(C,X,\LL_-,\LL_+) $ for some $t \in (-1,1)$ where $\zeta 
 \in g$ is a non-zero element, $G_\zeta$ is stabilizer, and 
 $\C^\times_\zeta \subset G_\zeta$ the unparametrized one-parameter 
 subgroup generated by $\zeta$, consisting of maps $u: \hat{C} \to 
 X/G_\zeta$ taking values in $X^\zeta/G_\zeta$ on the principal 
 component, and $X/G_\zeta$ on the bubbles.   
\end{proposition} 

\begin{proof} Any fixed object of $\C^\times$ in
  $\ol{\M}_n^{G}(C,X,\LL_-,\LL_+)$ not in the fixed point components
  $\ol{\M}_n^{G}(C,X,\LL_\pm)$ is a datum
  $\pi: P \to C, u: \hat{C} \to P(X)$ of $\ol{\M}^{G,\lev}(C,X,\LL_t)$
  for some $t \in (-1,1)$ with a one-parameter group of automorphisms
  $\psi_\alpha: P \to P, \alpha \in \C^\times$ and
  $\phi_\alpha: \hat{C} \to \hat{C}$ trivial on the principal
  component $C_0$ and intertwining the section $u$ in the sense that
  $ \psi_\alpha(X) \circ u = u \circ \phi_\alpha$.  The infinitesimal
  automorphism corresponding to $\alpha_\zeta$ is a section of the
  adjoint bundle $P(\g)$, given by an element
  $\zeta \in P(\g)_z \cong \g$ at a base point $z \in C$. The
  structure group of $P$ reduces to the centralizer $G_\zeta$, and the
  section $u$ takes values in the fixed point set
  $P(X^\zeta) = P(X)^{\zeta}$ of $\zeta$ on the principal component.
\end{proof}

\begin{remark} \label{bubbles}
 The fixed point locus admits a description in terms of ``bubble
 trees'' as follows: There is an isomorphism
\[ \ol{\M}_n^{G_\zeta}(C,X,\LL_t,\zeta) \to \bigcup_{r,[I_1,\ldots,I_r]}
 \left( \M^{G_\zeta,\fr}_r(C,X^{\zeta,t}) \times_{(X^\zeta)^r}
 \prod_{j=1}^r \ol{\M}_{|I_j|+1}(X) \right)^{\C^\times_\zeta} / (G_\zeta)^r \]
where $I_1 \cup \ldots \cup I_r \subset \{ 1 , \ldots, n \}$ is a
disjoint union of subsets describing markings lying on bubble
components and $ \M^{G_\zeta,\fr}_r(C,X^{\zeta,t})$ denotes the moduli
stack of gauged maps with framings at the marked points.  Indeed, by
definition each object of $ \ol{\M}_n^{G_\zeta}(C,X,\LL_t,\zeta)$
consists of a principal component mapping to $X^{\zeta,t}/G_\zeta$ and
a collection of bubble trees in $X$ fixed (up to reparametrization) by
the action of $\C^\times_\zeta$.
\end{remark} 

\begin{corollary} \label{obstheory} {\rm (Obstruction theory for the fixed point components)} 
$ \ol{\M}^{G_\zeta}_n(C,X,\LL_t,\zeta)$ is an Artin stack, and if every
  automorphism group is finite modulo $\C^\times_\zeta$, each substack
  with fixed homology class $d \in H_2^{G_\zeta}(X,\Z)$ is a proper
  Deligne-Mumford stack with a $\C^\times$-equivariant relatively
  perfect obstruction theory over $\ol{\MM}_n(C)$.
\end{corollary} 

\begin{proof} The relatively perfect obstruction theory
  $ \ol{\M}^{G_\zeta}_n(C,X,\LL_t,\zeta)$ is pulled back from that on
  the $\C^\times$-fixed point set in
  $\ol{\M}_n^G(C,X,\LL_-,\LL_+)^{\C^\times}$ in Proposition \ref{red}.
  The latter is a special case of existence of relatively perfect
  obstruction theories on fixed point loci discussed in \cite{gr:loc}.
\end{proof} 

\begin{lemma} The conormal complex $\nu_t^\dual$ of the morphism
  \[\ol{\M}^{G_\zeta}_n(C,X,\LL_t,\zeta) \to
  \ol{\M}_n^G(C,X,\LL_-,\LL_+)\]
  is isomorphic to the $\C_\zeta^\times$-moving part of the
  obstruction theory in $\ol{\M}^{G_\zeta}_n(C,X,\LL_t)$, whose
  relative part is $(Rp_* e^* T_{X/G})^\dual$.
\end{lemma} 

\begin{proof} By definition the obstruction theory for
  $\ol{\M}_n^G(C,X,\LL_-,\LL_+)$ fits into an exact triangle with that
  of $\ol{\M}^G_n(C,X)$ and a trivial factor corresponding to the
  fiber of $\P(D(\LL_-) \oplus D(\LL_+))$.  Under projection the
  normal complex to the fixed point component
  $\ol{\M}^{G_\zeta}_n(C,\LL_t,X,\zeta)$ of the $\C^\times$-action is
  isomorphic to the moving part of the obstruction theory on
  $\ol{\M}^{G_\zeta}_n(C,\LL_t,X)$, under the identification of
  $\C^\times$ with $\C^\times_\zeta$, as in Lemma \ref{vmaster}.
\end{proof}

Virtual integration gives rise to the fixed point contributions in the
wall-crossing formula.  Let
$[\ol{\M}^{G_\zeta}_n(C, X,\LL_t,\zeta,d)]$ denote the virtual
fundamental class in the homology of the coarse moduli space resulting
from Corollary \ref{obstheory}.  Integration with respect to these
classes yields {\em $\zeta$-fixed gauged Gromov-Witten invariants} of
Definition \ref{contrib}.  The $\zeta$-fixed gauged Gromov-Witten
invariants that appear in the wall-crossing formula involve further
twists by Euler classes of the virtual normal complex: Recall that we
constructed in the previous section a perfect obstruction theory on
$\ol{\M}_n^{G_\zeta}(C,X,\LL_t,\zeta)$, as well as a normal complex
for the embedding in $\ol{\M}_n^G(C,X,\LL_-,\LL_+,\zeta)$.

\begin{definition} \label{qfp}  [Fixed point contributions to wall-crossing
for Gromov-Witten invariants]
\label{contrib} Virtual integration over
  the stacks $\ol{\M}_n^{G_\zeta}(C,X,\LL_t,\zeta,d), d \in H_2^G(X,\Z)$
  defines a ``fixed point contribution''
\begin{multline} 
  \tau_{X,\zeta,t}: QH_{G,\fin}(X) \to \ti{\Lambda}_X^G[\xi,\xi^{-1}],
  \\ \quad \alpha \mapsto \sum_{d \in H_{G}^2(X,\Z)} \sum_{n \ge 0}
  \int_{[\ol{\M}_n^{G_\zeta}(C,X,\LL_t,\zeta,d)]} \frac{q^d}{n!}
  \ev^* (\alpha, \ldots, \alpha) \cup \Eul(\nu_t)^{-1} \cup f^*
  \beta_n \end{multline}
for $\alpha \in H_G(X)$ and a sequence of classes
$\beta_n \in H(\ol{\M}_n(C))$, extended by (multi)linearity of the
integral over $\Lambda_X^G$, and where we omit the restriction map
$H_{G}(X) \to H_{G_\zeta}(X)$ to simplify notation.
\end{definition} 

This completes the construction of the fixed point potential in
Definition \ref{contrib}.  

\begin{remark} The fixed point potential $\tau_{X,\zeta,t}$ takes
  values in $\ti{\Lambda}_X^G$ rather than in $\Lambda_{X,\LL_t}^G$.
  Indeed, the number of possible pairings of classes of gauged maps
  with $c_1(\LL_t)$ in the case that a central subgroup
  $\C^\times_\zeta$ acts trivially can be arbitrarily small, since
  twisting by a character of $\C^\times_\zeta$ does not change the
  pairing.
\end{remark}

\begin{remark} The right-hand-side of the formula in Theorem
  \ref{qkalkcirc} can also be re-written using the quantum Kirwan map
  for
\[ QH_{G_\zeta/\C^\times_\zeta}(X^{\zeta,t}) \to QH_{\C^\times}(X^{\zeta,t} \qu
  (G_\zeta/\C^\times_\zeta)) \] 
  using the adiabatic limit theorem for
  $(G_\zeta/\C^\times_\zeta)$-gauged maps.  However, in our examples
  the gauged Gromov-Witten invariants are always easier to compute
  than the Gromov-Witten invariants of the git quotients, so we have
  left the formula as written.
\end{remark}

\subsection{The wall-crossing formula}

By the adiabatic limit theorem \ref{largerel}, to prove the
wall-crossing formula \ref{qkalkcirc} it suffices to prove a formula
for the difference of gauged potentials.  The following result is an
algebro-geometric generalization of a wall-crossing formula of
Cieliebak-Salamon \cite{ciel:wall} for gauged Gromov-Witten invariants
of quotients of vector spaces defined using symplectic geometry.  We
will deduce our main result Theorem \ref{kalk3} by taking the large
area limit $\rho \to \infty$ of the following Theorem:

\begin{theorem}[Wall-crossing for gauged Gromov-Witten potentials] \label{gwall}
Let $X$ be a smooth projective $G$-variety.  Suppose that $\LL_\pm \to
X$ are polarizations such that semistable=stable for the stack
$\ol{\M}_n^G(C,X,\LL_-,\LL_+)$ of \ref{masterprop}.  The gauged
Gromov-Witten potentials are related by
\begin{equation}
 \tau^G_{X,+} - \tau^G_{X,-} = \sum_{[\zeta],t \in (-1,1)}
\Resid_\xi
 \tau_{X,\zeta,t} \end{equation}
where the sum is over equivalence classes $[\zeta]$ as in \eqref{equiv}.
\end{theorem}

\begin{proof} The statement follows from virtual localization applied
  to $\ol{\M}_n^G(C,X,\LL_-,\LL_+)$ and the identification of fixed
  point contributions in Proposition \ref{red}. \end{proof}

Combining Theorem \ref{gwall} with the adiabatic limit Theorem
\ref{largerel} implies:

\begin{theorem} [Quantum Kalkman formula, arbitrary group case]
\label{kalk3}
Suppose that $X$ is equipped with polarizations $\LL_\pm$ so
stable=semistable for the action of $G$ on $\P(\LL_- \oplus \LL_+)$. Then
the Gromov-Witten invariants of $X \qu_\pm G$ are related by a sum of
twisted gauged Gromov-Witten invariants for subgroups $G_\zeta \subset
G$
\begin{equation}
 \tau_{X \qu_+ G} \kappa_{X,+}^G - \tau_{X \qu_- G} \kappa_{X,-}^G =
 \lim_{\rho \to \infty} \sum_{[\zeta],t \in (-1,1)}
 \Resid_\xi \tau_{X,\zeta,t} ,\end{equation}
where the sum is over $[\zeta]$ in \eqref{equiv}. 
\end{theorem} 

We already gave a simple Example \ref{threepoint} of the formula in
Theorem \ref{kalk3} in the introduction.  We give another Fano example:

\begin{example} (Quantum powers of the first Chern class for the
  blow-up of the projective plane) Suppose that, as in Example
  \ref{classexamples} \eqref{twotorus}, $G = (\C^\times)^2$ acts on
  $X = \C^4$ with weights $(1,0),(1,0),(1,1),(0,1) \in \Z^2$.
  Consider the path from $(-1,2)$ to $(2,-1)$ in $H^2_G(X) \cong \Q^2$
  crossing through the chambers with git quotients
  $\emptyset, \P^2, \Bl(\P^2), \emptyset$ as in Example
  \ref{classexamples} \eqref{twotorus}.  Denote by $X \qu_- G$
  resp. $X \qu_+ G$ the second resp. third quotient.  The quantum
  Kirwan morphism for $\P^2, \Bl(\P^2)$ has no quantum corrections,
  since these varieties are Fano.  Hence
\begin{equation} \label{D0}  D_0 \kappa_{X,\pm}^G (c_1^G(X) ) = c_1(X \qu_\pm G) .
\end{equation}
We consider the wall-crossing formula for invariants with $5$ fixed
markings, corresponding to the small quantum product
\[ c_1(X \qu  G)^{\star 5} \in QH(X \qu  G) .\]   
In the notation of \eqref{gpot} we wish to compute
\[ \partial_{(c_1(X \qu  G),0)}^5 \tau_{X \qu G} (1, [\pt]) \in
\Lambda_X^G \]
where
\[ c_1(X \qu G) \in H( X \qu G), \quad [\pt] \in H(\ol{\M}_{0,5}) .\]
By Example \ref{toric}, the moduli space of gauged maps is, after
fixing the locations of the markings, the quotient of the space of
sections $H^0(\P,\mO_C(d)^\times \times_G X)$ by $G$, with stability
condition corresponding to the stability condition for $X$, see
\cite{qkirwan}.  We take $d = (1,0)$, so that
\[H^0(\P,\mO_C(d)^\times \times_G X) = \C_{(0,1)} \oplus \C_{(1,1)}^{\oplus
  2} \oplus \C_{(1,0)}^{\oplus 4} \]
(or $\C^7$ for short) where for any weight $\mu$, $\C_\mu$ denotes the
one-dimensional representation with weight $\mu$.  We consider the
sequence of polarizations $\LL_t$ corresponding to the vectors
\[ (-1,2), (1,2), (2,1), (2,-1) \in \g^\dual \]
lying in the path of chambers from left to right in Figure
\ref{twochamber}.  The moduli spaces of gauged maps corresponding to
the various chambers are therefore the empty set, $\P^5$, its blow-up
along a projective line $\Bl_{\P^1}(\P^5)$, and the empty set again.
The stabilizers for the wall-crossing terms are the perpendicular
vectors to the walls in Figure \ref{twochamber}.  The first
wall-crossing term for degree $(1,0)$ invariants corresponds to the
direction $\zeta = (1,0)$, for which there is a unique $\zeta$-fixed
stable map with normal weights $(1,0)$ with multiplicity 4 and $(1,1)$
with multiplicity $2$.  The wall-crossing term is
\[ \Resid_{\xi_1} \frac{ (3 \xi_1 + 2\xi_2)^5 }{ \xi_1^4 (\xi_1 +
  \xi_2)^2} |_{\xi_2 =0} = 243 .\]
Using \eqref{D0} we obtain that $c_1(\P^2)^{\star 5} \supset 243
q_1 [\on{pt}] $ which is shorthand for saying that the coefficient of
$q_1 [\on{pt}]$ is $243$; here $\star$ is the small quantum product,
as expected since
\[ c_1(\P^2)^{\star 5} = (3 \omega)^{\star 5} = 243 \omega^{\star 3}
\star \omega^{\star 2} = 243 q_1 [\on{pt}] .\]
The second wall-crossing term corresponds to the change of quotient
from $\P^2$ to $\Bl(\P^2)$ with $\zeta = (1,-1)$ is (after fixing the
five points on $\P$) an integration over $\ol{\M}_0^G(\P,X,\zeta,t)
\cong \P$, the quotient of the $\zeta$-fixed summand
$\C_{(1,1)}^{\otimes 2}$ by $\C^\times$,
\[ \hh \Resid_\xi \int_{[\P]} \kappa_{\C^7}^{\C^\times} \frac{ (3 \xi_1
  + 2\xi_2)^5 }{ \xi_1^4 \xi_2} \]
where $\kappa_{\C^7}^{\C^\times}$ is the descent map in equivariant
cohomology from $\C^7$.  Using that $\kappa_{\C^6,\C^\times}(\xi_1)$
is the generator $\omega$ of $H^2(\P)$ this gives
\begin{eqnarray*}
 \hh \Resid_\xi \int_{[\P]} \frac{ \left(5 \omega + \xi\right)^5}{ \left(\omega +
  \xi\right)^4 \left(\omega - \xi\right)} &=& \hh \Resid_\xi \int_{[\P]} - \xi^{-5} \left(25 \omega
\xi^4 + \xi^5\right)\left(1 - 4 \frac{\omega}{\xi}\right) \left(1 + \frac{\omega}{\xi}\right) \\ &=& - 11 .\end{eqnarray*}
It follows that the coefficient of $[\on{pt}]$ in
$c_1(\Bl(\P^2))^{\star 5}$ is $232 q_1 $ .

For the transition to the empty chamber, as before
$\kappa_{\C^4}^{\C^\times}$ maps to $\xi_1$ to the generator $\omega$
of $H^2(\P^3)$ while $\xi_2$ maps to the parameter $\xi$ for the
residual $\C^\times$-action. The wall-crossing term is
\begin{eqnarray*} 
 \Resid_\xi 
\int_{[\P^3]} \kappa_{\C^4}^{\C^\times} \frac{ \left(3\xi_1 +
  2\xi_2\right)^5}{ \xi_2\left(\xi_1 + \xi_2\right)^2} &=& 
 \Resid_\xi  
\int_{[\P^3]} \frac{ \left(3\omega +
  2\xi\right)^5}{ \xi\left(\omega + \xi\right)^2} \\
&=& 
\Resid_\xi  
\int_{[\P^3]}  
\left(3 \omega\right)^3 \left(2\xi\right)^2 10 \xi^{-3}  + 
\left(3 \omega\right)^2 \left(2\xi\right)^3 10 \xi^{-3} \left(-2 \frac{\omega}{\xi}\right)^2 
\\ && + \left(3 \omega\right)\left(2 \xi\right)^4 5 \xi^{-3}\left(3 \frac{\omega^2}{\xi^2}\right) 
+ \left(2\xi\right)^5 \xi^{-3} \left(- 4 \frac{\omega^3}{\xi^3}\right) \\
&=& 1080 - 1440 + 720 - 128
 \\
&=& 232
 .\end{eqnarray*}

This is as expected since the quotient $X \qu_t G$ is empty in the
last chamber.  One can verify that the expansion of
$c_1(\Bl(\P^2))^{\star 5}$ contains $232 q_1 [\on{pt}] $ using the
known quantum multiplication table for $\Bl(\P^2)$ from
Crauder-Miranda \cite{cm:dp}:
\[ \begin{array}{l|lll}
\star &   e              &   f             & p \\ \hline
e     & -p + eq^e + xq^f &  p - eq^e       & fq^f \\ 
f     &                 &    eq^e          &  xq^{e + f}  \\
p     &                  &                  &  (e + f)q^{e + f}
\end{array} \]
%
%
where $e,f,p,x \in H(\Bl(\P^2))$ are the exceptional resp. fiber
resp. point resp. fundamental classes respectively, $q_1 = q^{e + f}$,
and using a little help from Mathematica.  (We thank Eric Malm for
teaching us how to get Mathematica to compute these coefficients.)
\end{example}

\section{Invariance under crepant transformations} 
\label{cy}

Ruan and others, see \cite{cr:crep}, conjectured that {\em crepant
  resolutions} induce equivalences in Gromov-Witten theory.  
We prove a version Theorem \ref{cytype} of Ruan's conjecture for  
crepant birational equivalences induced by variation of git. 
 
\subsection{Crepant transformations} 

We consider the birational transformations that are crepant in the
following sense: 

\begin{definition} Suppose that $Y_\pm$ are smooth proper
  Deligne-Mumford stacks with projective coarse moduli spaces related
  by a birational equivalence given by open embeddings
\[ \begin{tikzcd}
Y_-  & Z  \arrow[l,"\phi_-"]  \arrow[r,swap,"\phi_+"] & Y_+ \end{tikzcd} .\]  
Such a birational equivalence is called {\em crepant} (or a {\em
  $K$-equivalence}) if $\phi$ extends to morphisms
$\ti{\phi}_\pm: \ti{Z} \to Y_\pm$ from a smooth stack $\ti{Z}$ with
projective coarse moduli space such that the pullbacks of the
canonical divisors to $\ti{Z}$ are equal, as in Kawamata
\cite{kaw:dk}.  This ends the definition.
\end{definition} 

A well-known conjecture of Li-Ruan \cite{liruan:surg}, Bryan-Graber
\cite{bryan:crep} and others (perhaps motivated by physics papers such
as Witten \cite{wi:ph}) that in such a situation (not necessarily
arising from geometric invariant theory) the Gromov-Witten theories of
$Y_-$ and $Y_+$ are equivalent, in a sense to be made precise.  Many
special cases have been proved, see for example Iwao-Lee-Lin-Wang
\cite{lee:flop}, Lee-Lin-Wang \cite{lee:fmi},
Boissi{\'e}re-Mann-Perroni \cite{bmp:com}, Bryan-Gholampour
\cite{bryan:ade}, \cite{bryan:hh}, Bryan-Graber-Pandharipande
\cite{bryan:c2z3}, Coates-Corti-Iritani-Tseng \cite{coates:computing}
and Coates-Iritani-Tseng \cite{coates:wall}.

We specialize to the case that the birational transformation is
obtained by variation of git quotient.  Suppose that $X$ is a smooth
projective $G$-variety, and $X \qu_\pm G$ are git quotients obtained
from polarizations $\LL_\pm \to X$.  Since the semistable loci are
open, the identity on the locus semistable for both polarizations
induces a birational transformation from $X \qu_- G$ to $X \qu_- G$.
We call such a birational transformation {\em of git type}.  Suppose
that stable=semistable for $\P(\LL_- \oplus \LL_+)$ so that the master
space $\ti{X} = \P(\LL_- \oplus \LL_+) \qu G$ is a smooth proper
Deligne-Mumford stack.

\begin{definition} \label{crepant} A birational transformation of git
  type $\phi = (\phi_-,\phi_+)$ will be called {\em crepant} if the
  sum of the weights $\mu_i(F) \in \Z$ of $\C^\times_\zeta$ on the
  normal bundle to any fixed point component $F \subset X^{\zeta,t}$,
  counted with multiplicity, vanishes:
  \[\sum_{i=1}^{\codim(F)} \mu_i(F) = 0, \quad \forall F \subset
  X^{\zeta,t} .\]
\end{definition}

\begin{definition} 
  The definition of crepant transformation of git type is a special
  case of the definition of crepant transformation ($K$-equivalence)
  in Kawamata \cite{kaw:dk} etc.  Indeed, Kempf's descent lemma
  \cite[Theorem 2.3]{dr:pi} and the crepant condition together imply
  that the canonical bundle descends to each singular quotient, from
  which the canonical bundles on $X \qu_\pm G$ are pulled back. The
  fiber product of these morphism is the required smooth stack in the
  definition of crepant transformation.
\end{definition}

\subsection{The Picard action} 

The proof of invariance in Theorem \ref{cytype} uses a symmetry of the
fixed point contributions under an action of the Picard stack
\[\Pic(C) := \Hom(C, B\C^\times)\]
of line bundles on $C$; a similar action was used in a proof of a
generalized Verlinde formula in \cite{tw}.  The Lie algebra $\g_\zeta$
has a distinguished factor $\C \zeta$ generated by $\zeta$, and using
an invariant metric the weight lattice of $\g_\zeta$ has a
distinguished factor $\Z$ given by its intersection with the Lie
algebra of $\C^\times_\zeta$.  After passing to a finite cover, there
exists a splitting
$G_\zeta \cong G_\zeta/\C^\times_\zeta \times \C^\times_\zeta$.

We define an action of the Picard group on the moduli stack as
follows.  Recall that an object of
$\ol{\M}_n^{G_\zeta}(C,X,\LL_t,\zeta)$ consists of a tuple
$(P,\hat{C},u)$ where $P \to C$ is a $G$-bundle and
$u: \hat{C} \to P(X)$ is $\zeta$-fixed, in particular, the restriction
of $u$ to the principal component of $C$ maps into the fixed point
locus $X^\zeta$.  Let $\C^\times_\zeta \subset G_\zeta$ denote the
subgroup of $G_\zeta$ generated by $\zeta \in \g_\zeta$. 

\begin{definition} {\rm (Picard action)} For $Q \to C$ a line bundle
  and $(P,\hat{C},u)$ an object of
  $\ol{\M}_n^{G_\zeta}(C,X,\LL_t,\zeta)$ define
\begin{equation} \label{picact}
 Q (P,\hat{C},u) := (P \times_{\C^\times_\zeta} Q, \hat{C}, v
 ) \end{equation} 
where $v$ is defined as follows: We have an isomorphism of associated bundles
\[ (P \times_{\C^\times_\zeta} Q) (X^\zeta) \cong P(X^\zeta) \] 
since the
action of $\C^\times_\zeta$ on $X^\zeta$ is trivial. Hence the
principal component of $u$, which is a section of $P(X^\zeta)$ induces
a corresponding section of
$ ( P \times_{\C^\times_\zeta} Q) (X^\zeta)$.  Each bubble component
of $u$ maps into a fiber of $P(X)$, canonically identified with $X$ up
to the action of $G_\zeta$, and so induces a corresponding map into a
fiber of $( P \times_{\C^\times_\zeta} Q) (X)$, well-defined up to
isomorphism. 
\end{definition} 

The action of the Picard group preserves semistable loci in the large
area limit.  Indeed, because the Mundet weights
$\mu_M(\sigma,\lambda)/\rho$ approach $\mu_H(\sigma,\lambda)$ as
$\rho \to \infty$, the limiting Mundet weight is unchanged by the
shift by $Q$ in the limit $\rho \to \infty$ and so Mundet
semistability is preserved, see Remark \ref{fixedlargearea}.  It
follows that for $\rho$ sufficiently large the action of an object $Q$
of $\Pic(C)$ induces an isomorphism
\begin{equation} \label{translate} \cS^\delta:
  \ol{\M}_n^{G_\zeta}(C,X,\LL_t,\zeta,d)
 \to \ol{\M}_n^{G_\zeta}(C,X,\LL_t,\zeta,d + \delta) \end{equation}
where $\delta = c_1(Q)$.  The action lifts in an obvious way to the
universal curves
$\ol{\cC}_n^{G_\zeta}(C,X,\LL_t,\zeta,d) \to
\ol{\cC}_n^{G_\zeta}(C,X,\LL_t,\zeta,d + \delta) $,
denoted with the same notation.

\begin{lemma} \label{pic} The action of $Pic(C)$ in \eqref{translate}
  induces isomorphisms of the relative obstruction theories, and so
  the Behrend-Fantechi virtual fundamental classes. Furthermore, the
  action preserves the class $\ev^* \alpha$ for any
  $\alpha \in H_{G_\zeta}(X)^n$. \end{lemma}

\begin{proof} The action of $\Pic(C)$ lifts to the universal curves,
  denoted by the same notation. Since the relative part of the
  obstruction theory on $\ol{\M}_n^{G_\zeta}(C,X,\LL_t,\zeta,d)$ is
  the $\C^\times_\zeta$-invariant part of $(Rp_* e^* T_{X/G})^\dual$
  up to the factor $\C\zeta$, the isomorphism $S^\delta$ preserves the
  relative obstruction theories on
  $ \ol{\M}_n^{G_\zeta}(C,X,\LL_t,\zeta,d) $ and
  $ \ol{\M}_n^{G_\zeta}(C,X,\LL_t,\zeta,d + \delta) $ and so the
  Behrend-Fantechi virtual fundamental classes
  $[\ol{\M}_n^{G_\zeta}(C,X,\LL_t,\zeta,d)]$ and
  $[ \ol{\M}_n^{G_\zeta}(C,X,\LL_t,\zeta,d + \delta)]$. (Note that on
  the principal component, the obstruction theory is
  $(Rp_* e^* TX^\zeta/G_\zeta)^\dual$ which is unchanged by the tensor
  product by $\C^\times_\zeta$-bundles. On the bubble components
  $(Rp_* e^* TX/G)^\dual$ is unchanged by the tensor product since the
  pull-back of $Q$ to $\hat{C}$ is trivial.) Since the evaluation map
  is unchanged by pull-back by $\cS^\delta$ (up to isomorphism given
  by twisting by $Q$), the class $\ev^* \alpha$ is preserved.
\end{proof}

\begin{remark} \label{distrib} To interpret the main result we recall
  the basic definitions from the Schwartz theory of {\em
    distributions} for which the standard reference is H\"ormander
  \cite{ho:an}.  We only need the case of distributions on
  the unit circle $S$.  Denote by $\D'(S)$ the space of continuous
  linear functionals on the smooth functions on $S$, and by
  $\mE'(S) \subset \D'(S)$ the space of {\em tempered distributions}.
  Fourier transform defines an isomorphism of $\mE'(S)$ with the space
  of functions on $\Z$ with polynomial growth.  We view $q$ as a
  coordinate on the punctured plane $\C^\times$.  Any formal power
  series in $q,q^{-1}$ defines a distribution on $S$, which is
  tempered if the coefficient of $q^d$ has polynomial growth in $d$.
  In particular $\sum_{d \in \Z} q^d$ is the delta function at
  $q = 1$, and has Fourier transform the constant function with value
  $1$.  Any distribution of the form $\sum_{d \in \Z} f(d) q^d$, for
  $f(d)$ polynomial, is a sum of derivatives of the delta function
  (since Fourier transform takes multiplication to differentiation)
  and so is almost everywhere zero.
\end{remark} 

\subsection{Proof of invariance}

In this section we prove Theorem \ref{cytype}.  We study the
dependence of the fixed point contributions 
with respect to the Picard action defined in \eqref{picact}. 
 Suppose
that $Q$ is a $\C^\times_\zeta$-bundle of first Chern class the
generator of $H^2(C)$, after the identification
$\C^\times_\zeta \to \C^\times$.  Denote the corresponding class in
$H_2^{G_\zeta}(X^\zeta)$ by $\delta$.  Consider the action of the
$\Z$-subgroup of $\Pic(C)$ generated by $Q$.  The contribution of any
component $\ol{\M}_n^{G_\zeta}(C,X,\LL_t,\zeta,d)$ of class
$d \in H_2^G(X)$ differs from that from the component induced by
acting by $Q^{\otimes r}$, of class $d + r \delta$, by the ratio of
Euler classes of the virtual normal complex $ (Rp_* e^* T(X/G))^+$
  \begin{equation} \label{eulerdiff} \frac{\Eul_{\C^\times_\zeta}( (Rp_* e^*
    T(X/G) )^+)}{ \Eul_{\C^\times_\zeta}( \cS^{r\delta,*}
    (Rp_* e^* T(X/G))^+)}\in H(
    \ol{\M}_n^{G_\zeta}(C,X,\LL_t,\zeta,d))
\end{equation}
which we now compute.
Let $X^{\zeta,t}$ be the component of the fixed point set $X^\zeta$
which is semistable for $t \in (-1,1)$.  For simplicity, we assume that
$X^{\zeta,t}$ is connected; in general, one should repeat the
following argument for each connected component.    Let 
\[ [\nu_{X^{\zeta,t}/G_\zeta} ] = [TX/TX^{\zeta,t}] - [\g/\g_\zeta] \]
denote the class of the virtual normal complex for
$X^{\zeta,t} /G_\zeta \to X/G$. Consider the decomposition into
$\C^\times$-bundles
\[ \nu_{X^{\zeta,t}/G_\zeta}  = \bigoplus_{i=1}^{m_t}
\nu_{X^{\zeta,t},i} \]
where $\C^\times$ acts on $\nu_{X^{\zeta,t},i}$ with non-zero weight
$\mu_i \in \Z$ and $m_t$ is the codimension of $X^{\zeta,t}$, which
for simplicity we assume is constant.  Then $ e^* T(X/G)$ is
canonically isomorphic to $\cS^{r\delta} e^* T(X/G)$ on the bubble
components, since the $G$-bundles are trivial on those components.
Because the pull-back complexes are isomorphic on the bubble
components, the difference
\[ (e^* T(X/G))^+ - \cS^{r\delta,*} (e^* T(X/G))^+ 
\in K( \ol{\M}_n^{G_\zeta}(C,X,\LL_t,\zeta,d)) \] 
is the pullback of the difference of the restrictions to the principal
part of the universal curve, that is, the projection on the second
factor
\[p_0: C \times \ol{\M}_n^{G_\zeta}(C,X,\LL_t,\zeta) \to
\ol{\M}_n^{G_\zeta}(C,X,\LL_t,\zeta) .\]
These restrictions are given  by 
\begin{eqnarray}
 (e^* T(X/G))^{+,\on{prin}} &\cong& \bigoplus_{i=1}^{m_t} e^*
\nu_{X^{\zeta,t},i} - \g_\zeta, \\  \label{restr2}
\cS^{r\delta,*} (e^* T(X/G))^{+,\on{prin}}
&\cong & 
\bigoplus_{i=1}^{m_t}
e^* \nu_{X^{\zeta,t},i} \otimes (e^* Q \times_{\C^\times_\zeta}
\C_{r\mu_i}) - \g_\zeta
\end{eqnarray}
where $e$ is the map from the universal curve to $C$.  The projection
$p_0$ is a representable morphism of stacks given as global quotients.
To compute the difference in push-forwards we apply
Grothendieck-Riemann-Roch for such stacks \cite{toen:rr},
\cite{ed:rr}.  The Todd class on the curve is 
\[ \Td_\cC = 1 + (1-g) \omega_C \]  
so
\begin{equation} \label{td}
 \Td_{\cC \times \M} = (1-g)\omega_C + \pi^* \Td_{\M} . \end{equation}  
Let
\[z: \ol{\M}_n^{G_\zeta}(C,X,\LL_t,\zeta) \to C \times 
\ol{\M}_n^{G_\zeta}(C,X,\LL_t,\zeta)\] 
be a constant section of $p_0$.  Then
\begin{eqnarray*}
\Td_\M  \Ch(\cS^{r\delta,*} Rp_* (e^* T(X/G))^+) &=& 
p_{0,*} (\Td_{C \times \M} 
\Ch(\cS^{r\delta,*}  e^* T(X/G))^+ )\\
&=& (1-g) z^* \Ch(\cS^{r\delta,*} (e^*  T(X/G))^+)+\\ &&  \Td_\M 
p_{0,*} \Ch(\cS^{r\delta,*} e^*  T(X/G)^+) \end{eqnarray*} 
by Grothendieck-Riemann-Roch and \eqref{td}.  Continuing we have 
\begin{eqnarray*} \ldots &=& (1-g) z^* \Ch(e^* (T(X/G))^+)+ \\ && \Td_\M
  p_{0,*} \sum_{i=1}^{m_t} \Ch(e^* \nu_{X^{\zeta,t},i}) \Ch( (e^* Q
  \times_{\C^\times_\zeta} \C_{r\mu_i}))
  \\
                         &=& (1-g) z^* \Ch(\cS^{r\delta,*} e^*
                             (T(X/G))^+)+ \\ && \Td_\M p_{0,*}
                                              \sum_{i=1}^{m_t}
                                              \Ch(e^*
                                              \nu_{X^{\zeta,t},i}) (1
                                              + r \mu_i
                                              \omega_C) \end{eqnarray*} 
since the bundle $Q$ is trivial on any fiber of $C \times \M \to \M$
and using \eqref{restr2}.  Continuing using multiplicativity of the
Chern character and Grothendieck-Riemann-Roch again this equals
\begin{eqnarray*} 
  \ldots                       &=& p_{0,*}( \Td_{C \times \M} \Ch\left( Rp_*
                             e^* (T(X/G))^+ \oplus \bigoplus_{i=1}^{m_t}
                             ( e^* \nu_{X^{\zeta,t},i} )^{\oplus
                             r \mu_i}\right) \\
                         &=& \Td_\M \Ch\left( Rp_* e^* (T(X/G)))^+
                             \oplus \bigoplus_{i=1}^{m_t} (z^* e^*
                             \nu_{X^{\zeta,t},i})^{\oplus r
                             \mu_i}\right).
\end{eqnarray*}
 Hence
\begin{equation} \label{diff}
 \Ch(\cS^{r\delta,*} Rp_* e^* (T(X/G))^+) = 
\Ch\left( Rp_* e^* (T(X/G))^+ \oplus 
 \bigoplus_{i=1}^{m_t} (z^* e^* \nu_{X^{\zeta,t},i})^{\oplus 
   r \mu_i}\right)
 \end{equation}
 The equality of Chern characters above implies by injectivity of the
 Todd map \cite{ed:rr} an equality 
 \[ [\cS^{r\delta,*} \Ind( T(X/G))^+] = [ \Ind(T(X/G))^+ \oplus
 \bigoplus_{i=1}^{m_t} (z^* e^* \nu_{X^{\zeta,t},i})^{\oplus r
   \mu_i}] . \]
 By the splitting principle we may assume that the
 $\nu_{X^{\zeta,t},i}$ are line bundles.  The difference in Euler
 classes \eqref{eulerdiff} is therefore given by the Euler class of
 the last summand in \eqref{diff}
\begin{eqnarray*}
\frac{\Eul_{\C^\times_\zeta}( Rp_* e^* T(X/G)^+) }{\Eul_{\C^\times_\zeta}(
\cS^{r\delta,*} Rp_* e^* T(X/G)^+)} &=&
\Eul_{\C^\times_\zeta} \left( \bigoplus_{i=1}^{m_t} (z^* e^*
\nu_{X^{\zeta,t},i})^{\oplus \mu_ir} \right) \\ &=& \prod_{i=1}^{m_t} (
\mu_i \xi + c_1(\nu_{X^{\zeta,t},i}) )^{\mu_ir } \\ &=&
\prod_{i=1}^{m_t} \left( \xi + \frac{c_1(\nu_{X^{\zeta,t},i})}{\mu_i} \right)^{\mu_ir }
\prod_{i=1}^{m_t} \mu_i^{\mu_i r} .
\end{eqnarray*}
Let 
\[ \mu = \sum_{i=1}^{m_t} \mu_i \in \Z \]
be the sum of weights of the action of $\C^\times_\zeta$ at the fixed
point component $X^{\zeta,t}$.  Expanding out the product we obtain
\begin{multline} 
 \prod_{i=1}^{m_t} \left( \xi + \frac{c_1(\nu_{X^{\zeta,t},i})}{\mu_i}
   \right)^{\mu_ir } 
= \xi^{r \mu } + \xi^{r \mu - 1} \left( r \sum_{i=1}^k
c_1(\nu_{X^{\zeta,t},i}) \right) \\ + \xi^{r \mu - 2} \left( r^2 \sum_{i
  \neq j} c_1(\nu_{X^{\zeta,t},i}) c_1(\nu_{X^{\zeta,t},j}) + \sum_{i=1}^k r
 \left(r - \frac{1}{\mu_i} \right) c_1(\nu_{X^{\zeta,t},i})^2 \right) + \ldots \end{multline}
and $\ldots$ indicates further terms with the property that the
coefficient of $\xi^{r \mu- m}$ is polynomial in $r$.  By the crepant
assumption in Definition \ref{crepant}, the sum of the weights is
$\mu = \sum_{i=1}^{m_t} \mu_i = 0.$
Write 
\[ 
\tau_{X,\zeta,t}  = \sum_d 
\tau_{X,\zeta,d,t} \] 
where $\tau_{X,\zeta,d,t}$ is the contribution from gauged maps of
class $d$.  For any singular value $t \in (-1,1)$,
\begin{multline}
 \sum_{r \in \Z} q^{d + \delta r} \tau_{X,\zeta, d + \delta
   r,t}(\alpha) = \sum_{r \in \Z} \prod_{i=1}^k \mu_i^{\mu_i r}q^{d +
   r \delta} \int_{[\ol{\M}_n^{G_\zeta}(C,X,\LL_t,\zeta,d)]} \\  \left( 1 +
 \xi^{-1} \sum_{i=1}^k r c_1(\nu_{X^{\zeta,t},i}) + \ldots \right)  \ev^*
 (\alpha, \ldots, \alpha)  \cup \Eul(\nu_t)^{-1} \cup f^* \beta_n
 \end{multline}
 where as before the terms $\ldots$ are polynomial in the $r$.  In the
 language of distributions, for any polynomial $f(r)$ in $r$,
\begin{equation} \label{form} 
\sum_{r \in \Z} f(r) \left( \prod_{i=1}^k \mu_i^{\mu_i} q^\delta \right)^r
=_{a.e.} 0\end{equation} 
vanishes almost everywhere in $q$, being a function times a sum of
derivatives of delta functions in $q$, see Remark \ref{distrib}.)
Since
$\kappa_X^{G,+} \tau_{X \qu_+ G} - \kappa_X^{G,-} \tau_{X \qu_- G}$ is
a sum of wall-crossing terms of the form \eqref{form}, this completes the proof of Theorem
\ref{cytype}. 

\begin{remark} 
  The standard formulation of the crepant transformation conjecture in
  Coates-Ruan \cite{cr:crep} etc. uses analytic continuation.  The
  above results say nothing about convergence of the gauged
  Gromov-Witten potentials, so it is rather difficult to put the
  version above in this language.  However, if the potentials
  $\kappa_X^{G,+} \tau_{X \qu_+ G} $ and
  $\kappa_X^{G,-} \tau_{X \qu_- G}$ have expressions as analytic
  functions with overlapping regions of definition on the torus
  $H^2_G(X,\Q)/H^2_G(X,\Z)$ with coordinate $q$, then they are equal
  on that region.
\end{remark}

\def\cprime{$'$} \def\cprime{$'$} \def\cprime{$'$} \def\cprime{$'$}
\def\cprime{$'$} \def\cprime{$'$}
\def\polhk#1{\setbox0=\hbox{#1}{\ooalign{\hidewidth
      \lower1.5ex\hbox{`}\hidewidth\crcr\unhbox0}}} \def\cprime{$'$}
\def\cprime{$'$}

\end{document}